\documentclass[12pt,amstex]{amsart}
\usepackage{etex}
\usepackage{txfonts}

\usepackage{epsfig}
\usepackage{amsmath}
\usepackage{amssymb}
\usepackage{amscd}
\usepackage{graphicx}
\usepackage[all]{xy}
\usepackage{pstricks}
\usepackage{pst-node}
\usepackage[english]{babel}
\usepackage{float}
\usepackage{tikz}
\usepackage{multicol}
\usepackage{enumitem}
\usetikzlibrary{matrix}
\newtheorem*{thm*}{Theorem}
\newtheorem{thm}{Theorem}[section]
\newtheorem{lem}[thm]{Lemma}

\newtheorem{prop}[thm]{Proposition}

\newtheorem{cor}[thm]{Corollary}

\newtheorem{conv}[thm]{Convention}
\newtheorem{notation}[thm]{Notation}

\theoremstyle{definition}
\newtheorem{defn}[thm]{Definition}

\theoremstyle{remark}
\newtheorem{rem}[thm]{Remark}

\numberwithin{equation}{section}

\def \Q {{\bf Q}}

\def \wh {\widehat}
\def \to {\rightarrow}

\def \N {\mathbb N}

\def \Z {\mathbb Z}
\def \R {\mathbb R}
\def \E {\mathbb E}

\def \e {\epsilon}
\def \d {\delta}
\def \l {\lambda}

\def \I {\mathcal{I}}
\def \id {{\rm id}}

\newcommand*\normm[1]{\left|\mspace{-1mu}\left|\mspace{-1mu}\left|#1\right|\mspace{-1mu}\right|\mspace{-1mu}\right| }

\newcommand*\Bignormm[1]{\Bigl|\mspace{-1mu}\Bigl|\mspace{-1mu}\Bigl|#1\Bigr|\mspace{-1mu}\Bigr|\mspace{-1mu}\Bigr| }

\begin{document}
\title{Pointwise convergence of some multiple ergodic averages}
\author{Sebasti\'an Donoso}
\address{Einstein Institute of Mathematics, The Hebrew University of Jerusalem, Givat Ram,
Jerusalem 91904, Israel} \email{sdonosof@gmail.com}

\author{Wenbo Sun}
\address{Department of Mathematics, The Ohio State University, 231 West 18th Avenue, Columbus OH, 43210-1174, USA}
 \email{sun.1991@osu.edu}


\thanks{The first author is supported by ERC grant 306494. The second author is partially supported by NSF grant 1200971.}

\begin{abstract}
We show that for every ergodic system $(X,\mu,T_1,\ldots,T_d)$ with commuting transformations, the average
\[\frac{1}{N^{d+1}} \sum_{0\leq n_1,\ldots,n_d \leq N-1} \sum_{0\leq n\leq N-1} f_1(T_1^n \prod_{j=1}^d T_j^{n_j}x)f_2(T_2^n \prod_{j=1}^d T_j^{n_j}x)\cdots f_d(T_d^n \prod_{j=1}^d T_j^{n_j}x).  \]
converges for $\mu$-a.e. $x\in X$ as $N\to\infty$. If $X$ is distal, we prove that the average
\[\frac{1}{N}\sum_{i=0}^{N} f_1(T_1^nx)f_2(T_2^nx)\cdots f_d(T_d^nx) \]
converges for $\mu$-a.e. $x\in X$ as $N\to\infty$.

We also establish the pointwise convergence of averages along cubical configurations arising from a system commuting transformations. 

Our methods combine the existence of sated and magic extensions introduced by Austin and Host respectively with ideas on topological models by Huang, Shao and Ye.  
\end{abstract}

\maketitle

\section{Introduction}
\subsection{Pointwise convergence for multiple averages}
Let $(X,\mathcal{X},\mu)$ be a probability space and $T_1,\ldots,T_d$ be measure preserving transformations on $X$.
The study convergence of multiple averages of the form \begin{equation} \label{MultipleAverage}
\frac{1}{N}\sum_{i=0}^{N} f_1(T_1^nx)f_2(T_2^nx)\cdots f_d(T_d^nx)
 \end{equation}   has a rich history, starting from Furstenberg's proof of Szemer\'edi's Theorem via an ergodic theoretical analysis and its combinatorial consequences.
Different approaches had lead to the development of various deep tools in ergodic theory and topological dynamics (see for example \cite{Aus, H, HK05, Tao, Walsh12}).   
The most general result up to now for the $L^2$ convergence of multiple ergodic averages was given by Walsh \cite{Walsh12}, who proved the convergence of \eqref{MultipleAverage} and more general expressions, under the assumption of nilpotency of the group generated by $T_1,\ldots,T_d$.    

However, little is known about the pointwise convergence, i.e. whether (\ref{MultipleAverage}) converges for a.e. $x\in X$ as $N\to\infty$. 
 Bourgain \cite{Bourgain} gave an affirmative answer to this question for the case $d=2, T_{1}=T^{a},T_{2}=T^{b}, a,b\in\mathbb{Z}$ using methods from harmonic analysis.  
Recently, by a new method using topological models, Huang, Shao and Ye \cite{HSY} proved the pointwise convergence of (\ref{MultipleAverage}) for $T_{i}=T^{i},i=1,\dots,d$ under the assumption that $(X,\mu,T)$ is distal.

In this article, we investigate pointwise convergence of multiple averages and obtain partial results. We first obtain a pointwise limit for an easier multiple average:

\begin{thm}[Averaged multiple ergodic average] \label{THM:averagedmultiple}
Let $(X,\mu,T_1,\ldots,T_d)$ be an ergodic system with commuting transformations. Then the average
\[\frac{1}{N^{d+1}} \sum_{0\leq n_1,\ldots,n_d \leq N-1} \sum_{0\leq n\leq N-1} f_1(T_1^n \prod_{j=1}^d T_j^{n_j}x)f_2(T_2^n \prod_{j=1}^d T_j^{n_j}x)\cdots f_d(T_d^n \prod_{j=1}^d T_j^{n_j}x).  \]
converges $\mu$-a.e. $x\in X$ as $N\to\infty$ (here $\prod_{j=1}^d T_j^{n_j}$ is the composition of these transformations). 
\end{thm}

The difference between the expression in Theorem \ref{THM:averagedmultiple} and (\ref{MultipleAverage}) is that the former  includes an additional average over the {\em diagonal} transformations.

We also deduce an expression for the $L^2$-limit of a multiple average:  

\begin{thm} \label{THM:ergodicmeasurelimit}
Let $(X,\mu,T_1,\ldots,T_d)$ be an ergodic system with commuting transformations. Then there exist a collection of measures \footnote{We use the notation $\mu^{F}_x$ to indicate that these measures come from the Furstenberg-Ryzhikov self-joining $\mu^{F}$. See Section \ref{Sec:MultiplesAverages}.} $\{\mu^{F}_x\}_{x\in X}$ on $X^d$ such that for $\mu$-a.e. $x\in X$, $\mu^{F}_x$ is ergodic for $T_1\times T_2\cdots\times T_d$, and the $L^2$ limit of 
\[\lim_{N\to \infty} \frac{1}{N}\sum_{n=0}^{N_1} f_1(T_1^nx)f_2(T_2^nx)\cdots f_d(T_d^nx)\] is equal to \[\int f_1\otimes f_2\cdots\otimes  f_d d\mu^{F}_x.\]
\end{thm}

We remark that the ergodicity of the measures that describe the $L^2$-limit is a non-trivial statement. When all the transformations are the powers of a a single transformation $T$, one can use nilsystems as characteristic factors to reduce the study of the limit to the case when the system itself is a nilsystem. Limit of multiple averages for nilsystems were first described by Ziegler \cite{Z} and then by Bergelson, Host and Kra \cite{BHK} in their study of correlation sequences. Even in this case the description of the limit is non-trivial and the machinery of nilsystems is required. 

In Section \ref{Sec:Distalaverages}, we show that Theorem \ref{THM:ergodicmeasurelimit} implies the following: 

\begin{thm} \label{THM:Averagesdistal}
Let $(X,\mu,T_1,\ldots,T_d)$ be a distal (see Definition \ref{dis}) ergodic system with commuting transformations.
Then the average
\[\frac{1}{N}\sum_{n=0}^{N} f_1(T_1^nx)f_2(T_2^nx)\cdots f_d(T_d^nx) \]
converges for $\mu$-a.e. $x\in X$ as $N\to\infty$.
\end{thm}

 This theorem was proved by Huang, Shao and Ye \cite{HSY} for the case $T_{i}=T^{i}$. In a previous work \cite{DS2}, the authors proved Theorem \ref{THM:Averagesdistal} for $d=2$. The main point to remark is that the ergodicity of the measures given in Theorem \ref{THM:ergodicmeasurelimit} allows to lift the pointwise convergence of the multiple averages through isometric extension. Then standard limiting arguments and the Furstenberg-Zimmer Structure Theorem allow to prove pointwise convergence for distal systems.

\subsection{Pointwise convergence for cubic averages}
Cubic averages are averages along some {\em cubical configurations} for commuting measure preserving transformations. These cubical objects appear in the proof of $L^2$ convergence of multiple averages for commuting transformations when using inequalities derived from the Van der Corput Lemma (see \cite{H} for instance). 
The 2-dimensional cubic average is defined as
\[\frac{1}{N^2} \sum_{n_1,n_2=0}^{N-1} f_{10}(T_1^{n_1} x)f_{01}(T_2^{n_2}x)f_{11}(T_1^{n_1}T_2^{n_2}x), \]
where $f_{10},f_{01}$ and $f_{11}$ are bounded measurable functions. 
The 3-dimensional cubic average is
{\scriptsize
	\[\frac{1}{N^3} \sum_{n_1,n_2,n_3=0}^{N-1} f_{100}(T_1^{n_1} x)f_{010}(T_2^{n_2}x)f_{110}(T_2^{n_1}T_3^{n_2}x)f_{001}(T_3^{n_3}x)f_{101}(T_1^{n_1}T_3^{n_3}x)f_{011}(T_2^{n_2}T_3^{n_3}x)f_{111}(T_1^{n_1}T_2^{n_2}T_3^{n_3}x).\] 
}
More generally, the  d-dimensional cubic average is (we refer the readers to Section \ref{Sec:notation} for the notations)
\begin{equation}\label{equ:cube}
\begin{split}
\frac{1}{N^{d}} \sum_{n_{1},\dots,n_{d}=0}^{N-1}\prod_{\e\in\{0,1\}^{d}, \e\neq 0\dots 0}f_{\e}\Bigl(\prod_{i=1}^{d}T_{i}^{n_{i}\cdot\e_{i}}x\Bigr).
\end{split}
\end{equation}

The $L^{2}$ convergence of (\ref{equ:cube}) was proved by Austin \cite{Aus} and Host \cite{H} using different methods. The pointwise convergence of (\ref{equ:cube}) was proved in various ways by Assani \cite{Assani}, Chu and Frantzikinakis \cite{CF} and Huang, Shao and Ye \cite{HSY} for the case $T_{1}=T_{2}=\dots=T_{d}$. The pointwise convergence of (\ref{equ:cube})
for commuting transformations for $d=2$ was previously established by the authors in \cite{DS}.  

In this paper, we establish the pointwise convergence for (\ref{equ:cube}) in the general setting.
\begin{thm}\label{THM:cubicpointwise}
 Let $d\in\N$ and $(X,\mu,T_{1},\dots,T_{d})$ be an ergodic system with commuting transformations. 
Let $f_{\e}\in L^{\infty}(\mu)$, $\e=\epsilon_1\cdots\epsilon_d \in\{0,1\}^{d}, \e\neq 0\dots 0$ be $2^d-1$ bounded measurable functions. Then the average (\ref{equ:cube})
converges for $\mu$-a.e. $x\in X$ as $N\to\infty$.
\end{thm}

\subsection{Methods and paper organization}
Our strategy to prove pointwise convergence properties is: (a) construct an appropriate topological model for the system $X$; (b) then use this model to study the corresponding averages. This method generalizes the ideas of Huang, Shao and Ye in \cite{HSY0, HSY}, where the same questions were studied for the case $T_{i}=T^{i}$. 

The techniques used in step (b) are straightforward extensions of the ``standard'' methods used in \cite{DS, DS2,HSY0, HSY}, which are developed in Sections \ref{Sec:app1}, \ref{Sec:app2}, \ref{Sec:app3} and \ref{Sec:Distalaverages}.

It is in step (a) where innovations are involved.
While the methods in \cite{HSY0, HSY} relies heavily on the structure theorem of Host and Kra \cite{HK05} for the case $T_{i}=T^{i}$, there is no explicit structure theorem in the general setting. To overcome this difficulty, we use the theories developed by Austin \cite{Aus} and Host \cite{H} in Section \ref{Sec:Sated} to replace the structure theorem, and then use it them to prove the results of the paper in the rest sections.

In Section \ref{Sec:Preliminaries}, we provide the background materials in ergodic theory and topological dynamics. In Section \ref{Sec:Sated}, we introduce more recent tools in ergodic theory (the sated and magic extensions) and prove some general results for later uses.

We prove Theorem  \ref{THM:cubicpointwise} in Section \ref{Sec:CubicPointwise}, Theorem \ref{THM:averagedmultiple} in Section \ref{Sec:MultiplesAverages}, and Theorem \ref{THM:ergodicmeasurelimit} and \ref{THM:Averagesdistal} in Section \ref{Sec:AverageDistal}.
These sections are written in a such a way that the interested reader can read them independently. 

\section*{Acknowledgments} 
We thank Tim Austin for bringing the idea of sated extensions to the authors and for the helpful discussions relating to the materials in Section \ref{Sec:CubicPointwise}.
We also thank Bernard Host, Bryna Kra and Alejandro Maass for useful comments.

\section{Preliminaries and notation} \label{Sec:Preliminaries}
\subsection{Measure preserving and topological dynamical systems}
A {\it measure preserving system} is a tuple $(X,\mathcal{X},\mu,G)$ where $(X,\mathcal{X},\mu)$ is a probability space and $G$ is a group of measure preserving transformations of it. Sometimes we omit writing $\mathcal{X}$ to ease the notation.  We refer to $(X,\mu,G)$ as a {\em $G$-measure preserving system}. A $G$-measure preserving system is {\it ergodic} if $A=gA$ for all $g\in G$ implies that $\mu(A)=0$ or 1 for all $A\in\mathcal{X}$.

A measure preserving system $(X,\mathcal{X},\mu,G)$ is {\em free} it it has no fixed points, {\em i.e.} 
$\mu(\{x: gx=x\})=0$ for every $g\in G$ but the identity transformation.

A {\em factor map} between the measure preserving systems $(X,\mathcal{X},\mu,G)$ and $(Y,\mathcal{Y},\nu,G)$  
is a measurable function from $X$ to $Y$ such that $\pi\circ g=g\circ \pi$ for all $g\in G$ and that projects the measure $\mu$ into the measure $\nu$, {\em i.e.} $\mu(\pi^{-1}A)=\nu(A)$ for all $A\in \mathcal{Y}$. An equivalent formulation of a factor map is given by an invariant $\sigma$-algebra $\mathcal{Y}$ of $\mathcal{X}$. This equivalence is done identifying $\mathcal{Y}$ with $\pi^{-1}(\mathcal{Y})$. We freely use both notions depending on the context. 
We say that $(Y,\mathcal{Y},\nu,G)$ is a {\em factor} of $(X,\mathcal{X},\mu,G)$ or that $(X,\mathcal{X},\mu,G)$ is an {\em extension} of $(Y,\mathcal{Y},\nu,G)$. When the factor map $\pi$ is bi-measurable and bijective (modulo null sets) we say that $\pi$ is an {\em isomorphism} and that $(X,\mathcal{X},\mu,G)$ and $(Y,\mathcal{Y},\nu,G)$ are isomorphic.

A {\it topological dynamical system} $(X,G)$ consists of a compact metric space $X$ and a group of homeomorphisms $G\colon X\to X$ of the space. We say that $(X, G)$ is {\em minimal} if the only closed invariant subsets of $X$ are itself and the empty-set. This is equivalent to say that the orbit $\{gx\colon g\in G\}$ of any point $x\in X$ is dense in X. A {\it (topological) factor map} is an onto
continuous function $\pi\colon X \to Y$ such that $\pi\circ g = g \circ \pi$ for every $g\in G$.

\begin{conv}
	If the group $G$ is spanned by $T_1\ldots,T_d$,
	we also use $(X,\mu,$ $T_1,\ldots,T_d)$ to denote the system $(X,\mu,G)$.  When $T_1,\ldots,T_d$ commute, we think of $(X,\mu,T_1,\ldots,T_d)$ as a $\Z^d$ action where $e_i$, the $i$-th element of the canonical base of $\R^d$, acts as $T_i$. We use the same convention for topological systems.
\end{conv}

\subsection{Existence of strictly ergodic models}
A topological system $(X, G)$ is {\it strictly ergodic} if it is minimal and there is a unique $G$-invariant probability measure on $X$. 
We say that $(\wh{X},G)$ is a {\it strictly ergodic model} for $(X,\mu,G)$ if  $(\wh{X},G)$ is strictly ergodic with unique $G$-invariant measure $\wh{\mu}$ and $(X,\mu,G)$ is isomorphic to $(\wh{X},\wh{\mu},G)$.

The Jewett-Krieger Theorem \cite{J,K} states that every ergodic $\Z$-measure preserving system is measure theoretical isomorphic a strictly ergodic model. We use a generalization of this result in the commutative case. 
\begin{thm}[Weiss-Rosenthal, \cite{R,W}] \label{Weiss}
Let $G$ be a countable abelian group and let $\pi\colon(X,\mathcal{X},\mu,G)\to (Y,\mathcal{Y},\nu,G)$ be a factor map between ergodic and free $G$-systems. Let  $(\wh{Y},G)$ be a strictly ergodic model for $(Y,\mathcal{Y},\nu,G)$. Then there exist a strictly ergodic model $(\wh{X}, G)$ for $(X,\mathcal{X},\mu,G)$ together with a topological factor map $\wh{\pi}:\wh{X}\to \wh{Y}$ such that the following diagram commutes:
\begin{figure}[h]
 \begin{tikzpicture}
  \matrix (m) [matrix of math nodes,row sep=3em,column sep=4em,minimum width=2em,ampersand replacement=\&]
  {
     X \& \widehat{X} \\
     Y \& \widehat{Y} \\};
  \path[-stealth]
     (m-1-1) edge node [left] {$\pi$} (m-2-1)
    (m-1-1) edge node [above] {$\Phi$} (m-1-2)
    (m-1-2) edge (m-1-1)
    (m-2-2) edge (m-2-1)
    (m-1-2) edge node [right] {$\widehat{\pi}$} (m-2-2)
    (m-2-1) edge node [below] {$\phi$} (m-2-2);
\end{tikzpicture}
 \end{figure}
 
 where
$\Phi$ and $\phi$ are measure preserving isomorphisms and $\pi\circ \Phi=\phi \circ \widehat{\pi}$.
\end{thm}
In this case, we say that $\wh{\pi}\colon \wh{X}\to \wh{Y}$ is a {\it topological model} for $\pi\colon X\to Y$.

\subsection{Notation about cubes}\label{Sec:notation}
For $d\in\mathbb{N}$, $[d]$ denotes the set $\{1,\ldots,d\}$ and $V_d$ denotes the $d$-dimensional cube $\{0,1\}^d$. If $X$ is a set, we index points in $X^{[d]}=X^{2^d}$ using the coordinates in $V_d$. So a point $\vec{x}\in X^{[d]}$ is written as $\vec{x}=(x_{\epsilon})_{\epsilon \in V_d}$.

For every $\e\in V_d$ and $1\leq i\leq d$, denote $\e_{i}$ to be the $i$-th coordinate of $\e$ and write $\vert\e\vert=\e_{1}+\dots+\e_{d}$. 
	For $\e,\e'\in V_d$, we say that $\e\leq\e'$ if $\e_{i}\leq \e'_{i}$ for all  $1\leq i\leq d$. For $\e,\e'\in V_d$, let $\e\cap\e'\in V_d$ be the element such that $(\e\cap\e')_{i}=\min\{\e_{i},\e'_{i}\}$ for all $1\leq i\leq d$.

If $f_{\epsilon}$, $\epsilon \in V_d$ are $2^d$ measurable functions we let $\bigotimes_{\epsilon \in V_d} f_{\epsilon}$ denote the tensor product of the $f_{\epsilon}$'s, {\em i.e. }\footnote[1]{In this paper, we assume all the functions are real-valued to ease the notations, but the results hold for complex-valued functions as well.}

\[ \bigotimes_{\epsilon \in V_d} f_{\epsilon}\left(\vec{x}\right)=\prod_{\epsilon \in V_d} f_{\epsilon}(x_{\epsilon}) \]

If $\phi\colon X\to Y$ is a function, we let $\phi^{[d]}$ denote $\phi\times\cdots \times \phi$ ($2^d$ times), {\em i.e.} $\phi^{[d]}((x_{\epsilon})_{\epsilon \in V_d})=(\phi(x_{\epsilon}))_{\epsilon\in V_d}$.

\subsection{Host's measures} \label{Sec:HostMeasures}
For a measure preserving system $(X,\mathcal{X},\mu,G)$ and $T_{1},\dots,T_{d}\in G$,\footnote[2]{When we say ``$(X,\mu,T_{1},\dots,T_{d})$ is a measure preserving system'', we mean the group $G$ is spanned by $T_{1},\dots,T_{d}$. But when we say ``$(X,\mu,G)$ is a measure preserving system and $T_{1},\dots,T_{d}\in G$'', $G$ may contain more transformations than $T_{1},\dots,T_{d}$.} let $\I_{T_{1},\dots,T_{d}}(X)$ denote the sub $\sigma$-algebra of $\mathcal{X}$ invariant under $T_{1},\dots,T_{d}$. When there is no confusion, we write $\I_{T_{1},\dots,T_{d}}=\I_{T_{1},\dots,T_{d}}(X)$ for short.

\begin{defn}
	Let $(X,\mu, G)$ be an ergodic measure preserving system and $T_{1},\dots,T_{d}\in G$. The {\em Host measure} $\mu_{T_1,\ldots,T_d}$ is defined inductively as follows: 
	for $d=1$, define $$\mu_{T_1}=\mu \times_{\mathcal{I}_{T_{1}}}\mu.$$
	For $d\geq 1$, let
	$$\mu_{T_1,\ldots,T_{d}}=\mu_{T_1,\ldots,T_{d-1}}\times_{\I_{T_{d}^{[d-1]}}(X^{[d-1]})} \mu_{T_1,\ldots,T_{d-1}}.$$ This means that for all bounded measurable functions $f_{\epsilon}$, $\epsilon \in V_d$, we have
	
	\[\int_{X^{[d]}} \bigotimes_{\epsilon \in V_{d}} f_{\epsilon} ~ d\mu_{T_1,\ldots,T_{d}}=\int_{X^{[d-1]}} \mathbb{E}\Bigl(\bigotimes_{\eta \in V_{d-1}} f_{\eta 0}\vert\I_{T_{d}^{[d-1]}} \Bigr )\mathbb{E}\Bigl (\bigotimes_{\eta \in V_{d-1}} f_{\eta 1}\vert\I_{T_{d}^{[d-1]}} \Bigr )~ d\mu_{T_1,\ldots,T_{d-1}}.\   \] 
\end{defn}

For $i\leq d$ we define the upper and lower $i$-face transformations on $X^{[d]}$ to itself as \[(\mathcal{F}_i^{0}(x_{\epsilon})_{\epsilon\in V_d})_{\epsilon}=\begin{cases} T_ix_{\epsilon} ~ \text{ if } ~ \epsilon_i=0 \\ x_{\epsilon} ~ \text{ if } ~ \epsilon_i=1  \end{cases}
~ \text{  and  } ~ ~ \mathcal{F}_i^{1}(x_{\epsilon})_{\epsilon\in V_d}=\begin{cases} x_{\epsilon} ~ \text{ if } ~ \epsilon_i=0 \\ T_ix_{\epsilon} ~ \text{ if } ~ \epsilon_i=1  \end{cases} \] 

Remark that $\mathcal{F}_i^0 \mathcal{F}_i^1=T_i^{[d]}$. The transformations $\mathcal{F}_i^0$ and $\mathcal{F}_i^1$, $i=1,\ldots,d$ preserve the measure $\mu_{T_1,T_2,\ldots,T_d}$.

\begin{defn}
	Let $(X,\mu, G)$ be an ergodic measure preserving system and $T_{1},\dots,T_{d}\in G$.
	The {\it Host seminorm} for $f\in L^{\infty}(\mu)$ is the quantity
	\[\normm{f}_{\mu,T_{1},\ldots,T_{d}}\coloneqq \left( \int_{X^{[d]}} \bigotimes_{\epsilon \in V_d} f ~ d\mu_{T_{1},\ldots,T_{d}} \right)^{1/2^d}. \]
\end{defn}

The following properties of the Host seminorms appear basically in \cite{H}, Sections 2 and 4. 
\begin{thm}\label{ine} Let $(X,\mu,G)$ be a system with commuting transformations and $T_1,\ldots,T_d\in G$. Then
	\begin{enumerate}[leftmargin=*]
		\setlength\itemsep{1em}
		\item (Cauchy-Schwartz) $$\Bigl\vert\int_{X^{[d]}} \bigotimes_{\epsilon \in V_d} f_{\epsilon} ~ d\mu_{T_1,\ldots,T_d}\Bigr\vert \leq \prod_{\epsilon\in V_d} \normm{f_{\epsilon}}_{\mu,T_1,\ldots,T_d}. $$
		
		\item $\normm{f}_{\mu,T_1,T_2,\ldots,T_d}=\normm{f}_{\mu,T^{-1}_1,T_{2},\ldots,T_d}$.

		\item $\normm{\cdot}_{\mu,T_1,\ldots,T_d}$ does not depend on the order of the transformations, {\em i.e.} $\normm{\cdot}_{\mu,T_1,\ldots,T_d}=\normm{\cdot}_{\mu, T_{\sigma(1)},\ldots,T_{\sigma(d)}}$ for every permutation $\sigma\colon [d]\to [d]$.

		\item If $\normm{f}_{\mu,T_1,\ldots,T_d}=0$, then $\mathbb{E}(f\vert \bigvee_{i=1}^d I_{T_i})=0$. 
		
		\item If $\pi\colon (X,\mu,T_1,\ldots,T_d)\to (Y,\nu,S_1,\ldots,S_d)$ is a factor map, then 
		$$\normm{f}_{\nu,S_1,\ldots,S_d}=\normm{f\circ \pi}_{\mu,T_1,\ldots,T_d} $$

	\item If $\mu=\int \mu_x d\mu(x)$ is the ergodic decomposition of $\mu$ under $T_1,\ldots,T_d$ then 
	
	$$\normm{f}_{\mu,T_1,\ldots,T_d}^{2^d}=\int \normm{f}_{\mu_x,T_1,\ldots,T_d}^{2^d} d\mu(x) $$
	\end{enumerate}
\end{thm}

Because of the second property, if $I=\{T_1,\ldots,T_k\}$, we can write $\normm{\cdot}_{\mu,I}\coloneqq\normm{\cdot}_{\mu,T_1,\ldots,T_k}$ for short.

\section{Sated and magic extensions} \label{Sec:Sated}
The notions of sated and magic extensions were introduced by Austin \cite{A} and Host \cite{H} respectively, in order to give an ergodic theoretical proof of Tao's norm convergence for multiple averages in a commutative group \cite{Tao}. 
\subsection{Sated extensions}
The following definitions were introduced in \cite{A}. 
\begin{defn}
	Let $(X,\mathcal{X},\mu,G)$ be a measure preserving system. Let $\mathcal{X}_1,\ldots,\mathcal{X}_d$ and $\mathcal{Z}_1,\ldots \mathcal{Z}_d$ be factors of $X$ with $\mathcal{Z}_i\subseteq \mathcal{X}_i\subseteq\mathcal{X}$. We say that $\mathcal{X}_1,\ldots,\mathcal{X}_d$ are {\it relatively independent} \footnote{Do not confuse with the term relatively independent product which is classical in ergodic theory.} over $\mathcal{Z}_1,\ldots \mathcal{Z}_d$ if 
	
	\[ \int_X f_1\cdots f_d d\mu=\int_X \mathbb{E}(f_1\vert \mathcal{Z}_1)\cdots \mathbb{E}(f_d \vert \mathcal{Z}_d)d\mu  \]  for all $f_i$ measurable with respect to $\mathcal{X}_i$. 
\end{defn}

\begin{defn}
	A class $\mathcal{C}$ of $G$-measure preserving systems is {\it idempotent} if it contains the trivial system and is closed under inverse limits, joinings and under measure theoretical isomorphisms.
\end{defn}

An important idempotent class is the one defined by the triviality of the action of a given subgroup $\Lambda\subseteq G$. This class is denoted by $Z^{\Lambda}$. When $\Lambda$ is spanned by a single transformation $T$, we write $Z^{T}=Z^{\Lambda}$ for short. 

When $G=\Z^d$ is spanned by commuting transformations $T_1,\ldots, T_d$ and $J\subseteq \{1,\ldots,d\}$ we write $Z^{J}$ to denote the idempotent class $Z^{\Lambda}$, where $\Lambda$ is the subgroup spanned by $T_i$ $i\in J$. Note that this notation only makes sense when we have fixed the order of the generators.

For example, if $(X,\mu,T_1\ldots,T_d)$ is a measure preserving system, the $\sigma$-algebra of invariant sets under $T_i$ for $i\in J\subseteq \{1,\ldots,d\}$ is a factor map which belongs to the class $Z^{J}$.

For any idempotent class $\mathcal{C}$, every measure preserving system $(X,\mu,G)$ has a {\em maximal factor} (unique up to isomorphism) which belongs to the class $\mathcal{C}$ (Lemma 2.2.2 in \cite{A}). We let $\mathcal{C}_X$ and $\pi_{\mathcal{C}_X}$ denote the $\sigma$-algebra and the factor map associated to the maximal factor of $X$.

\begin{defn}
	Let $\mathcal{C}$ be an idempotent class  of $G$-measure preserving systems. We say that the measure preserving system $(X,\mu,G)$ is {\it $\mathcal{C}$-sated} if for every extension $\pi\colon(\widetilde{X},\mu',G)\to (X,\mu,G)$, $X$ and $\pi_{\mathcal{C}}(\widetilde{X})$ are relatively independent over $\pi_{\mathcal{C}}(X)$ (as factor systems of $\widetilde{X}$).
\end{defn}

The existence of $\mathcal{C}$-sated extensions was proved by Austin \cite{A} in order to derive the $L^2$ convergence of multiple averages for commuting transformations.
We state this theorem in generality:
\begin{thm}[Austin, \cite{A}]
	Let $(X,\mu,G)$ be a measure preserving system and $(\mathcal{C}_i)_{i\in I}$ be a countable collection of idempotent classes (of $G$-measure preserving systems). Then there exists an extension $\pi\colon (X',\mu',G)\to(X,\mu,G)$ such that $(X',\mu',G)$ is $\mathcal{C}_i$-sated for all $i\in I$. Furthermore, $(X',\mu',G)$ is a joining of $(X,\mu,G)$ the maximal $\mathcal{C}_i$-factors of $X'$.   
\end{thm}

For our purposes, we strengthen this result by adding an ergodicity condition for some specific idempotent classes. The bulk of the proof follows from \cite{A}, but we provide details for completion:
\begin{thm}\label{Thm:ext}
	Let $(X,\mu,G)$ be an ergodic measure preserving system and $(\mathcal{C}_i)_{i\in I}$ be a countable collection of idempotent classes (of $G$-measure preserving systems). If for all $i\in I$, $\mathcal{C}_i=\bigvee_{k} Z^{\Gamma_{i,k}}$ for subgroups $\Gamma_{i,k}\subseteq G$, then there exists an extension $\pi\colon (X',\mu',G)\to(X,\mu,G)$ such that $(X',\mu',G)$ is ergodic and is $\mathcal{C}_i$-sated for all $i\in I$.
\end{thm}
\begin{proof}

    We only prove this theorem for a single idempotent class $\mathcal{C}$ and the general case can be proved easily (we comment further on this in the end of the proof). 

	Let $(f_i)_{i\in \N}$ be a countable subset of the unit ball in $L^{\infty}(\mu)$ dense for the $L^2$-norm. Suppose additionally that all $f_j$ appear infinitely times in the sequence $(f_i)_{i\in \N}$.
We construct a sequence of ergodic extensions $X_i$ of $X$ inductively, 
	and their inverse limit will be the system we are looking for. 
	
	Suppose we have constructed an ergodic extension 
	$\pi_n\colon X_n\to X$. The extension $\pi_{n+1,n}\colon X_{n+1}\to X_{n}$ is constructed as follows.
	Let \[\alpha_n:=\sup \{ \| \mathbb{E}(f_n\circ \pi_{W}\vert \mathcal{C}({W})) \|_{L^2}   \}   \]
	where the supremum is taken over all extensions $ \pi_{W}\colon W\to X_{n}$ which are joinings of $X_n$ with some element in $\mathcal{C}$. We remark (as pointed out in \cite{A}) that one can always assume that the spaces are given by a model in a Cantor space with a Borel invariant probability measure, and so the supremum can be considered in a set rather than in a proper class.   
	
	We first claim that the supremum remains unchanged if we restrict to ergodic extensions. 
	 Let $\pi_{W}\colon(W,\mu_{W},G)\to(X_{n},\mu_{X_{n}},G)$ be an extension of $X_{n}$
	and $\mu_W=\int \mu_{z} d\nu(z)$ be the ergodic decomposition of $\mu_W$. Since $X_n$ is ergodic, we have that for $\nu$-a.e. $w$, $\pi_{W}(\mu_w)=\mu_{X_{n}}$. So for $\nu$-a.e $w$, the system $(W,\mu_{w},G)$ is an extension of $(X_{n},\mu_{X_n},G)$. 
	On the other hand, the equality
	\[ \| \mathbb{E}(f_n\circ \pi_{W}\vert \mathcal{C}_{W}) \|_{L^2(\mu_W)}^2 =\int \|\mathbb{E}(f_n\circ \pi_{W}\vert \mathcal{C}_{W}) \|_{L^2(\mu_z)}^2 d\nu(z)      \]
	implies that there exists a set of $z$ with positive $\nu$ measure such that \[ \|\mathbb{E}(f_n\circ \pi_{W}\vert \mathcal{C}_{W}) \|_{L^2(\mu_z)}^2\geq \| \mathbb{E}(f_n\circ \pi_{W}\vert \mathcal{C}_{W}) \|_{L^2(\mu_W)}^2. \]
	
	Since $\mathcal{C}_i=Z^{\Gamma}$ for some subgroup $\Gamma\subseteq G$, we have that $\nu$-a.e. this invariant sets are also invariant under the measure $\mu_z$, meaning that invariant $\sigma$-algebras with respect to the measures $\mu_z$ are finer than the one with respect to $\mu$. This implies that there is a set of $z$ with positive $\nu$ measure such that \[ \| \mathbb{E}(f_n\circ \pi_{W}\vert \mathcal{C}_{W,z}) \|_{L^2(\mu_z)} \geq \| \mathbb{E}(f_n\circ \pi_{W}\vert \mathcal{C}_{W}) \|_{L^2(\mu_W)} \]
	(we write $\mathcal{C}_{W,z}$ to emphasize its dependence on the measure $\mu_z$). This finishes the claim.
	
	Now we can take an ergodic extension W of $X_{n}$ such that 
	\[\alpha_n-2^{n}\leq  \| \mathbb{E}(f_n\circ \pi_{W}\vert \mathcal{C}_{W}) \|_{L^2}      \]
	an we put $X_{n+1}=W$.

	Let $X_{\infty}$ be the inverse limit of the systems $(X_n)_{n\in \N}$. It is an ergodic extension of $X$ and it is a joining of $X$ and a system $Y$ in $\mathcal{C}$. The rest follows equal as in \cite{A}.

\end{proof}
As we said, the case of a countable number of idempotent classes follows by applying the case of a single idempotent class several times (\cite{Aus}). To achieve this, let $(\mathcal{C}_i)_{i\in \N}$ a countable collection of idempotent classes (we index them with $\N$ for convenience) and let $(a_k)_{k\in \N}$ be a point in $\N^{\N}$, where each value  $i$ appears infinitely often (i.e. $\{k: a_k=i\}$ is infinite). Starting from the system $X=X_0$, for each $k\in \N$, can apply Theorem \ref{Thm:ext} to find an extension $X_{i+1}$ of $X_{i}$, sated with respect to the class $\mathcal{C}_{a_k}$. The inverse limit of the systems $X_{i}$, $i\in \N$ is sated with respect to all the classes simultaneously. 

\subsection{Magic extensions}

\begin{defn}
	Let $(X,\mu,G)$ be a system of commuting transformations and $T_{1},\dots,T_{d}\in G$. Denote $Z_d:=\bigvee_{i=1}^d I_{T_i}$. We say $X$ is {\it magic} for $T_1,\ldots,T_d$ (or $(X,\mu,T_1,\ldots,T_d)$ is magic) if 
	\[\normm{f}_{T_1,\ldots,T_d}=0 \text{ if and only if } \mathbb{E}\Bigl(f\vert \bigvee_{i=1}^d I_{T_i}\Bigr)=\mathbb{E}\Bigl(f\vert Z_d \Bigr)=0 \]
	for all $f\in L^{\infty}(\mu)$.
\end{defn}
The existence of magic extension was proved in \cite{H} (recall the definitions in Section \ref{Sec:HostMeasures}):
\begin{thm}[Host, \cite{H}] \label{Thm:MagicExtension}
	Let $(X,\mu,T_1,\ldots,T_d)$ be a system of commuting transformations. Then $(X^{[d]},\mu_{T_1,\ldots,T_d},\mathcal{F}_{1}^1,\ldots,\mathcal{F}_{d}^1)$ is a magic extension of $(X,\mu,T_1,\ldots,T_d)$ (the factor map is the projection into the last coordinate of $X^{[d]}$). 
\end{thm}

\begin{rem}
	If $(X,\mu,T_1,\ldots,T_d)$ is a system of commuting transformations, we can also look at a subset of the transformations $T_1\ldots,T_d$, say $T_{a_1},\ldots,T_{a_k}$, $\{a_1\ldots,a_k\}\subseteq \{1,\ldots,d\}$. 
	We can regard the ($\Z^k)$-system  $(X,\mu,T_{a_1},\ldots,T_{a_k})$ and find a corresponding magic extension $(X^{[k]},\mu_{T_{a_1},\ldots,T_{a_k}}$, $\mathcal{F}_{a_1}^1,\ldots,\mathcal{F}_{a_k}^1)$. To this system, we can add the diagonal transformations $T_{b_i}^{[k]}$, $b_i\notin \{a_1,\ldots,a_k\}$ in order to complete to $d$ transformations. Doing this, $(X^{[k]},\mu_{T_{a_1}, \ldots,T_{a_k}}$, $\mathcal{F}_{a_1}^1,$ $\ldots,$ $\mathcal{F}_{a_k}^1$, $T_{b_1}^{[k]},\ldots, T_{b_{d-k}})$ is also an extension (as $\Z^d$)-systems of $(X,\mu,T_{a_1},$ $\ldots,T_{a_k},T_{b_1},\ldots,T_{b_{d-k}})$. Of course we can rearrange the order of transformations to get an extension of $(X,\mu,T_1,\ldots,T_d)$.

\end{rem}

Using this result, we can see that being magic is a sated condition as the following shows.
\begin{lem} \label{lem:SatedAndMagic}
	Let $(X,\mu,T_1,\ldots,T_d)$ be a measure preserving system and let $I\subseteq [d]$. If $X$ is $\bigvee_{i\in I} Z^{\{i\}}$-sated, then it is magic for $T_i$, $i\in I$.
\end{lem}
\begin{proof}
	We assume that $I=[k]=\{1,\ldots,k\}$ as the general case is the same modulo small changes of notations.
	It suffices to show that $\mathbb{E}(f\vert \bigvee_{i=1}^k \mathcal{I}_{T_i})=0$ implies that $\normm{f}_{\mu,[k]}=0$.
	By Theorem \ref{Thm:MagicExtension}, we can consider a magic extension $(X',\mu',T_1',\ldots,T_k',T_{k+1}',\ldots,T_d')$ (we write it like this to ease notation) of $(X,\mu,T_1,\ldots,T_d)$ which is magic for $T_1',\ldots,T_k'$. Denote by $\pi$ the factor map. By the satedness assumption on $X$, we have that $\mathbb{E}(f\circ \pi \vert \bigvee_{i=1}^k \mathcal{I}_{T_i'})=\mathbb{E}(f \vert \bigvee_{i=1}^k \mathcal{I}_{T_i})\circ \pi =0$. Since $(X',\mu',T_1',\ldots,T_d')$ is magic for $T_1',\ldots,T_k'$, we have that $\normm{f\circ \pi}_{\mu',[k]}=0$ which implies that $\normm{f}_{\mu,[k]}=0$.  
\end{proof}

\begin{cor}\label{cubeind}
	Let $(X,\mu,G)$ be a measure preserving system with commuting transformations and $T_{1},\dots,T_{d}\in G$. The system $(X^{[d]},\mu_{T_{1},\dots,T_{d}},\mathcal{G}_{T_{1},\dots,T_{d}})$ can be viewed as a joining of $2^{d}$ copies of $X$. If $X$ is  $\bigvee_{i\in [d]} Z^{\{i\}}$-sated, then each copy $X$ of $\mu_{T_{1},\dots,T_{d}}$ is relatively independent over $\bigvee_{i=1}^{d}\I_{T_{i}}$.
\end{cor}
\begin{proof} We only prove it for the first copy as the other cases are similar. It suffices to show that 
	$$\int_{X^{[d]}} \bigotimes_{\epsilon \in V_d} f_{\epsilon} ~ d\mu_{T_1,\ldots,T_d}=0$$
	if $\E(f_{00\dots 0}\vert\bigvee_{i=1}^{d}\I_{T_{i}})=0$. By Lemma \ref{lem:SatedAndMagic}, $X$ is magic for $T_{1},\dots,T_{d}$ and so  $\normm{f_{00\dots 0}}_{\mu,[d]}=0$. By Theorem \ref{ine} (1), the proof is finished.
\end{proof}

We need the following result which will be used later.

\begin{lem} \label{lem:MeasurableI_k}
	Let $(X,\mu,T_1,\ldots,T_k)$ be a measure preserving system with commuting transformations, magic for $T_1,\ldots,T_k$.  Then the $\sigma$-algebra $\I_{T_{k}^{[k-1]}}$ in $(X^{[k-1]},\mu_{T_1,\ldots,T_{k-1}})$ is measurable with respect to $Z_k^{[k-1]}$. 
\end{lem}

\begin{proof}
 Recall that $Z_k=\bigvee_{i=1}^k \I_{T_i}$. We follow Proposition 4.7 in \cite{HK05}. The proof is similar to the one in \cite{DS}. It suffices to show that 
	\[\mathbb{E} \Bigl ( \bigotimes_{\epsilon \in \{0,1\}^{k-1}} f_{\epsilon} \vert \I_{T_{k}^{[k-1]}} \Bigr)= \mathbb{E}\Bigl(\bigotimes_{ \epsilon \in \{0,1\}^{k-1}} \mathbb{E}(f_{\epsilon}\vert Z_k) \vert\I_{T_{k}^{[k-1]}} \Bigr). \]
	
	It then suffices to prove this equality for the case when $\mathbb{E}(f_{\epsilon}\vert Z_k)=0$ for some $\epsilon \in \{0,1\}^{k-1}$. By the definition of $\mu_{T_1,\ldots,T_k}$, we have that 
	
	\begin{align*}
	\int_{X^{[k-1]}}\Bigl \vert \mathbb{E} \Bigl ( \bigotimes_{\epsilon \in \{0,1\}^{k-1}} f_{\epsilon} \vert \I_{T_{k}^{[k-1]}} \Bigr) \Bigr \vert^2d\mu_{T_1,\ldots,T_{k-1}}=& \int_{X^{k}} \bigotimes_{0\epsilon \in \{0,1\}^{k-1} }f_{\epsilon} \otimes \bigotimes_{1\epsilon \in \{0,1\}^{k-1}} f_{\epsilon}~ d\mu_{T_1,\ldots,T_k} \\
	\leq & \prod\limits_{\epsilon\in \{0,1\}^{k-1}} \normm{f_{\epsilon}}_{T_1,\ldots,T_{k}}^2, 
	\end{align*}
	and we are done since in $(X,\mu,T_1,\ldots,T_k)$,  $\normm{f_{\epsilon}}_{\mu,[k]}=0$ is equivalent to $\mathbb{E}(f_{\epsilon}\vert Z_k)=0$.
\end{proof}

\section{The pointwise convergence of cubic averages}  \label{Sec:CubicPointwise}
We prove Theorem  \ref{THM:cubicpointwise} in this section. We put all the satedness conditions we need in one definition.
\begin{defn}
	Let $(X,\mu,T_1,\ldots,T_d)$ be a measure preserving system of commuting transformations $T_1,\ldots,T_d$. We say that $(X,\mu,T_1,\ldots,T_d)$ is {\it $Z$-sated} if it is $\bigvee_{i\notin J} Z^{J\cup\{i\}}$-sated and $\bigvee_{i\in J} Z^{\{i\}}$-sated for all $J\subseteq [d]$.
\end{defn}

\begin{conv}
	In this section, for $J\subseteq [d]$, we let $\mathcal{I}_{J}$ denote the $\sigma$-algebra of invariant sets under all the transformations $T_i$, $i\in J$. For example, $\mathcal{I}_G=\mathcal{I}_{\{1,\ldots,d\}},\mathcal{I}_{T_{i}}=\mathcal{I}_{\{i\}}$. We also let $X_{J}$ to denote the factor of $X$ endowed with the sub-$\sigma$-algebra $\I_{J}$.
\end{conv}
\subsection{Topological model of the system} \label{Sec:BuildTopModel}
This section is devoted to building a suitable topological model for a sated enough ergodic measure preserving system. 
 We start with a useful lemma which states satedness conditions of some factors.

\begin{lem} \label{lem:SatednessFactors}
Let $(X,\mu,T_1,\ldots,T_d)$ be a measure preserving system with commuting transformations and let $J\subseteq [d]$. If $(X,\mu,T_1,\ldots,T_d)$ is $\bigvee_{j\notin J} Z^{J\cup \{j\}}$-sated, then $(X_{J},\mu,T_1,\ldots,T_d)$ is $\bigvee_{j \notin J} Z^{\{j\}}$-sated. 
\end{lem}

\begin{proof}
Remark first that for $i\in J$, $T_i$ is the identity transformation on $X_J$. Let $Y$ be a system in the idempotent class $\bigvee_{j \notin J} Z^{j}$ and let $\lambda$ be a joining of $X_J$ and $Y$. We define a new $\Z^d$ action on this joining, declaring that for $i\in J$, $T_i$ is the identity transformation on $\lambda$. Clearly $\lambda$ is invariant under this new action and defines a joining of $X_J$ with a member of the idempotent class $\bigvee_{j \notin J} Z^{J\cup \{j\}}$ that we call $\widetilde{Y}$ (it is the same space as $Y$ but we forget the transformations $T_i$, $i\in J$ and put identity instead). This joining can be regarded as a joining of $X$ and $\widetilde{Y}$ (by lifting the $X_J$ component to $X$) and the satedness condition on $X$ implies that this joining can be projected to $\bigvee_{j \notin J} Z^{J\cup \{j\}}(X)=\bigvee_{j\notin J} Z^{j}(X_J)$ in the $X$ component. This finishes proof.  
\end{proof}

\begin{defn}[Diagram of invariant factors]
Let $(X,\mathcal{X},T_1,\ldots,T_d)$ be an ergodic measure preserving system with commuting transformations. Its diagram of invariant factor is the commutative diagram which contains all $\sigma$-invariant algebras $\mathcal{I}_{J}$, $J\subseteq [d]$.  
\end{defn}
For example, the diagram of invariant factors of $(X,\mathcal{X},\mu,T_1,T_2,T_3)$ is 

{\tiny
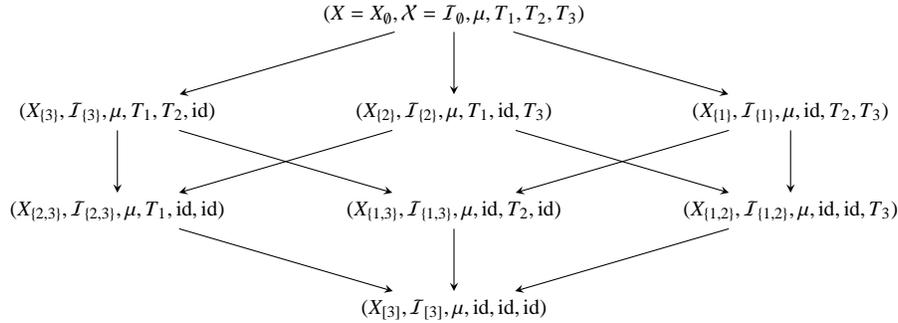
\begin{figure}[h]
 \begin{tikzpicture}
  \matrix (m) [matrix of math nodes,row sep=3em,column sep=4em,minimum width=2em,ampersand replacement=\&]
  {
     \&(X=X_{\emptyset},\mathcal{X}=\mathcal{I}_{\emptyset},\mu,T_1,T_2,T_3)\&   \\
     (X_{\{3\}},\mathcal{I}_{\{3\}},\mu,T_1,T_2,\id) \& (X_{\{2\}},\mathcal{I}_{\{2\}},\mu,T_1,\id,T_3) \& (X_{\{1\}},\mathcal{I}_{\{1\}},\mu,\id,T_2,T_3) \\
     (X_{\{2,3\}},\mathcal{I}_{\{2,3\}},\mu,T_1,\id,\id) \& (X_{\{1,3\}},\mathcal{I}_{\{1,3\}},\mu,\id,T_2,\id) \& (X_{\{1,2\}},\mathcal{I}_{\{1,2\}},\mu,\id,\id,T_3)  \\
  \& (X_{[3]},\mathcal{I}_{[3]},\mu,\id,\id,\id) \&   \\};
  \path[-stealth]
    (m-1-2) edge node[above] {$ $} (m-2-1)
    (m-1-2) edge node[above] {$ $} (m-2-2)
    (m-1-2) edge node[above] {$ $} (m-2-3)
    (m-2-1) edge node[above] {$ $} (m-3-1)
    (m-2-1) edge node[above] {$ $} (m-3-2)
    (m-2-2) edge node[above] {$ $} (m-3-1) 
    (m-2-2) edge node[above] {$ $} (m-3-3) 
    (m-2-3) edge node[above] {$ $} (m-3-2)
    (m-2-3) edge node[above] {$ $} (m-3-3)
    (m-3-1) edge node[above] {$ $} (m-4-2)
    (m-3-2) edge node[above] {$ $} (m-4-2)
    (m-3-3) edge node[above] {$ $} (m-4-2)      
    ;     
\end{tikzpicture}
\caption{Diagram of invariant factors for $(X,\mathcal{X},\mu,T_1,T_2,T_3)$} 
\end{figure}}

We remark that $\mathcal{I}_{\{i_1,\ldots,i_j\}}$ is an extension of $\bigvee\limits_{i\notin \{i_1,\ldots,i_j\}} \mathcal{I}_{\{i_1,\ldots,i_j,i\}}$.

The following proposition explains the topological model we are looking for.
\begin{prop} \label{Prop:ConstructionModel}
Let $(X,\mathcal{X},T_1,\ldots,T_d)$ be a $Z$-sated measure preserving system with commuting transformations. Suppose that all its further invariant factors are free, meaning that the $\Z^{d-\# J}$-action on $X_{J}$ induced by $T_{i}, i\notin J$ is free for all $J\subseteq [d]$. Then there exists a strictly ergodic topological model for the (measurable) diagram of invariant factors of $(X,\mathcal{X},T_1,\ldots,T_d)$, meaning that there exists a strictly ergodic model $\wh{X}_{J}$ of $X_{J}$ for all $J\subseteq [d]$ such that $\wh{X}_{J}\to \wh{X}_{J'}$ is a topological model of $X_{J}\to X_{J'}$ for all $J\subseteq J'\subseteq [d]$.
\end{prop}

\begin{rem}
The assumptions of freeness of the invariant factors is superfluous, we can pass to an extension and satisfy this condition. For instance, for any $J\subseteq [d]$ we may consider an ergodic free $\mathbb{Z}^{d-\vert J \vert}$-action on probability space $(Y_J,\nu_J)$, define $T_i=\id$ for $i\notin J$ and replace $X$ by an ergodic joining of $X$ with all the $Y_J$, $J\subseteq [d]$. It is not hard to check that we get a system where all the induced actions on the invariant factors are free. 
\end{rem}

\begin{proof}
We proceed by induction in the level of the diagram (from bottom to top). Since the factors $(X_{[d]\setminus \{i\}},T_{i})$, $i=1,\ldots,d$ are free ergodic $\Z$-systems, by the Jewett-Krieger Theorem, there exist strictly ergodic models $\wh{X}_{[d]\setminus \{i\}}$ for $X_{[d]\setminus \{i\}}$, $i=1,\ldots,d$. Moreover, it is easy to see the there is only one ergodic joining of the systems $(\wh{X}_{[d]\setminus \{i\}},\id,\dots,T_{i},\dots,\id), i=1,\ldots,d$, which is their product measure. 

What we need to be careful is that Jewett-Krieger type theorems may fail when considering a diagram with a tree form (see for example Section 8, Theorem 15.35 in \cite{Glas}). In our setting, the satedness condition we impose  allows us to overcome this difficulty.

Suppose now we have built a diagram until the depth $h+1\leq d$, i.e. we have constructed strictly ergodic models $\wh{X}_{J}$ for all factors systems 
\[ X_{J} \text{ where } J \subseteq [d], \# J =h+1 .\]  
Fix $J\subseteq [d]$ with  $\# J =h$. Recall that $\bigvee_{i\notin J}\mathcal{I}_{J\cup \{i\}}$ is a factor of $\mathcal{I}_{J}$. By induction hypothesis, we have strictly a ergodic model $\widehat{X}_{J\cup \{i\}}$ for $X_{J\cup \{i\}}$. Let $\phi_{J\cup\{i\}}$ denote the (measurable) factor map from $(X,\mu)$ to $(\widehat{X}_{J\cup \{i\}},\widehat{\mu}_{J\cup\{i\}})$.

Let $X^{*}_{J}$ denote the factor system of $X$ corresponding to the $\sigma$-algebra $\bigvee_{i\notin J}\mathcal{I}_{J\cup \{i\}}$. We look for a topological model $Y_J$ of $X^{*}_{J}$ which is a minimal subsystem of $\prod_{i \notin J} \widehat{X}_{J\cup \{i\}}$. Let

\[Y_J=\text{supp}\Bigl (\prod_{i\notin J} (\phi_{J\cup\{i\}})_{\ast}\mu\Bigr ) \text{ in } \prod_{i \notin J} \widehat{X}_{J\cup \{i\}}.  \]
Then
$Y_J$ is the smallest closed subset in $\prod_{i \notin J} \widehat{X}_{J\cup \{i\}}$ with measure 1.

{\bf Claim:} $Y_J$ is strictly ergodic. 

We start with an important property of $Y_J$. Consider two coordinates in  $Y_J$, say the $J\cup \{i_1\}$ and $J\cup \{i_2\}$ coordinates. Let $J_1,J_2\subseteq [d]$ such that $J\cup \{i_1\}\cup J_1= J\cup \{i_2\}\cup J_2\coloneqq\tilde{J}$. Let  $\widehat{\phi}_{J_1}$ and $\widehat{\phi}_{J_2}$ be the (continuous) projections from the coordinates $J\cup \{i_1\}$ and $J\cup \{i_2\}$ onto their respective $\tilde{J}$-factor. Then their projections coincide. More precisely, we have that the $J\cup \{i_1\}$ and $J\cup \{i_2\}$ coordinates of  a typical point in $Y_J$ have the form $(\phi_{J\cup\{i_1\}}(x),\phi_{J\cup \{i_2\} }(x))$ (typical in the sense that they have measure 1) and thus 
\[(\widehat{\phi}_{J_1}(\phi_{J\cup\{i_1\}}x),\widehat{\phi}_{J_2}(\phi_{J\cup\{i_2\}}x))=(\phi_{\tilde{J}}x,\phi_{\tilde{J}}x)\in \Delta_{\widehat{X}_{\tilde{J}}}, \]
where $\Delta_{\widehat{X}_{\tilde{J}}}$ is the diagonal of $\widehat{X}_{\tilde{J}}$. Since the projections $\widehat{\phi}_{J_1}$ and $\widehat{\phi}_{J_2}$ are continuous, we have that $\widehat{\phi}_{J_1}^{-1}\times \widehat{\phi}_{J_2}^{-1}(\Delta_{\widehat{X}_{\tilde{J}}})$ is a closed set in $\widehat{X}_{J\cup \{i_1\}}\times \widehat{X}_{J\cup \{i_2\}}$ and is of measure 1. Therefore, the projection of $Y_{J}$ onto the coordinates $J\cup\{i_1\}$ and $J\cup\{i_2\}$ is a subset of $\widehat{\phi}_{J_1}^{-1}\times \widehat{\phi}_{J_2}^{-1}(\Delta_{\widehat{X}_{\tilde{J}}})$. We refer to this property as {\em non-degeneracy of common invariant factors}. 

We are now ready to prove the claim. We prove inductively that the measure is determined by its projection onto the factors of $Y_J$ of level $d-1$. 
 
Let $\lambda_J$  be an ergodic measure on $Y_J$. The projection of $\lambda_J$ onto any $\widehat{X}_{J\cup \{i\}}$ is an ergodic measure and thus it is the unique ergodic measure $\widehat{\mu}_{J\cup\{i\}}$. Hence $\lambda_J$ is a joining of the systems $X_{J\cup \{i\}}$, $i=1,\ldots,d$. 

Take any $i\notin J$ and think of the measure $\lambda_J$ as a joining of $X_{J\cup \{i\}}$ with some system in the idempotent class $\bigvee_{j\notin J\cup \{i\}} Z^{\{j\}}$. This is possible since for $i'\neq i$, $X_{J\cup \{i'\}}$ is a member of the idempotent class $Z^{\{i'\}}$.  
By Lemma \ref{lem:SatednessFactors}, $X_{J\cup \{i\}}$ is $\bigvee_{j\notin J\cup \{i\}} Z^{\{j\}}$-sated  for all $i$. So $\lambda_J$ can be projected in the $X_{J\cup \{i\}}$-part to the factor corresponding to the $\sigma$-algebra \[ \bigvee_{j,i \notin J, j\neq i} \mathcal{I}_{J\cup \{i,j\}}. \]

Arguing similarly in the other coordinates, we have that the joining $\lambda_J$ is uniquely determined by its projection onto the systems of level $h+1$. It is worth noting that the non-degeneracy of common invariant factors in $Y_J$ ensures that the same invariant factor arising from two different coordinates in $Y_J$ are the same. This ensures that when projecting $\lambda$ onto the invariant factors of level $h+1$, there are no multiple copies of a given factor.   

Assume now that $\lambda_J$ is determined by its projection onto its factors of level $k$ for $h+1\leq k<d-1$. We show that it is also determined by its projection onto its factors of level $k+1$.

Following the notation from the beginning of the claim, the projection from $Y_J$ to its factors of level $k$ is $(\prod\limits_{J\subseteq J', \# J'=k} \widehat{\phi}_{J'} )Y_J$ and a similar expression holds for $k+1$.
Pick one of the factors, say $\widehat{\phi}_{J_0}(Y_J)=\widehat{X}_{J_0}$, and think of $(\prod\limits_{J\subseteq J', \# J'=k} \widehat{\phi}_{J})_{\ast}\lambda_J$ as a joining of  $X_{J_0}$ with the rest of the systems 
$\widehat{\phi}_{J'}(Y_J)=X_{J'}$, $J'\neq  J_0$, $\#J'=k$. All these other systems can be put together in the idempotent class $\bigvee_{j \notin J_0} Z^{\{j\}}$ since there exists $j\in J'\setminus J_0$ for each $J'$. By Lemma \ref{lem:SatednessFactors}, the factor $X_{J_0}$ is $\bigvee_{j \notin J_0} Z^{\{j\}}$- sated and therefore  $(\prod\limits_{J\subseteq J', \# J'=k} \widehat{\phi}_{J})_{\ast}\lambda_J$ can be pushed in the $X_{J_0}$ part to the factor $\bigvee_{j\notin J_0} X_{J_0\cup \{j\}}$. Arguing similarly for each $J'$, $\#J'=k$, the joining can be pushed to a joining of the systems of level $k+1$. Here it is important to stress that the non degeneracy of invariant factors discussed previously guarantees that in the joining there are no multiple (non isomorphic) copies of the same factor of level $k+1$. 
 
We get in the end that $\lambda_J$ is determined by its projection onto the factors of level $d-1$, which are $\Z$-systems. More precisely, these systems are $\widehat{X}_{[d]\setminus\{i\}}$ for $i\notin J$. It is not hard to see that $(\prod_{i\notin J} \widehat{X}_{[d]\setminus\{i\}}, \langle \{T_i\}_{i\notin J} \rangle)$ is a product of strictly ergodic systems and thus it is strictly ergodic, too (see for instance \cite{DS}, Section 4). So $Y_J$ is uniquely ergodic.

 Since any subsystem of $Y_{J}$ would have an invariant measure with support smaller than $\prod_{i\notin J} (\phi_{J\cup\{i\}})_{\ast}\mu$, we have that $Y_J$ is minimal and thus strictly ergodic. This concludes the claim.

\

By Theorem \ref{Weiss}, there exists a strictly ergodic models $\wh{X}_{J}\to Y_J$ for $X_{J}\to X^{*}_{J}$ for all $J\subseteq [d]$ with $\# J=h$. This concludes the induction. We end up at $h=0$ with a strictly ergodic model $\wh{X}_{\emptyset}$ for $X_{\emptyset}=X$, which finishes the proof.  
\end{proof}

\subsection{Strict ergodicity for the cube structure}\label{Sec:ts}
Let $(X,\mu,T_{1},\dots,T_{d})$ be a measure preserving system with commuting transformations. Recall from Section \ref{Sec:HostMeasures} that $T_i^{[d]}$ denote the diagonal transformation of $T_i$ on $X^{[d]}$ and that for $i\leq d$, the upper and lower $i$-face transformations on $X^{[d]}$ to itself are \[(\mathcal{F}_i^{0}(x_{\epsilon})_{\epsilon\in V_d})_{\epsilon}=\begin{cases} T_ix_{\epsilon} ~ \text{ if } ~ \epsilon_i=0 \\ x_{\epsilon} ~ \text{ if } ~ \epsilon_i=1  \end{cases}
~ \text{  and  } ~ ~ \mathcal{F}_i^{1}(x_{\epsilon})_{\epsilon\in V_d}=\begin{cases} x_{\epsilon} ~ \text{ if } ~ \epsilon_i=0 \\ T_ix_{\epsilon} ~ \text{ if } ~ \epsilon_i=1  \end{cases} \]

\begin{notation}
	Let $(X,\mu,T_{1},\dots,T_{d})$ be a measure preserving system with commuting transformations. We let $\mathcal{G}_{T_1,\ldots,T_d}$ denote the group spanned by the upper and lower face transformations. 
	If $k\leq d$, $\mathcal{G}_{T_1,\ldots,T_k}$ denote the projection of $\mathcal{G}_{T_1,\ldots,T_d}$ into the first $2^k$ coordinates. Thus $\mathcal{G}_{T_1,\ldots,T_k}$ is spanned by the upper face transformations of $T_1,\ldots,T_k$ in $X^{[k]}$ and the $d$ diagonal transformations in $X^{[k]}$.
\end{notation}

Let $(X,T_1,\ldots,T_d)$ be a topological dynamical system. We define the space of cubes of $X$ associated to $T_1,\ldots, T_d)$ as 

\[\Q_{T_1,\ldots,T_d}(X)=\overline{ \left \{ \left (\prod_{i=1}^d T_i^{n_i\e_i}x\right )_{\e \in V_d} \colon x\in X, (n_1,\ldots,n_d)\in \Z^d \right \} } \]

For example, for a system of three commuting transformations $(X,T_1,T_2,T_3)$, $\Q_{T_1,T_2,T_3}(X)$ is the closure in $X^8$ of the set 
\[(x,T_1^{n_1}x,T_2^{n_2}x,T_1^{n_1}T_2^{n_2}x,T_3^{n_3}x,T_1^{n_1}T_3^{n_3}x,T_2^{n_2}T_3^{n_3}x,T_1^{n_1} T_2^{n_2}T_3^{n_3}x) \] 
where $x\in X$ and $(n_1,n_2,n_3)\in \Z^3$. 

The topological structure $\Q_{T_1,T_2}(X)$ was studied in \cite{DS1} and it was shown that it provides criteria to characterize product systems and their factors. Up to now, we do not know what are the analogous results for spaces of cubes with more than 2 transformations. 

The space $\Q_{T_1,\ldots,T_d}(X)$ is invariant under the upper face transformations $\mathcal{F}_i^{1}$, $i=1,\ldots,d$ and under the diagonal ones $T_{i}^{[d]}$, $i=1,\ldots,d$. Thus, it also invariant under $T_i^{[d]}(\mathcal{F}_i^1)^{-1}=\mathcal{F}_i^{0}$. Moreover, we have

\begin{prop} \label{Prop:MinimalQ}
	Let $(X,T_1,\ldots,T_d)$ be a minimal system with commuting transformations. Then $(\Q_{T_1,\ldots,T_d}(X),\mathcal{G}_{T_1,\ldots,T_d})$ is a minimal topological dynamical system.
\end{prop}

\begin{proof}
	The proof is identical to the one given in Proposition 3.5 in \cite{DS1}, which follows Glasner's  proof of a similar result in page 46 in \cite{Glas}.  
\end{proof}

The proof of the next theorem is similar to the one given in Proposition \ref{Prop:ConstructionModel}, but involves different copies of some systems and more actions. Having more actions is useful because it allows to choose or combine them to get suitable joinings.

\begin{thm} \label{Thm:StrictlyModel}
Every $Z$-sated measure preserving system $(X,\mu,T_1,\ldots,T_d)$  has a strictly ergodic topological model $(\widehat{X}$ $,T_1,\ldots,T_d)$ such that $(\Q_{T_1,\ldots,T_d}(\wh{X})$ $,\mathcal{G}_{T_1,\ldots,T_d})$ is strictly ergodic.
\end{thm}

To help the readers better understand the proof of the general case, we start with a simple case.

\begin{proof}[Proof of Theorem \ref{Thm:StrictlyModel}, d=3]
	Let $\wh{X}_{J}$ be constructed as in Proposition \ref{Prop:ConstructionModel} for all $J\subseteq[3]$.
	By Theorem \ref{Prop:MinimalQ}, $(\Q_{T_1,\ldots,T_3}(\wh{X}),\mathcal{G}_{T_1,\ldots,T_3})$ is minimal. So it suffices to prove it is uniquely ergodic.
	
  The 3 dimensional cube $V_{3}$ has 8 vertices, 12 edges and 6 faces. Let $\vec{v}_{1}=(1,0,0), \vec{v}_{2}=(0,1,0),\vec{v}_{3}=(0,0,1)$.  On each vertex of $V_{3}$, we assign a copy of $\wh{X}$ and denote them by $\wh{X}_{1},\dots,\wh{X}_{8}$ (as in Figure \ref{CubeFactors}). On each edge of $V_{3}$ which is parallel to $\vec{v}_{i}$, we assign a copy of $\wh{X}_{\{i\}}$ and denote them by $\wh{E}_{1},\dots,\wh{E}_{12}$. On each face of $V_{3}$ which is perpendicular to $\vec{v}_{i}$, we assign a copy of $\wh{X}_{[3]\backslash\{i\}}$ and denote them by $\wh{F}_{1},\dots,\wh{F}_{6}$. 
  
  The group $\mathcal{G}_{T_1,\ldots,T_3}$ is a $\Z^{6}$-action generate by $\mathcal{F}_{i}^{j},j\in\{0,1\}, i=1,2,3$, each of which correspond to a 2 dimensional face on $V_{3}$. To be more precise, $\mathcal{F}_{1}^{0},\mathcal{F}_{2}^{0},\mathcal{F}_{3}^{0},\mathcal{F}_{1}^{1},\mathcal{F}_{2}^{1},\mathcal{F}_{3}^{1}$ correspond to the faces $\e_{1}=0,\e_{2}=0,\e_{3}=0,\e_{1}=1,\e_{2}=1,\e_{3}=1$, respectively. On each $Y=\wh{X}_{k}$, $\wh{E}_{k}$ or $\wh{F}_{k}$, we define the action $\mathcal{F}_{i}^{j}$ to be $T_{i}$ if the vertex/edge/face to which $Y$ belongs is contained in the face corresponding to the action $\mathcal{F}_{i}^{j}$, and let $\mathcal{F}_{i}^{j}$ be the identity map otherwise. For example,  $\mathcal{F}_{1}^{0},\mathcal{F}_{2}^{0},\mathcal{F}_{3}^{0},\mathcal{F}_{1}^{1},\mathcal{F}_{2}^{1},\mathcal{F}_{3}^{1}$ act respectively as
  \begin{itemize}
  	\item  $T_{1},\id,\id,\id,T_{2},T_{3}$ on the copy $Y_{1}$ of $\wh{X}$ located at $(0,1,1)$; 
  	\item  $T_{1},\id,\id,\id,T_{2},\id$ on the copy $Y_{2}$ of $\wh{X}_{\{3\}}$ located at the edge between $(0,1,0)$ and $(0,1,1)$;
  	\item  $T_{1},\id,\id,\id,\id,\id$ on the copy $Y_{3}$ of $\wh{X}_{\{2,3\}}$ located at the face containing $(0,0,0),(0,1,0),(0,0,1)$ and $(0,1,1)$.
  \end{itemize}
  It is easy to see that the projection of $(\Q_{T_1,\ldots,T_3}(\wh{X}),\mathcal{G}_{T_1,\ldots,T_3})$ on each coordinate is $(\wh{X}_{k},\mathcal{G}_{T_1,\ldots,T_3}),k=1,\dots,8$. Moreover, if $Y=\wh{X}_{k}$ is on a vertex contained in the edge where $Y'=\wh{E}_{k'}$ is located, or $Y=\wh{E}_{k}$ is on an edge contained in the face where $Y'=\wh{F}_{k'}$ is located, then  $(Y',\mathcal{G}_{T_1,\ldots,T_3})$ is a factor system of $(Y,\mathcal{G}_{T_1,\ldots,T_3})$ where the factor map is the projection onto the smaller invariant sub-$\sigma$-algebra (for example, in the previous example, we have the factor maps $Y_{1}\to Y_{2}\to Y_{3}$).
  
  We are now ready to prove the theorem. Let $\lambda_{3}$ be a joining of $\wh{X}_{k},1\leq k\leq 8$, $\lambda_{2}$ be a joining of $\wh{E}_{k},1\leq k\leq 12$, $\lambda_{1}$ be a joining of $\wh{F}_{k},1\leq k\leq 6$ (with respect to $\mathcal{G}_{T_1,\ldots,T_3}$). It suffices to show that $\lambda_{3}$ is unique. 
  
  Consider the system $(\Q_{T_1,\ldots,T_3}(\wh{X}),\mathcal{F}_{1}^{0},\mathcal{F}_{2}^{0},\mathcal{F}_{3}^{0})$ (we choose 3 out of the 6 transformations). Its projection to the coordinate (0,0,0) is $(\wh{X},T_{1},T_{2},T_{3})$, while its projection to any other coordinate is of the form $(\wh{X},T'_{1},T'_{2},T'_{3})$ with $T'_{i}=T_{i}$ or $\id$, and at least one of $T'_{i}=\id$. This implies that $(\Q_{T_1,\ldots,T_3}(\wh{X}),\lambda_{3},\mathcal{F}_{1}^{0}$, $\mathcal{F}_{2}^{0},\mathcal{F}_{3}^{0})$ is a joining of $(\wh{X},T_{1},T_{2},T_{3})$ and 7 other systems which can be put together in the idempotent class $\bigvee_{i=1}^{3}Z^{\{i\}}$. 
  Since $X$ is Z-sated, the first part $\wh{X}$ of the joining is relatively independent over $\bigvee_{i=1}^{3}Z^{\{i\}}(\wh{X})=\bigvee_{i=1}^{3} \wh{X}_{\{i\}}=\bigvee_{k} \wh{E}_{k}$, where the last joining is taken over all $\wh{E}_{k}$ which are located at an edge containing (0,0,0) (there are 3 of them). We may argue in a similar way for other coordinates of $V_{3}$ (by choosing 3 other transformations). Finally, by viewing the 6 transformations altogether, we conclude that $\lambda_{3}$ is relatively independent over a joining of $\wh{E}_{k},k=1,\dots,12$. It is worth noting that the non-degeneracy of common invariant factors (see the discuss in the proof of the general case) ensures that there are no multiple copies of $\wh{E}_{k}$ in this joining. 
  
  Therefore it suffices the show that $\lambda_{2}$ is unique. Let $\wh{E}_{1}$ be located on the edge $(0,0,0)-(1,0,0)$ and consider the system $(\prod_{k=1}^{12}(\wh{E}_{k}),\mathcal{F}_{1}^{0},\mathcal{F}_{2}^{0},\mathcal{F}_{3}^{0})$. Its projection to $\wh{E}_{1}$ is $(\wh{X}_{\{1\}},\id,T_{2},T_{3})$, while its projection to any other coordinate is of the form $(\wh{X}_{\{j\}},T'_{1},T'_{2},T'_{3})$ with $T'_{i}=T_{i}$ or $\id$, and at least one of $T'_{2}$ and $T'_{3}$ is the identity map (we leave its verification to the readers). This implies that $(\prod_{k=1}^{12}(\wh{E}_{k}),\mathcal{F}_{1}^{0},\mathcal{F}_{2}^{0},\mathcal{F}_{3}^{0})$ is a joining of $(\wh{X}_{\{1\}},\id,T_{2},T_{3})$ and 11 other systems which can be put together in the idempotent class $Z^{\{1,2\}}\vee Z^{\{1,2\}}$. By Lemma \ref{lem:SatednessFactors}, 
  $\wh{X}_{\{1\}}$ is $Z^{\{2\}}\vee Z^{\{3\}}$-sated, and so the first part $\wh{E}_{1}$ of the joining is relatively independent over $Z^{\{2\}}\vee Z^{\{3\}}(\wh{X}_{\{1\}})=\wh{X}_{\{1,2\}}\vee\wh{X}_{\{1,3\}}=\bigvee_{k} \wh{F}_{k}$, where the last joining is taken over all $\wh{F}_{k}$ which are located at a face containing the edge $(0,0,0)\to (1,0,0)$ (there are 2 of them). We may argue in a similar way for other coordinates (by choosing 3 other transformations). Finally, by viewing the 6 transformations altogether, we conclude that $\lambda_{2}$ is relatively independent over a joining of $\wh{F}_{k},k=1,\dots,6$. Again the non-degeneracy of common invariant factors ensures that there are no multiple copies of $\wh{F}_{k}$ in this joining. 
  
  It now suffices to show that $\lambda_{1}$ is unique. It is easy to verify that if $\wh{E}_{k}\cong\wh{X}_{[3]\backslash{i}}$ is located at the face $\e_{i}=j, i=1,2,3,j\in\{0,1\}$, then $\mathcal{F}_{i'}^{j'}$ acts as $T_{i}$ for $(i',j')=(i,j)$, and all other $\mathcal{F}_{i'}^{j}$ are the identity map. This implies that $(\prod_{k=1}^{6}\wh{F}_{k},\mathcal{G}_{T_{1},\dots,T_{3}})$ is a product system. Since the systems $(\wh{X}_{[3]\backslash{i}}, T_{i})$, $i=1,2,3$ are uniquely ergodic, their product is uniquely ergodic as well (see for instance \cite{DS}, Chapter 4) and $\lambda_{1}$ is their product measure.
\end{proof}

{\scriptsize
\begin{figure}[h] \label{CubeFactors}
 \begin{tikzpicture}
  \matrix (m) [matrix of math nodes,row sep=3em,column sep=0em,minimum width=2em,ampersand replacement=\&]
  {
    \& (\widehat{X}_7,T_1,\id,\id,\id,T_2,T_3) \& \& (\widehat{X}_8,\id,\id,\id,T_1,T_2,T_3) \\
 (\widehat{X}_3,T_1,T_2,\id,\id,\id,T_3) \& \& (\widehat{X}_4,\id,T_2,\id,T_1,\id,T_3)    \\
 {} \& (\widehat{X}_5,T_1,\id,T_3,\id,T_2,\id) \& \& (\widehat{X}_6,\id,\id,T_3,T_1,T_2,\id)\\
     (\widehat{X}_1,T_1,T_2,T_3,\id,\id,\id) \&  \& (\widehat{X}_2,\id,T_2,T_3,T_1,\id,\id) \\};
  \path
    (m-4-1) edge (m-4-3) 
    (m-4-1) edge (m-2-1)
    (m-2-1) edge (m-1-2) 
    (m-1-2) edge (m-1-4)
    (m-4-3) edge (m-3-4)
    (m-4-1) edge (m-3-2) 
    (m-3-2) edge (m-3-4) 
    (m-3-4) edge (m-1-4)  
    (m-4-3) edge (m-2-3)
    (m-3-2) edge (m-1-2)
    (m-2-1) edge (m-2-3)
    (m-2-3) edge (m-1-4)  ;
         
\end{tikzpicture} \caption{The eight copies of X placed in the vertices of $V_3$.}
 \end{figure}
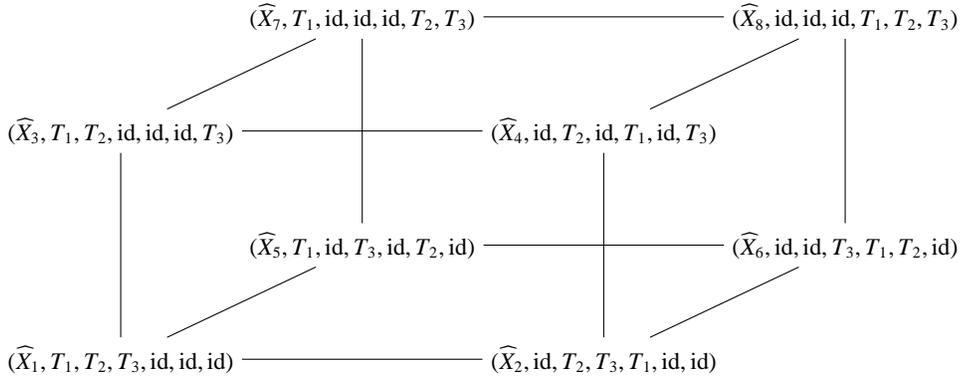 
 }
 
 {\small
  \begin{figure}[h]
 \begin{tikzpicture}
  \matrix (m) [matrix of math nodes,row sep=3em,column sep=2em,minimum width=2em,ampersand replacement=\&]
  {
     (\widehat{X}_3,T_1,T_2,\id,\id,\id,T_3) \& \& (\widehat{X}_4,\id,T_2,\id,T_1,\id,T_3)  \\
      \& (\widehat{F}_1,\id,T_2,\id,\id,\id,\id) \& \& \\
     (\widehat{X}_1,T_1,T_2,T_3,\id,\id,\id) \&  \& (\widehat{X}_2,\id,T_2,T_3,T_1,\id,\id)  \\};
  \path
    (m-1-1) edge node[above] {$(\widehat{E}_4,T_1,T_2,\id,\id,\id,\id)$} (m-3-1)
    (m-1-1) edge node[above] {$(\widehat{E}_3,\id,T_2,\id,T_1,\id,T_3)$} (m-1-3)
    (m-3-3) edge node[above] {$(\widehat{E}_2,\id,T_2,\id,T_1,\id,\id)$}(m-1-3)
    (m-3-1) edge node[below] {$(\widehat{E}_1,\id,T_2,T_3,\id,\id,\id)$}(m-3-3);
   ;      
\end{tikzpicture} 
\caption{Four vertices and their four edges and one face associated. The vertices are different copies of the system, the edges are the factor associated to the $\sigma$-algebra invariant under one transformations and the faces are associated to the $\sigma$-algebra invariant under two transformations. For instance, $\widehat{E}_{1}$ corresponds to the $\sigma$-algebra of $T_1$-invariant sets of $\widehat{X}_1$ (or $\widehat{X}_2$) and $\widehat{F}_1$ corresponds to the $\sigma$-algebra of $T_1,T_3$-invariant sets (of $\widehat{X}_1$ or $\widehat{X}_2$ or $\widehat{X}_3$ or $\widehat{X}_4$)}
 \end{figure}
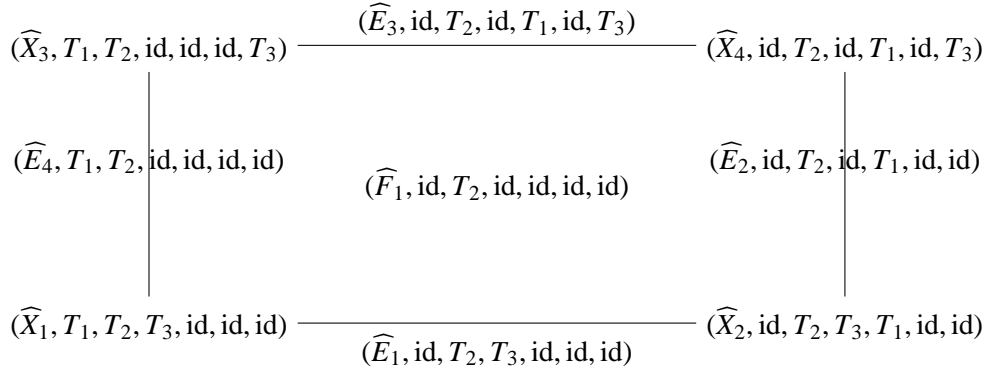
}

\begin{proof} [Proof of Theorem \ref{Thm:StrictlyModel}, the general case]
 We say that $V\subseteq V_{d}$ is a {\it face} if there exists $J\subseteq[d]$ and $a_{j}\in\{0,1\}, j\in J$ such that $V=V_{(a_{j})_{j\in J}}\coloneqq\{\e\in V_{d}\colon \e_{j}=a_{j}, j\in J\}$. We say the {\it dimension} of $V$ is $\dim(V)=d-\vert J\vert$. For each face $V_{(a_{j})_{j\in J}}$, we define a $\Z^{2d}$-system $(\wh{X}_{(a_{j})_{j\in J}},T_{1}^{0},\dots,T_{d}^{0},T_{1}^{1},\dots,T_{d}^{1})$ as the following. Let $\wh{X}_{(a_{j})_{j\in J}}=\wh{X}_{[d]\backslash J}$. Let $T_{i}^{\delta}=T_{i}$ if $V_{(a_{j})_{j\in J}}\subseteq\{\e\in V_{d}\colon \e_{i}=\delta\}$ and $T_{i}^{\delta}=\id$ otherwise.

 We write $V_{(a_{j})_{j\in J}}\leq V_{(a'_{j})_{j\in J'}}$ if $J'\subseteq J$ and $a_{j}=a'_{j}$ for all $j\in J'$. 
For all $V_{(a_{j})_{j\in J}}\leq V_{(a'_{j})_{j\in J'}}$, there is a natural (topological) factor map $\wh{X}_{(a'_{j})_{j\in J'}}\to \wh{X}_{(a_{j})_{j\in J}}$. All such factor maps induces a {\it diagram of factors} which is topological. This implies that the same diagram holds when considering any invariant measure on $\Q_{T_1,\ldots,T_d}(\widehat{X})$, so $\lambda$ has the non-degeneracy of common invariant factors (recall its meaning in the proof of Proposition \ref{Prop:ConstructionModel}).

An ergodic measure $\lambda_{d}$ in $(\Q_{T_1,\ldots,T_d}(\widehat{X}),\mathcal{G}_{T_1,\ldots,T_d})$ can be viewed as the joining of $\wh{X}_{(a_{j})_{j\in J}}$ with $\# J=d$.
Let $1<h\leq d$ and consider a joining $\lambda_{h}$ of all the systems $\wh{X}_{(a_{j})_{j\in J}}$ with $\# J=h$. We prove that $\lambda_{h}$ is relatively indepent over a joining $\lambda_{h-1}$ of $\wh{X}_{(a_{j})_{j\in J}}$ with $\# J=h-1$. Fix an arbitrary face $V_{(a_{j})_{j\in J}}$ with $\# J=h$ and 
 consider the system $(\wh{X}_{(a_{j})_{j\in J}},T_{1}^{a_{1}},\dots,T_{d}^{a_{d}})$ (for $j\notin J$, we pick an arbitrary $a_{j}\in\{0,1\}$ freely and fix it).
All these transformations leave $\lambda$ invariant, and the projection of $T_{j}^{a_{j}}$ acts as $T_{j}$ for all $j\in J$. On the other hand, for any other $\wh{X}_{(a'_{j})_{j\in J'}}$ with $\# J'=h$, at lease one of $T_{j}^{a_{j}}$, $j\in J$ act trivially. To see this, if $J'\neq J$, then $T_{j}^{a_{j}}=\id$ on  $\wh{X}_{(a'_{j})_{j\in J'}}$ for all $j\in J\backslash J'$; if $J'=J$, then $T_{j}^{a_{j}}=\id$ on  $\wh{X}_{(a'_{j})_{j\in J'}}$ for all $j\in J$ such that $a_{j}\neq a'_{j}$.
Hence, we can regard $(\prod_{\# J'=h} \wh{X}_{(a'_{j})_{j\in J'}},\lambda_{h},T_{1}^{a_{1}},\dots,T_{d}^{a_{d}})$ as a joining of $(\wh{X}_{(a_{j})_{j\in J}},T_{1}^{a_{1}},\dots,T_{d}^{a_{d}})$ with many systems that can be put together in the idempotent class $\bigvee_{j\in J} Z^{\{j\}}$. 
By Lemma \ref{lem:SatednessFactors}, $\wh{X}_{(a_{j})_{j\in J}}=\wh{X}_{[d]\backslash J}$ is $\bigvee_{j\in J} Z^{\{j\}}$-sated and thus the $\wh{X}_{(a_{j})_{j\in J}}$ coordinate of $\lambda_{h}$ can be pushed to $\bigvee_{j\in J} Z^{\{j\}}(\wh{X}_{(a_{j})_{j\in J}})= \bigvee_{V_{(a'_{j})_{j\in J'}}\leq V_{(a_{j})_{j\in J}},\# J'=h-1} \wh{X}_{(a'_{j})_{j\in J'}}$. We can argue similarly for all the other systems $\wh{X}_{(a_{j})_{j\in J}}$ with $\# J=h$ with a different choice of d transformations $T_{1}^{a_{1}},\dots,T_{d}^{a_{d}}$. Now viewing the $\Z^{2d}$ action as a whole, we conclude that $\lambda_{h}$ is relatively independent over a joining $\lambda_{h-1}$ of $\wh{X}_{(a_{j})_{j\in J}},\# J=h-1$.  Note that the non-degeneracy of common invariant factors ensures there are no multiple copies of $\wh{X}_{(a_{j})_{j\in J}}$ in the joining $\lambda_{h-1}$.  

Continuing this process, we conclude that $\lambda_{d}$ is relatively independent over a joining $\lambda_{1}$ of $\wh{X}_{(a_{j})_{j\in J}},\# J=1$. Note that there are 2d such factors, each of the form $\wh{X}_{i}^{\e}\coloneqq\wh{X}_{(a_{j}=\e)_{j\in\{i\}}}$ for some $\e\in\{0,1\}$ and $1\leq i\leq d$. Note that $T_{i'}^{\e'}$ acts as $T_{i}$ on $\wh{X}_{i}^{\e}$ if $i=i',\e=\e'$, and acts as the identity map otherwise. This implies that the joining $\lambda_{1}$ is a product system. The unique ergodicity of $(\wh{X}_{\{i\}},T_{i})$, implies that $\lambda_{1}$ is their product measure. Therefore $\lambda_{2},\lambda_{3},\dots,\lambda_{d}$ are all uniquely determined, which finishes the proof. 
\end{proof}

\subsection{Applications to pointwise results}\label{Sec:app1}
In this section, we use the topological model built in Section \ref{Sec:BuildTopModel} to derive Theorem \ref{THM:cubicpointwise}. Our strategy is as follows: we first prove the pointwise convergence of a slightly bigger average which involve both diagonal and face transformations (Theorem \ref{dexample}).  We then prove Theorem \ref{THM:cubicpointwise} when all functions involved are continuous arising from the suitable topological model. To do so, we need to apply several inequalities in order to get rid of the diagonal transformations on $(\Q_{T_1,\ldots,T_d}(\wh{X}),\mathcal{G}_{T_1,\ldots,T_d})$. Finally, by a density argument, we prove Theorem \ref{THM:cubicpointwise} for all measurable bounded functions.

We start with some notation.
For $0\leq k\leq d$, denote $E_{k}=\{\e\in V_d \colon \vert\e\vert\leq k\}$, $D_{k}=\{\e\in V_d\colon \vert\e\vert =k\}$. For $\e\in V_{d}$, let  $I_{\e}=\{i\in\{1,\dots,d\}\colon \e_{i}=1\}$. 
For $\sigma\in D_{k}$, $f_{\e}\in L^{\infty}(\mu),\e\leq\sigma$, denote
$$I(\{f_{\e}\}_{\e\leq\sigma})\coloneqq\int_{X^{[k]}} \bigotimes_{\e\leq\sigma}f_{\e}d\mu_{\sigma}.$$

\begin{thm}\label{dexample}
 Let $d\in\N$ and $(X,\mu,T_{1},\dots,T_{d})$ be an ergodic system with commuting transformations. 
Let $f_{\e}\in L^{\infty}(\mu)$ for $\e\in\{0,1\}^{d}$. Then
 $$ \lim_{N\to\infty}\frac{1}{N^{2d}} \sum_{m_{1},\dots,m_{d}=0}^{N-1}\sum_{n_{1},\dots,n_{d}=0}^{N-1}\prod_{\e\in\{0,1\}^{d}}f_{\e}\Bigl(\prod_{i=1}^{d}T_{i}^{m_{i}+n_{i}\cdot\e_{i}}x\Bigr)$$
converges for $\mu$-a.e. $x\in X$. Moreover, if $X$ is $Z$-sated, this limit is $$I(\{f_{\e}\}_{\e\leq(1\dots 1)})=\int_{X^{[d]}} \bigotimes_{\e\in\{0,1\}^{d}}f_{\e}d\mu_{T_{1},\dots,T_{d}}.$$
\end{thm}
\begin{proof}
Since it suffices to prove the convergence of this average in an extension system of $X$, we may assume that $(X,\mu,T_{1},\dots,T_{d})$ is a free ergodic $Z$-sated system. By Theorem  \ref{Thm:StrictlyModel}, we may take a strictly topological model $(\wh{X},\wh{T}_{1},\dots,\wh{T}_{d})$ for $X$ such that $({\bf Q}_{\wh{T}_{1},\dots,\wh{T}_{d}}(\wh{X}),\mathcal{G}_{\wh{T}_{1},\dots,\wh{T}_{d}})$ is strictly ergodic. It now suffices to work on $(\wh{X},\wh{\mu},\wh{T}_{1},\dots,\wh{T}_{d})$ instead of $(X,\mu,T_{1},\dots,T_{d})$, where $\wh{\mu}$ is the unique ergodic measure on $\wh{X}$.

Let $f_{\e}\in L^{\infty}(\wh{\mu}),\e\in\{0,1\}^{d}$ and fix $\d>0$. Let 
$\wh{f}_{\e}\in L^{\infty}(\wh{\mu}),\e\in\{0,1\}^{d}$ be continuous functions on $X$ such that $\|f_{\e}-\wh{f}_{\e}\|_{1}<\d$ for $\e\in\{0,1\}^{d}$. 
We may assume that all functions are bounded by 1 in the $L^{\infty}$ norm. For simplicity, denote 
$$\E_{N}(\{h_{\e}\}_{\e\leq(1\dots 1)})(x)=\frac{1}{N^{2d}} \sum_{m_{1},\dots,m_{d}=0}^{N-1}\sum_{n_{1},\dots,n_{d}=0}^{N-1}\prod_{\e\in\{0,1\}^{d}}f_{\e}\Bigl(\prod_{i=1}^{d}\wh{T}_{i}^{m_{i}+n_{i}\cdot\e_{i}}x\Bigr)$$
for $x\in \wh{X}, h_{\e}\in L^{\infty}(\wh{\mu})$.
By the telescoping inequality, we have

\begin{align*}
&\quad \Big\vert \mathbb{E}_N(\{f_{\e}\}_{\e\leq(1\dots 1)})(x) - I(\{f_{\e}\}_{\e\leq(1\dots 1)})   \Big\vert   \\
& \leq \Big\vert \mathbb{E}_N(\{f_{\e}\}_{\e\leq(1\dots 1)})(x) - \mathbb{E}_N(\{\wh{f}_{\e}\}_{\e\leq(1\dots 1)})(x) \Big\vert  
 + \Big\vert\mathbb{E}_N(\{\wh{f}_{\e}\}_{\e\leq(1\dots 1)})(x)- I(\{f_{\e}\}_{\e\leq(1\dots 1)}) \Big\vert \\
& \leq \sum_{\e\in\{0,1\}^{d}}\frac{1}{N^{2d}}\sum_{m_{1},\dots,m_{d}=0}^{N-1}\sum_{n_{1},\dots,n_{d}=0}^{N-1}
    \Big\vert f_{\e}(\prod_{i=1}^{d}\wh{T}_{i}^{m_{i}+n_{i}\cdot\e_{i}}x)-\widehat{f}_{\e}(\prod_{i=1}^{d}\wh{T}_{i}^{m_{i}+n_{i}\cdot\e_{i}}x)\Big\vert     \\
& + \Big\vert\mathbb{E}_N(\{\wh{f}_{\e}\}_{\e\leq(1\dots 1)})(x)- I(\{\wh{f}_{\e}\}_{\e\leq(1\dots 1)}) \Big\vert 
+ \Big\vert I(\{f_{\e}\}_{\e\leq(1\dots 1)})- I(\{\wh{f}_{\e}\}_{\e\leq(1\dots 1)})\Big\vert.
\end{align*}
 Since $(\Q_{\wh{T}_{1},\dots,\wh{T}_{d}}({\wh{X}}), \mathcal{G}_{\wh{T}_{1},\dots,\wh{T}_{d}})$ is uniquely ergodic, we have that 
\[\left |\mathbb{E}_N(\{\wh{f}_{\e}\}_{\e\leq(1\dots 1)})(x)- I(\{\wh{f}_{\e}\}_{\e\leq(1\dots 1)}) \right |\] 
converges to 0 for every $x\in \wh{X}$ as $N\to\infty$.
On the other hand, by Birkhoff Ergodic theorem, the first term of the last inequality converge a.e. to 
$$\sum_{\e\in\{0,1\}^{d}}\|f_{\e}-\wh{f}_{\e}\|_1,$$ which is at most $2^{d}\d$.
Finally, by the telescoping inequality and the fact that the marginals of $\mu_{T_{1},\dots,T_{d}}$ are equal to $\mu$,
 we deduce that $$\left |I(\{f_{\e}\}_{\e\in\{0,1\}^{d}})- I(\{\wh{f}_{\e}\}_{\e\leq(1\dots 1)}) \right |\leq \sum_{\e\in\{0,1\}^{d}}\|f_{\e}-\wh{f}_{\e}\|_1\leq 2^{d}\d.$$

Therefore, we can find $N$ large enough and a subset $X_N\subset X$ with measure larger than $1-\d$ such that for every $x\in X_N$,
$$\Big\vert \mathbb{E}_N(\{f_{\e}\}_{\e\in\leq(1\dots 1)})(x) - I(\{f_{\e}\}_{\e\leq(1\dots 1)})  \Big\vert \leq 10\cdot 2^{d}\d. $$

Since $\d$ is arbitrary, we conclude that $\mathbb{E}_N(\{f_{\e}\}_{\e\leq(1\dots 1)})$ converges to $I(\{f_{\e}\}_{\e\leq(1\dots 1)})$ a.e.
as $N\to\infty$.
\end{proof}

Let $(X,\mu,T_{1},\dots,T_{d})$ be a measure preserving system with commuting transformations. 
For all $1\leq k\leq d, \sigma\in D_{k}, f\in L^{\infty}(\mu)$ and $x\in X$, denote
\[S_{\sigma,N}(f,x)\coloneqq \frac{1}{N^{2k}} \sum_{ { 0\leq m_{i}\leq N-1 \atop i\in I_{\sigma}}}\sum_{-m_{i}\leq n_{i}\leq N-1-m_{i}}\prod_{\e\leq \sigma}f\Bigl(\prod_{i\in I_{\sigma}}T_{i}^{m_{i}+n_{i}\cdot\e_{i}}x\Bigr).\]

Specially \[S_{\vec{1},N}(f,x)\coloneqq \frac{1}{N^{2d}} \sum_{ 0\leq m_{i}\leq N-1}\sum_{-m_{i}\leq n_{i}\leq N-1-m_{i}}\prod_{\e\in\{0,1\}^{d}}f\Bigl(\prod_{i=1}^{d}T_{i}^{m_{i}+n_{i}\cdot\e_{i}}x\Bigr)\]

The topological model constructed in the proof of Theorem \ref{dexample} also allows us to show:

\begin{lem}
	Let $(\wh{X},\wh{\mu},\wh{T}_{1},\dots,\wh{T}_{d})$ be the topological model constructed in the proof of Theorem \ref{dexample} and let $f$ be a continuous function on $\wh{X}$.
	Then
	$\lim_{N\to\infty}S_{\vec{1},N}(f,x)$
	converges to $\normm{f}_{\mu,T_{1},\dots,T_{d}}^{2^{d}}$ for all $x\in \wh{X}$.
\end{lem}
\begin{proof}
Suppose that for some $x\in \wh{X}$ the average does not converge to $\normm{f}_{\mu,T_1,\ldots,T_d}^{2^{d}}$. Then there exist  an increasing sequences $N_j\to \infty$ and $\d>0$ such that $S_{\vec{1},N_{j}}(f,x)$
differs from $\normm{f}_{\mu,,T_1,\ldots,T_d}^{2^{d}}$ by at least $\d$ for  every $j\in\N$. 

Let $\lambda$ be any weak$^{\ast}$-limit of the sequence
\begin{equation}\label{eq:limitmeasure}
\lambda_{j}:=\frac{1}{N_{j}^{2d}} \sum_{0\leq m_{i}\leq N_{j}-1}\sum_{-m_{i}\leq n_{i}\leq N_{j}-1-m_{i}}\bigotimes_{\e\in\{0,1\}^{d}}\Bigl(\prod_{i=1}^{d}\wh{T}_{i}^{m_{i}+n_{i}\cdot\e_{i}}\Bigr)\d_{x^{[d]}}..
\end{equation}

Let $F$ be a continuous function on $\Q_{\wh{T}_1,\ldots,\wh{T}_d}(\wh{X})$. We claim that $\int F\circ \mathcal{F}_{k} \lambda- \int F\lambda =0$. 
Note that the difference between { \[ \sum_{-m_{i}\leq n_{i}\leq N_{j}-1-m_{i}}F\Bigl(\Bigl(\prod_{i=1}^{d}\wh{T}_{i}^{m_{i}+n_{i}\cdot\e_{i}}x\Bigr)_{\e\in V_d}\Bigr) =\sum_{-m_{k}\leq n_{k}\leq N_{k}-1-m_{k}} \sum_{-m_{i}\leq n_{i}\leq N_{j}-1-m_{i}\atop i\neq k}F\Bigl(\Bigl(\prod_{i=1}^{d}\wh{T}_{i}^{m_{i}+n_{i}\cdot\e_{i}}x\Bigr)_{\e\in V_d}\Bigr)\] }and 

\begin{align*} & \sum_{-m_{i}\leq n_{i}\leq N_{j}-1-m_{i}}F\Bigl(\Bigl(\prod_{i=1\atop i\neq k}^{d}\wh{T}_{i}^{m_{i}+n_{i}\cdot\e_{i}}\wh{T}_{k}^{(n_k+1)\cdot \e_i }x\Bigr)_{\e\in V_d}\Bigr)\\
= & \sum_{-m_{k}+1\leq n_{k}\leq N_{k}-m_{k}} \sum_{-m_{i}\leq n_{i}\leq N_{j}-1-m_{i}\atop i\neq k}F\Bigl(\Bigl(\prod_{i=1}^{d}\wh{T}_{i}^{m_{i}+n_{i}\cdot\e_{i}}x\Bigr)_{\e\in V_d}\Bigr)
\end{align*}

is less than $2N_j^{d-1}\|F\|_{\infty}$. 
So
\[ \left \vert \int F d\lambda_j - \int F\circ \mathcal{F}_{k}d\lambda_j \right \vert \leq  \frac{2\|F\|_{\infty}}{N_j}, \]
which implies that   $\int F\circ \mathcal{F}_{k} d\lambda= \int Fd\lambda$ as claimed. 

A similar but tedious computation shows that $\int F\circ \wh{T}_i^{[d]} d\lambda= \int Fd\lambda $ for all continuous function $F$ on $\Q_{\wh{T}_1,\ldots,\wh{T}_d}(\wh{X})$ and $i=1\ldots, d$. We omit the detail. 

Thus any weak limit of the sequence \eqref{eq:limitmeasure} is invariant under $\mathcal{G}_{\wh{T}_{1},\dots,\wh{T}_{k}}$ and therefore it equals to $\mu_{T_1,\dots,T_{{k}}}$ by unique ergodicity. This means that \eqref{eq:limitmeasure} converges to $\mu_{T_{{1}},\dots,T_{{k}}}$ and hence \[\frac{1}{N_{j}^{2d}} \sum_{ 0\leq m_{i}\leq N_{j}-1}\sum_{-m_{i}\leq n_{i}\leq N_{j}-1-m_{i}}\prod_{\e\in\{0,1\}^{d}}f\Bigl(\prod_{i=1}^{d}T_{i}^{m_{i}+n_{i}\cdot\e_{i}}x\Bigr)\]
converges to $\int \otimes f d\mu_{T_1,\ldots,T_{d}}=\normm{f}_{\mu,T_1,\ldots,T_d}^{2^{d}}$ as $j\to\infty$, a contradiction.
\end{proof}

By a density argument, we have 

\begin{lem}\label{ConvergenceContinuous}
	Let $(X,\mu,T_{1},\dots,T_{d})$ be an ergodic $Z$-sated measure preserving system with commuting transformations and $f\in L^{\infty}(\mu)$. 
	Then
	$\lim_{N\to\infty}S_{\vec{1},N}(f,x)$
	converges to $\normm{f}_{\mu,T_{1},\dots,T_{d}}^{2^{d}}$ for  $\mu$-a.e. $x\in X$.
\end{lem}

We deduce

\begin{prop} \label{prop:convergenceseminorm}
Let $(X,\mu,T_{1},\dots,T_{d})$ be an ergodic measure preserving system with commuting transformations and $f\in L^{\infty}(\mu)$. 
	Then
	$\lim_{N\to\infty}S_{\vec{1},N}(f,x)$
	converges to $\normm{f}_{\mu,T_{1},\dots,T_{d}}^{2^{d}}$ for  $\mu$-a.e. $x\in X$.
\end{prop}

\begin{proof}
We can find an $Z$-sated ergodic and free extension $(X',\mu',T_{1}',\dots,T_{k}')$ of $(X,\mu,T_1,\ldots,T_d)$ with a factor map $\pi$. By Lemma \ref{ConvergenceContinuous}, $\lim_{N\to\infty}S_{\vec{1},N}(f\circ \pi,x')$ converges to $\normm{f\circ \pi}_{\mu',T_{1}',\dots,T_{d}'}^{2^{d}}$ for  $\mu'$-a.e. $x'\in X$. Since $\normm{f\circ \pi}_{\mu',T_{1}',\dots,T_{d}'}^{2^{d}}=\normm{f}_{\mu,T_1,\ldots,T_d}^{2^d}$ we get the result.
\end{proof}

We also need the following Van der Corput type estimate:

\begin{lem} \label{BoundAverage}
  Let $(X,\mu,T_{1},\dots,T_{d})$ be a measure preserving system with commuting transformations and $f_{\e}\in L^{\infty}(\mu), \e\in\{0,1\}^{d}$ be functions with $\Vert f_{\e}\Vert_{\infty}\leq 1$. Then for all $N\in\N$, $1\leq k\leq d$, $\sigma\in D_{k}$ and $x\in X$, we have that
  $$\Bigl(\frac{1}{N^{d}} \sum_{n_{1},\dots,n_{d}=0}^{N-1}\prod_{\e\in E_{k}}f_{\e}\Bigl(\prod_{i=1}^{d}T_{i}^{n_{i}\cdot\e_{i}}x\Bigr)\Bigr)^{2^{k}}\leq  S_{\sigma,N}(f_{\sigma},x).$$
  Particularly, $S_{\sigma,N}(f_{\sigma},x)\geq 0$.
\end{lem}
\begin{proof}
	We first prove the case $\sigma=(1\dots 1)$. In other words, we show
	 \begin{equation}\label{tem3}
	 \begin{split}
	\Bigl(\frac{1}{N^{d}} \sum_{n_{1},\dots,n_{d}=0}^{N-1}\prod_{\e\in \{0,1\}^{d}}f_{\e}\Bigl(\prod_{i=1}^{d}T_{i}^{n_{i}\cdot\e_{i}}x\Bigr)\Bigr)^{2^{d}}\leq S_{(1\dots 1),N}(f_{(1\dots 1)},x).
	\end{split}
	 \end{equation}

	Fix $f_{\e}\in L^{\infty}(\mu), \e\in\{0,1\}^{d}$ with $\Vert f_{\e}\Vert_{\infty}\leq 1$. For $0\leq k\leq d$, denote
    \begin{equation}\nonumber
    \begin{split}
	& A_{k}=\Bigl(\frac{1}{N^{d-k}}\sum_{n_{1},\dots,n_{d-k}=0}^{N-1}\frac{1}{N^{2k}} \sum_{m_{d-k+1},\dots,m_{d}=0}^{N-1}\sum_{-m_{i}\leq n_{i}\leq N-1-m_{i}, d-k+1\leq i\leq  d}
	\\&\qquad
	\prod_{\e\in\{0,1\}^{d-k}}\prod_{\e'\in\{0,1\}^{k}}f^{\vert\e'\vert}_{\e 1\dots 1}\Bigl(\prod_{i=1}^{d-k}T_{i}^{n_{i}\cdot\e_{i}}\prod_{i=d-k+1}^{d}T_{i}^{m_{i}+n_{i}\cdot\e'_{i}}x\Bigr)\Bigr)^{2^{d-k}}.
	 \end{split}
\end{equation}
  It suffices to show that $A_{0}\leq A_{d}$. It then suffices to show that $A_{k}\leq A_{k+1}$ for all $0\leq k<d$.
  
  Fix $0\leq k<d$.	
  We separate all functions $f_{\e 1\dots 1}, \e\in\{0,1\}^{k}$ into two class: the first class consists of all $f_{\e 1\dots 1}$ with $\e_{k}=0$, and the second consists of all $f_{\e 1\dots 1}$ with $\e_{k}=1$. Since all $f_{\e 1\dots 1}$ in the first class are bounded, we may use  Cauchy-Schwartz inequality to drop all the functions in the first class. We have
  \begin{equation}\nonumber
  \begin{split}
  & A_{k}\leq \Bigl(\frac{1}{N^{d-k-1}}\sum_{n_{1},\dots,n_{d-k-1}=0}^{N-1}\frac{1}{N^{2k}} \sum_{m_{d-k+1},\dots,m_{d}=0}^{N-1}\sum_{-m_{i}\leq n_{i}\leq N-1-m_{i}, d-k+1\leq i\leq  d}
  \\&\qquad
  \Bigl(\frac{1}{N}\sum_{n_{d-k}=0}^{N-1}\prod_{\e\in\{0,1\}^{d-k-1}}\prod_{\e'\in\{0,1\}^{k}}f_{\e 1\dots 1}\Bigl(\prod_{i=1}^{d-k-1}T_{i}^{n_{i}\cdot\e_{i}}\prod_{i=d-k+1}^{d}T_{i}^{m_{i}+n_{i}\cdot\e'_{i}}(T_{d-k}^{n_{d-k}}x)\Bigr)\Bigr)^{2}\Bigr)^{2^{d-k-1}}
  \end{split}
  \end{equation}
Expanding the square on the right hand side and reparametrizing the indices, it is easy to see that the right hand side is exactly $A_{k+1}$, which finishes the proof for the case $\sigma=(1\dots 1)$.

\

Now we prove the general case. Let $1\leq k\leq d$ and $\sigma\in E_{k}$.
To ease the notation, we assume that $\sigma=(1\dots 10\dots0)$ (the general case can be proved in a similar way). Note that
  \begin{equation}\nonumber
  \begin{split}
  & \quad\frac{1}{N^{d}} \sum_{n_{1},\dots,n_{d}=0}^{N-1}\prod_{\e\in E_{k}}f_{\e}\Bigl(\prod_{i=1}^{d}T_{i}^{n_{i}\cdot\e_{i}}x\Bigr)
  \\&
  =\frac{1}{N^{d-k}}\sum_{n_{k+1},\dots,n_{d}=0}^{N-1}
  \Bigl(\frac{1}{N^{k}}\sum_{n_{1},\dots,n_{k}=0}^{N-1}\prod_{\e\leq\sigma}g_{n_{k+1},\dots,n_{d},\e}(\prod_{i=1}^{k}T_{i}^{n_{i}\cdot\e_{i}}x)\Bigr),
  \end{split}
  \end{equation}
where $$g_{n_{k+1},\dots,n_{d},\e}(x)=\prod_{\e'\in E_{k},\e'\cap\sigma=\e}
f_{\e'}\Bigl(\prod_{i=k+1}^{d}T_{i}^{n_{i}\cdot\e'_{i}}x\Bigr).$$
Note that $g_{n_{k+1},\dots,n_{d},\sigma}(x)=f_{\sigma}$. Replacing $d$ by $k$ in (\ref{tem3}), we have that
  \begin{equation}\nonumber
  \begin{split}
\quad\frac{1}{N^{d}} \sum_{n_{1},\dots,n_{d}=0}^{N-1}\prod_{\e\in E_{k}}f_{\e}\Bigl(\prod_{i=1}^{d}T_{i}^{n_{i}\cdot\e_{i}}x\Bigr)
  \leq\frac{1}{N^{d-k}}\sum_{n_{k+1},\dots,n_{d}=0}^{N-1}
  \vert S_{\sigma,N}(f_{\sigma},x)\vert^{\frac{1}{2^{k}}}=\vert S_{\sigma,N}(f_{\sigma},x)\vert^{\frac{1}{2^{k}}},
  \end{split}
  \end{equation}
which finishes the proof.
\end{proof}

The following lemma shows that the cubic pointwise average result passes through the $L^{1}(\mu)$ limit:
\begin{lem}\label{fd_appro}
	Let $(X,\mu,T_{1},\dots,T_{d})$ be a measure preserving system with commuting transformations and $f_{\e}\in L^{\infty}(\mu), \e\in\{0,1\}^{d}$. Let $\sigma\in\{0,1\}^{d}$ and $(g_{n,\sigma})_{n\in\N}$ be a sequence of $L^{\infty}(\mu)$ functions such that $\Vert f_{\sigma}-g_{n,\sigma}\Vert_{1}\to 0$ as $n\to\infty$. If for all $n\in\N$, the average
	\begin{equation}\nonumber
	\begin{split}
	\frac{1}{N^{d}} \sum_{n_{1},\dots,n_{d}=0}^{N-1}g_{n,\sigma}(\prod_{i=1}^{d}T_{i}^{n_{i}\cdot\sigma}x)\prod_{\sigma\neq\e\in\{0,1\}^{d}}f_{\e}\Bigl(\prod_{i=1}^{d}T_{i}^{n_{i}\cdot\e_{i}}x\Bigr)
	\end{split}
	\end{equation}
	converges for $\mu$-a.e $x\in X$ as $N\to\infty$. Then the average
		\begin{equation}\nonumber
		\begin{split}
		\frac{1}{N^{d}} \sum_{n_{1},\dots,n_{d}=0}^{N-1}\prod_{\e\in\{0,1\}^{d}}f_{\e}\Bigl(\prod_{i=1}^{d}T_{i}^{n_{i}\cdot\e_{i}}x\Bigr)
		\end{split}
		\end{equation}
		also converges for $\mu$-a.e $x\in X$ as $N\to\infty$.
\end{lem}	
\begin{proof}
	We may assume without loss of generality that 
	$\Vert f_{\sigma}-g_{n,\sigma}\Vert_{1}\leq \frac{1}{2^{n}}$ for all $n\in\N$.
Denote
$$S_{N}(x)=\frac{1}{N^{d}} \sum_{n_{1},\dots,n_{d}=0}^{N-1}\prod_{\e\in\{0,1\}^{d}}f_{\e}\Bigl(\prod_{i=1}^{d}T_{i}^{n_{i}\cdot\e_{i}}x\Bigr)$$
and
$$S'_{n,N}(x)=\frac{1}{N^{d}} \sum_{n_{1},\dots,n_{d}=0}^{N-1}g_{n,\sigma}(\prod_{i=1}^{d}T_{i}^{n_{i}\cdot\sigma}x)\prod_{\sigma\neq\e\in\{0,1\}^{d}}f_{\e}\Bigl(\prod_{i=1}^{d}T_{i}^{n_{i}\cdot\e_{i}}x\Bigr).$$
Suppose that $I_{\sigma}=\{a_{1},\dots,a_{k}\}$.
By the telescoping theorem, Birkhoff Theorem and the assumption, there exists a set $A$ with $\mu(A)=1$ such that for all $x\in A$
\begin{equation}\label{Bound}
\begin{split}
\limsup_{N\to\infty}\Bigl\vert S_{N}(x)-S'_{n,N}(x)
\Bigr\vert
\leq \E\Bigl(\vert f_{\sigma}-g_{n,\sigma}\vert \big\vert \I_{T_{a_{1}},\dots,T_{a_{k}}}\Bigr)(x)
\end{split}
\end{equation}
and $\lim_{N\to\infty}S'_{n,N}(x)$
exists
for all $x\in A, n\in\N$. Let $$A_{n}=\Bigl\{x\in X\colon\E\Bigl(\vert f_{\sigma}-g_{n,\sigma}\vert \big\vert \I_{T_{a_{1}},\dots,T_{a_{k}}}\Bigr)(x)>\frac{1}{n}\Bigr\}.$$
By Markov inequality, we have that 
$$\mu(A_{n})\leq n\Vert f_{\sigma}-g_{n,\sigma}\Vert_{1}\leq\frac{n}{2^{n}}.$$
By Borel-Cantelli Lemma, $\mu(\limsup_{n}A_{n})=0.$ Let
$B=A\cap(\limsup_{n}A_{n})^{c}$. Then $\mu(B)=1$. 

Fix $\d>0$ and $x\in B$. Since $x\notin\limsup_{n}A_{n}$, there exists
$n\geq\frac{1}{\d}$ such that $x\notin A_{n}$.
By (\ref{Bound}), there exists $N_{0}\in\N$ such that
$$\Bigl\vert S_{N}(x)-S'_{n,N}(x)\Bigr\vert\leq 2\E\Bigl(\vert f_{\sigma}-g_{n,\sigma}\vert \big\vert \I_{\e}\Bigr)(x)$$
for all $N\geq N_{0}$. 
Since $\lim_{N\to\infty}S'_{n,N}(x)$ exists, there exists $N_{1}>N_{0}$ such that
$$\Bigl\vert S'_{n,M}(x)-S'_{n,N}(x)\Bigr\vert<\d$$
for all $M,N\leq N_{1}$.
Let $M,N\geq N_{1}$, then
\begin{equation}\nonumber
\begin{split}&
\quad\Bigl\vert S_{M}(x)-S_{N}(x)\Bigr\vert
\\&\leq\Bigl\vert S_{M}(x)-S'_{n,M}(x)\Bigr\vert+\Bigl\vert S'_{n,M}(x)-S'_{n,N}(x)\Bigr\vert+\Bigl\vert S_{N}(x)-S'_{n,N}(x)\Bigr\vert
\\& \leq 4\E\Bigl(\vert f_{\sigma}-g_{n,\sigma}\vert \big\vert \I_{T_{a_{1}},\dots,T_{a_{k}}}\Bigr)(x)+\d\leq \frac{4}{n}+\d<5\d.
\end{split}
\end{equation}
Since $\d$ is arbitrary, we deduce that $\{S_{N}(x)\}_{N\in\N}$ is a Cauchy sequence for all $x\in B$ and so $\lim_{N\to\infty}S_{N}(x)$ exists for $\mu$-a.e. $x\in X$.	
\end{proof}	

Now we are able to prove the Theorem \ref{THM:cubicpointwise}. The difference between Theorems \ref{THM:cubicpointwise} and \ref{dexample} is that there is an additional double average in the latter theorem. Our strategy is to use Lemma \ref{ConvergenceContinuous}, \ref{BoundAverage} and \ref{fd_appro} to drop this double average to obtain Theorem \ref{THM:cubicpointwise}.

\begin{thm*}
	Let $(X,\mu,T_{1},\dots,T_{d})$ be a measure preserving system with commuting transformations and $f_{\e}\in L^{\infty}(\mu), \e\in\{0,1\}^{d},\e\neq(0,\dots,0)$. Then the average
	\begin{equation}\nonumber
	    \begin{split}
	    \frac{1}{N^{d}} \sum_{n_{1},\dots,n_{d}=0}^{N-1}\prod_{(0,\dots,0)\neq\e\in\{0,1\}^{d}}f_{\e}\Bigl(\prod_{i=1}^{d}T_{i}^{n_{i}\cdot\e_{i}}x\Bigr)
	    \end{split}
	\end{equation}
converges for $\mu$-a.e $x\in X$ as $N\to\infty$.

\end{thm*}

\begin{proof} We may assume without loss of generality that all the functions are bounded by 1 in the $L^{\infty}(\mu)$ norm. 

We say that \emph{property-$k$} holds if 
the average
	\begin{equation}\nonumber
	\begin{split}
	\frac{1}{N^{d}} \sum_{n_{1},\dots,n_{d}=0}^{N-1}\prod_{\e\in E_{k}}f_{\e}\Bigl(\prod_{i=1}^{d}T_{i}^{n_{i}\cdot\e_{i}}x\Bigr)
	\end{split}
	\end{equation}
converges for $\mu$-a.e $x\in X$ as $N\to\infty$ for all $f_{\e}\in L^{\infty}(\mu),\e\in E_{k}$. It suffices to show that property-d holds (and then take $f_{(0\dots 0)}\equiv 1$). 

Note that
	\begin{equation}\nonumber
	\begin{split}
	\frac{1}{N^{d}} \sum_{n_{1},\dots,n_{d}=0}^{N-1}\prod_{\e\in E_{1}}f_{\e}\Bigl(\prod_{i=1}^{d}T_{i}^{n_{i}\cdot\e_{i}}x\Bigr)
	=f_{0\dots 0}(x)\cdot\prod_{i=1}^{d}\Bigl(\frac{1}{N}\sum_{n=0}^{N-1}f_{i0\dots 0}(T^{n}_{i}x)\Bigr).
	\end{split}
	\end{equation}	
So property-1 holds by Birkhoff Ergodic Theorem. Suppose that property-$k$ holds for some $1\leq k<d$, we claim that property-$(k+1)$ also holds.
	
Suppose first that for all $\e\in D_{k+1}$,  $f_{\e}$ is an \emph{$I_{\e}$-product}, meaning that $f_{\e}=\prod_{j\in I_{\e}}h_{j,\e}$ for some $h_{j,\e}$ measurable with respect to $\I_{T_{j}}$.
Note that for  $\e\in D_{k+1}$, we have $$h_{j,\e}\Bigl(\prod_{i=1}^{d}T_{i}^{n_{i}\cdot\e_{i}}x\Bigr)=h_{j,\e}\Bigl(\prod_{1\leq i\leq d, i\neq j}T_{i}^{n_{i}\cdot\e_{i}}x\Bigr)=h_{j,\e^{\hat{j}}}\Bigl(\prod_{i=1}^{d}T_{i}^{n_{i}\cdot\e_{i}}x\Bigr),$$
where $\e^{\hat{j}}\in D_{k}$ is defined by $(\e^{\hat{j}})_{i}=\e_{i}$ for $i\neq j$ and $(\e^{\hat{j}})_{j}=0$. This implies that by combining similar terms together, we have that
	\begin{equation}\label{tem1}
	\begin{split}
	\frac{1}{N^{d}} \sum_{n_{1},\dots,n_{d}=0}^{N-1}\prod_{\e\in E_{k+1}}f_{\e}\Bigl(\prod_{i=1}^{d}T_{i}^{n_{i}\cdot\e_{i}}x\Bigr)
	=\frac{1}{N^{d}} \sum_{n_{1},\dots,n_{d}=0}^{N-1}\prod_{\e\in E_{k}}g_{\e}\Bigl(\prod_{i=1}^{d}T_{i}^{n_{i}\cdot\e_{i}}x\Bigr).
	\end{split}
	\end{equation}	
for some $g_{\e}$ bounded by 1. By induction hypothesis, (\ref{tem1}) converges for $\mu$-a.e. $x\in X$.

We now suppose that for all $\e\in D_{k+1}$, $f_{\e}$ is a finite sum of $I_{\e}$-products. It is easy to see that property-$(k+1)$ holds in this case.

We then suppose that for all $\e\in D_{k+1}$, $f_{\e}$ is measurable with respect to $\mathcal{Z}_{I_{\e}}=\vee_{i\in I_{\e}}\I_{T_{i}}$. Thus each $f_{\e}$ can be approximated in $L^{1}(\mu)$ by a sequence of functions which are finite sums of $I_{\e}$-products. So property-$(k+1)$ holds in this case by Lemma \ref{fd_appro}. 

It now suffices to prove property-$(k+1)$ under the assumption that $\mathbb{E}(f_{\sigma}|\mathcal{Z}_{I_{\sigma}})=0$ for some $\sigma\in D_{k+1}$. In fact we show that the average goes to 0 for a.e. $x\in X$. Since $X$ is magic for $I_{\sigma}$, we have that $\normm{f_{\sigma}}_{\mu,\sigma}=0$. 

Let $\mu=\int_{X}\mu_{x}d\mu(x)$ be the ergodic decomposition of $\mu$ under $T_{a_{1}},\dots,T_{a_{k+1}}$, where $I_{\sigma}=\{a_{1},\dots,a_{k+1}\}$. It suffices to show that average goes to 0 $\mu_x$-a.e. for $\mu$-a.e. $x\in X$. 
Since $0=\normm{f_{\sigma}}_{\mu,\sigma}=\int\normm{f_{\sigma}}_{\mu_{x},\sigma}d\mu(x) $ we have that for $\mu$-a.e. $x\in X$, $\normm{f_{\sigma}}_{\mu_x,\sigma}=0$.

Applying Proposition \ref{prop:convergenceseminorm} to $(X,\mu_x,T_{a_{1}},\dots,T_{a_{k+1}})$, we have
\begin{equation}\nonumber
\begin{split}
\lim_{N\to\infty}S_{\sigma,N}(f_{\sigma},y)=\normm{f_{\sigma}}^{k+1}_{\mu_{x},\sigma}
\end{split}
\end{equation}
for $\mu_{x}$-a.e. $y\in X$. 
By Lemma \ref{BoundAverage}, we have
	\begin{equation}\nonumber
	\begin{split}
	0\leq 	\Bigl(\frac{1}{N^{d}} \sum_{n_{1},\dots,n_{d}=0}^{N-1}\prod_{\e\in E_{k+1}}f_{\e}\Bigl(\prod_{i=1}^{d}T_{i}^{n_{i}\cdot\e_{i}}y\Bigr)\Bigr)^{2^{k+1}}\leq  S_{\sigma,N}(f_{\sigma},y)
	\end{split}
	\end{equation}
	for all $y\in X$.
So for $\mu_x$-a.e. $y\in X$, $$\lim_{N\to \infty}\frac{1}{N^{d}} \sum_{n_{1},\dots,n_{d}=0}^{N-1}\prod_{\e\in E_{k+1}}f_{\e}\Bigl(\prod_{i=1}^{d}T_{i}^{n_{i}\cdot\e_{i}}y\Bigr)=0,$$
which implies that $$\lim_{N\to \infty}\frac{1}{N^{d}} \sum_{n_{1},\dots,n_{d}=0}^{N-1}\prod_{\e\in E_{k+1}}f_{\e}\Bigl(\prod_{i=1}^{d}T_{i}^{n_{i}\cdot\e_{i}}x\Bigr)=0$$ for $\mu$-a.e. $x\in X$.	
This finishes the proof.
\end{proof}

\section{Pointwise convergence of multiple averages} \label{Sec:MultiplesAverages}
This section is devoted to proving Theorem \ref{THM:averagedmultiple}.

\subsection{The Furstenberg-Ryzhikov self-joining}\label{Sec:Fj}
Let $(X,\mu,T_1,\ldots,T_d)$ be a system with commuting transformations. The Furstenberg-Ryzhikov self-joining $\mu^F$ associated to $(T_1,\ldots,T_d)$ is the measure on $X^{d}$ defined by 

\[\mu^F(f_1\otimes f_2\cdots \otimes f_d) =\lim_{N\to \infty} \frac{1}{N} \sum_{n=0}^{N-1} \int_{X} f_1(T_1^n x)\cdots f_d(T_d^n x) d\mu(x).   \]  
The convergence of the right hand follows from the $L^2$-convergence of multiple averages for commuting transformations \cite{Tao}  so $\mu^{F}$ is well defined. In \cite{Ryzhikov}, Ryzhikov linked for the first time (as far as we know) the convergence of multiple averages with the collection of  self joinings of a dynamical system and we thank J. Paul Thouvenot for pointing us this reference. 

The measure $\mu^F$ is invariant under the diagonal transformations $T_i\times \cdots \times T_i$ $i=1,\ldots,d$ and under $T_1\times T_2\cdots \times T_d$. We let $H_d$ denote the group generated by all these transformations and we think of it as a $\Z^{d+1}$ action on $(X^d,\mu^F)$.

We write $X_i$ to denote the $i$-th copy of $X$ inside $X^d$ under $\mu^F$. So we think of $(X^d,\mu^F)$ as $(X_1\times\cdots \times X_d,\mu^F)$.

\begin{rem}\label{rem:proj}
The projection of the Furstenberg-Ryzhikov self-joining of $(T_1,\ldots,T_d)$ onto the last $d-1$ coordinates is the Furstenberg-Ryzhikov self-joining of $(T_1^{-1}T_2, T_1^{-1}T_3,$ $\ldots,T_1^{-1}T_d)$. 
\end{rem}

Let $(X,\mu,T_1,\ldots,T_d)$ be a measure preserving system. For $\epsilon \subseteq [d]$, $\# \epsilon \geq 2$,  let $\Phi_{\epsilon}(X)$ denote the factor of $X$ associated to the $\sigma$-algebra invariant under all the transformations $T_i^{-1}T_j$, for $i,j \in \epsilon$, $i\neq j$. Note that in $\Phi_{\epsilon}(X)$, all transformations $T_i$, $i\in \epsilon$ are equal. Thus, we can regard $\Phi_{\epsilon}(X)$ as a $\mathbb{Z}^{d-\# \epsilon +1}$-action (by taking one candidate from $T_i$, $i\in \epsilon$ and removing the rest). When there is no confusion, we write $\Phi_{\epsilon}$ instead of $\Phi_{\epsilon}(X)$ for short. 
The importance about these factors is that they allow us to the study the Furstenberg-Ryzhikov self-joining of a ``sated enough system''. For $\epsilon \subseteq [d]$, $\# \epsilon \geq 2$ we denote $Z_{\Phi}^{\epsilon}=Z^{\Lambda}$, where $\Lambda$ is the subgroup spanned by $T_{i}^{-1}T_j, i,j\in\e$. In other words, $Z_{\Phi}^{\epsilon}$ is the idempotent class where the transformations $T_i$, $i\in\epsilon$ are all equal. 

\begin{defn}\label{zast}
Let $(X,\mu,T_1,\ldots,T_d)$ be a measure preserving system. We say that it is {\it $\mathcal{Z}^{\ast}$-sated} if it is sated for the class
\[  \mathcal{Z}_{\epsilon}:= \bigvee_{j\in [d]\setminus \epsilon} Z_{\Phi}^{\epsilon \cup \{j\}} \] 
for every $\epsilon \subseteq [d]$,  $\# \epsilon \geq 2$. 
\end{defn}

This satedness assumption gives a nice picture of what happens in the Furstenberg-Ryzhikov self-joining. It implies for instance that the systems $X_1,\ldots,X_d$ under $\mu^F$ are relative independent over their factors \[\bigvee_{j\neq 1} \Phi_{\{1,j\}}, \quad \bigvee_{j\neq 2} \Phi_{\{2,j\}}, \quad \ldots, \bigvee_{j\neq d} \Phi_{\{d,j\}}\]

We refer to \cite{A}, Chapter 4 for details. We use Austin's diagram of the Furstenberg-Ryzhikov self-joining (\cite{A}, Chapter 5) to have a clear picture of it (and we build a topological model for it in what follows):

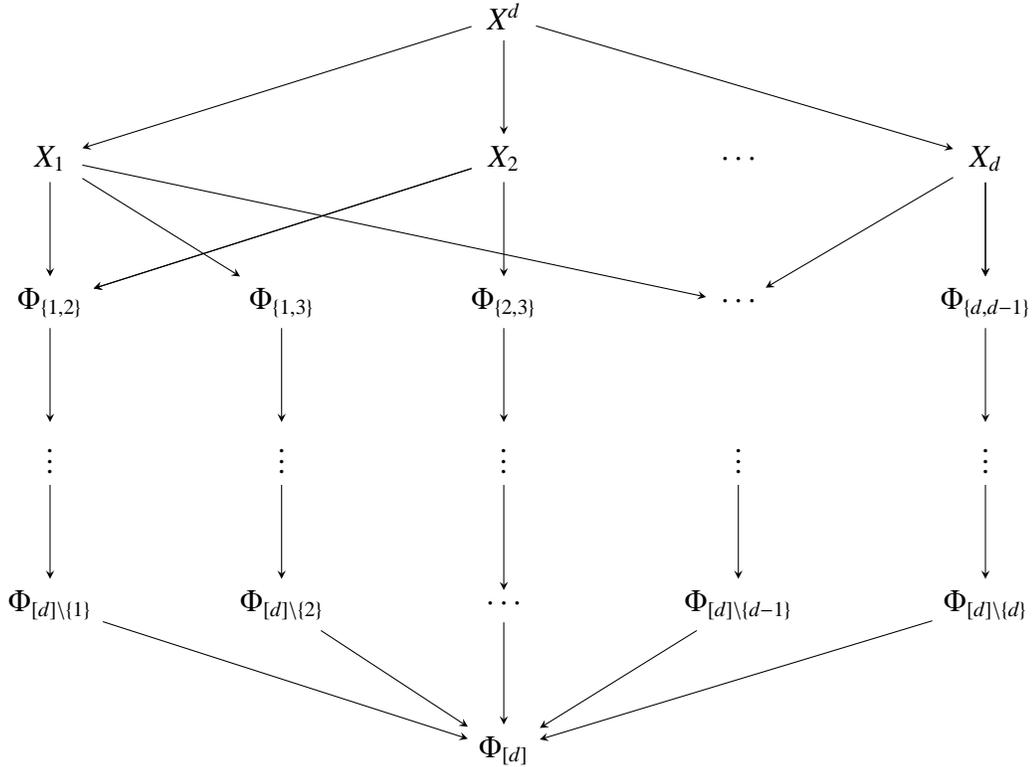
\begin{figure}[H]
 \begin{tikzpicture}
  \matrix (m) [matrix of math nodes,row sep=3em,column sep=4em,minimum width=2em,ampersand replacement=\&]
  {
      \&  \&  X^d  \& \&  \\
     X_1 \& \& X_2 \& \cdots \& X_d \\
    \Phi_{\{1,2\}}  \& \Phi_{\{ 1,3\}} \& \Phi_{ \{ 2,3\}}  \&  \cdots \& \Phi_{\{d,d-1\}}\\  
    \vdots \& \vdots \& \vdots \& \vdots \& \vdots  \\
    \Phi_{[d]\setminus \{1\} } \& \Phi_{[d]\setminus \{2\} } \& \cdots \&\Phi_{[d]\setminus \{d-1\} } \& \Phi_{[d]\setminus \{d\} }   \\  
     \&  \& \Phi_{[d]}  \& \& \\ };
  \path[-stealth]
     (m-1-3) edge  (m-2-1) 
     (m-1-3) edge  (m-2-3)
     (m-1-3) edge  (m-2-5)
       (m-2-1) edge  (m-3-1)  
       (m-2-1) edge  (m-3-2)
       (m-2-1) edge  (m-3-4)       
       
       (m-2-5) edge  (m-3-5)
       (m-2-5) edge  (m-3-4)
       
       (m-2-3) edge  (m-3-3) 
       (m-2-3) edge  (m-3-1)
       (m-2-5) edge  (m-3-5)
       (m-2-3) edge  (m-3-1)        
       
       (m-3-3) edge  (m-4-3)         
       (m-3-5) edge  (m-4-5)
       (m-3-1) edge  (m-4-1)
       (m-3-2) edge  (m-4-2)

       (m-4-1) edge  (m-5-1)        
       (m-4-2) edge  (m-5-2) 
       (m-4-3) edge  (m-5-3)
       (m-4-4) edge  (m-5-4)
       (m-4-5) edge  (m-5-5)   
       
       (m-5-1) edge  (m-6-3)   
       (m-5-2) edge  (m-6-3)
       (m-5-5) edge  (m-6-3)
       (m-5-4) edge  (m-6-3)
       (m-5-3) edge  (m-6-3)
       ;
\end{tikzpicture}
\caption{Austin's diagram of the Furstenberg-Ryzhikov self-joining and its factors. We omit writing the measure for convenience.}
 \end{figure}

\subsection{Construction of a model for the Furstenberg-Ryzhikov self-joining}\label{Sec:Fj2}

In this section, we build a topological model for the Furstenberg-Ryzhikov self joining of (an extension) of an ergodic system. We show that this implies Theorem \ref{THM:averagedmultiple}. We mainly follow ideas from Austin's works, which are nicely exposed in \cite{A}. 
\begin{rem}
We can assume from the beginning that $\Phi_{\epsilon}(X)$ is a free $\mathbb{Z}^{d-\# \epsilon +1}$-action. To do this, since we can freely pass to extensions, we may replace our system $X$ by (an arbitrary) ergodic joining $Y$ of $X$ with a free $\mathbb{Z}^{d-\# \epsilon +1}$-system. The $\Phi_{\epsilon}(Y)$ factor of $Y$ is an extension of the free $\mathbb{Z}^{d-\# \epsilon +1}$- system and thus the $\mathbb{Z}^{d-\# \epsilon +1}$ respective action is also free in $\Phi_{\epsilon}(Y)$. We include the freeness assumption since it is needed to build relative strictly ergodic models and we assume it in all of what follows.
\end{rem}

\begin{lem}\label{Lem:productfactor}
Let $(X,\mu,T_1,T_2)$ be an ergodic free $\Z^2$ action. Suppose its $T_1$-invariant $\sigma$-algebra and its $T_2$ invariant $\sigma$-algebras define factor maps free under the action of $T_{2}$ and $T_1$ respectively. Then, there exist strictly ergodic topological models for the factors $X\to \I_{T_{i}}(X)$, $i=1,2$. 
\end{lem}
\begin{proof}
By Theorem \ref{Weiss}, we may find strictly ergodic topological $(\widehat{X}_{\{2\}},T_1)$ and  $(\widehat{X}_{\{1\}},T_2)$ for the factor maps $(\I_{T_{2}},T_1)$ and $(\I_{T_{1}},T_2)$. Their product $(\widehat{X}_\{{2}\} \times \widehat{X}_{\{1\}},T_1\times\id,\id\times T_2)$ is naturally a strictly ergodic model for $(\I_{T_{1}}\vee \I_{T_{2}},T_1,T_2)$. By Theorem \ref{Weiss}, we can find a strictly ergodic model $\widehat{X}\to \widehat{X}_{\{2\}}\times \widehat{X}_{\{1\}}$  for $X\to \I_{T_{2}}\vee\I_{T_{1}}$. The projection onto each coordinate in $\widehat{X}_{\{2\}}\times \widehat{X}_{\{1\}}$ define the strictly ergodic topological models for $X\to \I_{T_{1}}(X)$ and $X\to \I_{T_{2}}(X)$. 
\end{proof}

\begin{thm} \label{Thm:ConstructionModel}
There exist a topological model for the diagram.
\end{thm}
\begin{proof}
The proof is similar to the one of Proposition \ref{Prop:ConstructionModel} and it goes by building strictly ergodic models on each level of the diagram inductively.
We build first the topological models for the levels $d$ and $d-1$ (which are special cases) and then we give an induction argument to build all the other levels. After building all the diagram, we check that it is indeed a uniquely ergodic model for the measurable diagram. 

{\bf Level $d$:} 
For the $d$-th level, $\Phi_{[d]}$ is just a free $\Z$-action (all transformations are the same) and we may use Theorem \ref{Weiss} to get a strictly ergodic model for it.

{\bf Level $d-1$:} 
Take $\# \epsilon=d-1$ and write $\epsilon=[d]\setminus \{j_0\}$ for $j_0\in [d]$. $\Phi_{\epsilon}$ can be viewed as a $\Z^2$ action generated by $T_{i_0}$ and $T_{j_0}$, where $i_0$ is any in $\epsilon$ (recall that all $T_i$, $i\in \epsilon$ are the same in $\Phi_{\epsilon}$).
Consider the following transformations: $\widetilde{T_{i_0}}=T_{i_0}$ and $\widetilde{T}_{j_0}=T_{i_0}^{-1}T_{j_0}$. These transformation generate the same action as $T_{i_0}$ and $T_{j_0}$, and their projections onto $\Phi_{[d]}$ are $T_{i_0}$ and the identity map, respectively. We can regard $\Phi_{[d]}$ as the $\widetilde{T}_{j_0}$ invariant factor of $\Phi_{[d]\setminus \{j_0\}}$. In this case, Lemma \ref{Lem:productfactor} gives us the existence of the desired topological model.

Now suppose we have built strictly ergodic topological models for all the factors of level $h+1$, where $1\leq h\leq d-1$. We want to build now strictly ergodic models for the factors of level $h$ together with continuous projection into their factors in the level $h+1$. 
Take an $\epsilon$, with $\epsilon=h$, $1\leq h\leq d-1$. $\Phi_{\epsilon}$ is an extension of its the factors in the level $h+1$, so it is an extension of \[ \bigvee_{j\notin \epsilon} \Phi_{\epsilon\cup \{j\}}.\]

{\bf Claim :} If $h+1\leq d$, then the induced $\mathbb{Z}^{d-h +1}$-action on $\bigvee_{j\notin \epsilon} \Phi_{\epsilon\cup \{j\}}$ is free.

To show this claim, take $i_0\in \epsilon$ and suppose that $T_{i_0}^{n_0}\prod_{j\notin \epsilon} T_{j}^{n_j}$ acts trivially on \[ \bigvee_{j\notin \epsilon} \
\Phi_{\epsilon\cup \{j\}}.\]

Projecting onto the different factors defining this joining, we have that $T_{i_0}^{n_0}\prod_{j\notin \epsilon} T_{j}^{n_j}$ acts trivially in $\Phi_{\epsilon\cup \{j_0\}}$ for $j_0\notin \epsilon$. But in this factor, we have that $T_{i_0}^{n_0}\prod_{j\notin \epsilon} T_{j}^{n_j}$ equals $T_{i_0}^{n_{i_0}+n_{j_0}}\prod_{j\notin \epsilon, j\neq j_0} T_{j}^{n_j}$ (because $T_{i_0}=T_{j_0}$ here). Since the $\mathbb{Z}^{d-(h+1)-1}$ induced action is free in $\Phi_{\epsilon\cup \{j_0\}}$, we have that $n_{i_0}+n_{j_0}=n_j=0$ for all $j\notin \epsilon \cup \{j_0\}$. From here, if the set $[d] \setminus \epsilon$ have at least two elements, then $n_{i_0}=n_j=0$ for all $j\notin \epsilon$. This proves the claim.

\

We first build a strictly ergodic model for  \[ \bigvee_{j\notin \epsilon} \Phi_{\epsilon\cup \{j\}}.\] 

Remark that by induction hypothesis, we have strictly ergodic models $\widehat{\Phi}_{\epsilon\cup \{j\}}$ for $j\notin \epsilon$. Let \[Y_{\epsilon}=\text{ supp }(\prod_{j\notin \epsilon} (\Phi_{\epsilon \cup \{j\}})_{\ast}\mu^F) \text{ in } \prod_{j\notin \epsilon} \widehat{\Phi}_{\epsilon \cup \{j\}}, \] where we abuse of notation and also write $\Phi_{\epsilon \cup \{j\}}$ for the factor map from $(X^d,\mu^F)$ to $(\Phi_{\epsilon \cup \{j\}})$.  This space has the property that the projection of the coordinates $\widehat{\Phi}_{\epsilon\cup \{j\}}$ and $\widehat{\Phi}_{\epsilon \cup \{j'\}}$ onto $\widehat{\Phi}_{\epsilon \cup \{j,j'\}}$ are the same (and similarly for further projections). We refer this property as {\it non degeneracy of invariant factors}.

We claim that $Y_{\epsilon}$ is a uniquely ergodic system (under the action of the associate projection of $H_d$). To see this,  let $\lambda$ be an invariant measure on $Y_{\epsilon}$. By unique ergodicity of $\widehat{\Phi}_{\epsilon \cup \{j\}}$, the projection of $\lambda$ onto $\Phi_{\epsilon \cup \{j\}}$ is the unique invariant measure $\mu_{\epsilon \cup \{j\}}$ on this system. Therefore $\lambda$ can be regarded as a joining (in the measure theoretical category) of the systems $\Phi_{\epsilon \cup \{j\}}$, $j\notin \epsilon$. 

We show that the $\lambda$ can be pushed to the further levels of invariant factors. This property is a consequence of the satedness condition we imposed and it appears basically in \cite{A}, Chapter 4. We give the first step and sketch the key arguments for completeness. 

 Pick one system of the joining, say $\Phi_{\epsilon\cup \{j_0\}}$, $j_0\notin J$. For any other $j\notin \epsilon$, we have that $\Phi_{\epsilon\cup \{j\}}$ is a factor of the $j$-th copy of $X$ inside $\mu^F$. If we regard $X_j$ as a a factor of $(X^d,\mu^F)$ and a $\Z^{d+1}$-action, we have that it belongs to the idempotent class $\mathcal{Z}_{\phi}^{j,d+1}$ (the $j$-th and $(d+1)$-th transformations are equal) and therefore the joining $\lambda$ can be viewed as a joining of $\Phi_{\epsilon \cup \{j_0\}}$ with  a system in the idempotent class $\bigvee_{j\notin \epsilon_0}\mathcal{Z}_{\phi}^{j,d+1}$. Arguing as in Proposition 4.2.6 in \cite{A}, we can conclude that this joining can be projected in the $\Phi_{\epsilon\cup \{j_0\}}$- part to \[\bigvee_{j\notin \epsilon, j\neq j_0} \Phi_{\epsilon\cup \{j_0,j\}}.\] A similar argument works for a joining of invariant factors of a given level. Iterating the argument, we have that $\lambda$ is in the end determined by its projection into the $\widehat{\Phi}_{[d]}$-factor, which is uniquely determined since this later system is uniquely ergodic.

Hence $Y_{\epsilon}$ is a strictly ergodic model for $\bigvee_{j\notin \epsilon} \Phi_{\epsilon\cup \{j\}}$ (which is free) and we may use relative Theorem \ref{Weiss} to get a strictly ergodic model $\widehat{\Phi}_{\epsilon}\to Y_{\epsilon}$ for the factor map $\Phi_{\epsilon}\to Y_{\epsilon}$.

Iterating the construction, we end up with  $d$ topological models for $X$ (one for each of its copies inside $(X^d,\mu^F))$, say $\widehat{X}_1,\ldots, \widehat{X}_d$ and a joining of these systems  that we will denote $N_d$, invariant under $H_d$ and with a unique invariant measure for which it is isomorphic to $(X^d,\mu^F)$. 
\end{proof}

In what follows, we continue with the precedent notations and let $N_d$ denotes  the topological model for $\mu^F$ given by Theorem \ref{Thm:ConstructionModel}. 
For  $i=1,\ldots,d$, let $\widehat{X}_i$ denote the  topological model for the $i$-th copy of $X$ inside $\mu^F$ and $\phi_i\colon X\to \widehat{X}_i$ the measure theoretical isomorphism. Denote $\vec{\phi}=\phi_1\times \cdots \times \phi_d \colon X^d \to \widehat{X}_1\times \cdots \widehat{X}_d$. 
Theorem \ref{Thm:ConstructionModel} says (by construction) that $N_d$ is the support of $\vec{\phi}_{\ast}(\mu^F)$ on $\widehat{X}_1\times \cdots \times \widehat{X}_d$.

In order to deduce pointwise convergence results, we need to determine if images of diagonal points are in $N_d$. The following result is fundamental.

\begin{thm}[Theorem 4.0.2, \cite{A}]\label{Recurrence:Austin}
If $\mu(A_1\cap A_2\cdots \cap A_d)>0$, then $\mu^F(A_1\times A_2\cdots \times A_d)>0$. 
\end{thm}

\begin{cor} \label{Lemma:Support}
$\vec{\phi}(x,\ldots,x) \in N_d$ for $\mu$-a.e $x\in X$.
\end{cor}
\begin{proof}
 For $x\in X$, to prove that $\vec{\phi}(x,\ldots,x)\in N_d$, it suffices to show that $\vec{\phi}\mu_F(U_1\times \cdots \times U_d)=\mu_F(\vec{\phi}^{-1} U_1\times \cdots \times U_d)$ is positive for every neighbourhood $U_1\times\cdots\times U_d$ of $\vec{\phi}(x,\ldots,x)$. By Theorem \ref{Recurrence:Austin}, it suffices to show that $\mu(\phi_{1}^{-1}U_1 \cap \cdots \cap \phi_d^{-1}U_d)>0$.

Recall that the maps $\phi_i$ are (only) measurable. By Lusin's Theorem, we can find a sequence of closed sets $\{F_n\}_{n\in \N}$ such that $\phi_i$ restricted to $F
_n$ is continuous for all $i=1,\ldots,d$, and that $\mu(F_n)\geq 1-2^{-n}$. Now consider the sequence $H_n=\text{supp} (\mu_{F_n})$, where $\mu_{F_n}$ is the measure $\mu$ restricted to $F_n$. Then $\mu(H_n)=\mu( F_n)\geq 1-2^{-n}$. If $H=\liminf H_n$ then by the Borel-Cantelli Lemma, we have that $\mu(H)=1$.
Let $x\in H$ and let $U_1\times U_2\cdots \times U_d$ be an open neighbourhood of $\vec{\phi}(x,\ldots,x)$.
Pick $n$ large enough so that $x\in H_n$. Since $\phi_i$ restricted to $H_n$ is continuous, we have that there exists $r>0$ such that $B(x,r)\cap H_n \subseteq \phi_i^{-1}(U_i)\cap H_n$ for all $i=1,\ldots,d$. Since $x$ is in the support of $\mu$ restricted to $H_n$, we get that
\[ \mu(\phi_1^{-1}U_1 \cap \phi_2^{-1}U_2 \cdots \cap \phi_d^{-1}U_d))\geq \mu(B(x,r)\cap H_n) >0.\]

Hence $\mu_F(\phi_1^{-1}U_1 \times \phi_2^{-1}U_2 \times \cap \phi_d^{-1}U_d)=\vec{\phi}_{\ast}\mu_F(U_1\times \cdots \times U_d)>0$. That is, $\vec{\phi}(x,\ldots,x)\in N_d$ for $x\in H$. 
\end{proof}

\subsection{Applications}\label{Sec:app2}

In this section we derive Theorem \ref{THM:averagedmultiple} as a consequence of the construction of the model in the previous section.

\begin{proof}[Proof of Theorem \ref{THM:averagedmultiple}]
	Since the projection of the Fustenberg self-joining of a system to a factor equals to the Furstenberg-Ryzhikov self-joining of the factor system, 
by Theorem \ref{Thm:ext}, we may assume that $X$ is $\mathcal{Z}^{\ast}$-sated. 

Let $f_1,\ldots, f_d$ be measurable bounded functions on $X$ and $\phi_i:X\to \widehat{X}_i$ be the measure theoretical isomorphism given by Theorem \ref{Thm:ConstructionModel}, where $\widehat{\mu}_i$ is the measure on $\widehat{X}_i$. Fix $\epsilon >0$. We can find a continuous function $\widehat{f}_i$ on $\widehat{X}_i$ such that $\|\widehat{f}_i-f_i\circ \phi_i^{-1}\|_{L^1(\widehat{\mu}_i)}\leq \epsilon$. 
By Corollary \ref{Lemma:Support}, $\vec{\phi}(x,\ldots,x)\in N_d$ for $\mu$-a.e $x\in X$. 
Thus for $\mu$-a.e $x\in X$, the average $S_N(\widehat{f}_1,\ldots,\widehat{f}_d,x)$ defined as

\[\frac{1}{N^{d+1}} \sum_{0\leq n_1,\ldots,n_d \leq N-1} \sum_{0\leq n\leq N-1} \widehat{f}_1(T_1^n \prod_{j=1}^d T_j^{n_j} \phi_1 x)\widehat{f}_2(T_2^n \prod_{j=1}^d T_j^{n_j}\phi_2 x)\cdots \widehat{f}_d(T_d^n \prod_{j=1}^d T_j^{n_j} \phi_d x) \]
converges to $\int \widehat{f_1}\otimes\widehat{f}_2\cdots\otimes \widehat{f}_d d(\vec{\phi}_{\ast}\mu^F)$ as $N\to\infty$.

This follows from the fact that any weak limit of 
\[\frac{1}{N^{d+1}} \sum_{0\leq n_1,\ldots,n_d \leq N-1} \sum_{0\leq n\leq N-1} (T_1^n \prod_{j=1}^d T_j^{n_j}, T_2^n \prod_{j=1}^d T_j^{n_j},\cdots ,T_d^n \prod_{j=1}^d T_j^{n_j} )\delta_{(\phi_1 x,\ldots,\phi_d x)} \]
is $H_d$ invariant and thus equal to $\vec{\phi}_{\ast}\mu^F$.

Let $S_n(f_1,\ldots,f_d,x)$ denote the average
\[ \frac{1}{N^{d+1}} \sum_{0\leq n_1,\ldots,n_d \leq N-1} \sum_{0\leq n\leq N-1} f_1(T_1^n \prod_{j=1}^d T_j^{n_j}x)f_2(T_2^n \prod_{j=1}^d T_j^{n_j}x)\cdots f_d(T_d^n \prod_{j=1}^d T_j^{n_j}x) \]

Using the telescoping inequality, we have that for $\mu$-a.e. $x\in X$, 
\begin{equation}
\begin{split}  \limsup_{N\to \infty} \left|S_N(f_1,\ldots,f_d,x)-\widehat{S}_N(\widehat{f}_1,\ldots,\widehat{f}_d,x)\right |
\leq  \sum_{i=1}^{d} \|f_i\circ \phi_i^{-1}-\widehat{f}_1\|_{L^1(\widehat{\mu}_i)}.
\end{split}
\end{equation}
Thus,

\begin{align*}
& \left \vert S_N(f_1,\ldots,f_d,x)-\int f_1\otimes\cdots \otimes f_d d\mu^F \right \vert \\
\leq & \left \vert S_N(f_1,\ldots,f_d,x)-\widehat{S}_N(\widehat{f}_1,\ldots,\widehat{f}_d,x) \right \vert \\ + & \left \vert \widehat{S}_N(\widehat{f}_1,\ldots,\widehat{f}_d,x)- \int  \widehat{f_1}\otimes\widehat{f}_2\cdots\otimes \widehat{f}_d d(\vec{\phi}_{\ast}\mu^F) \right \vert \\
+ & \left \vert \int (f_1\circ \phi_1^{-1}\otimes f_1 \circ \phi_2^{-1} \cdots \otimes f_d\circ \phi_d^{-1} - \widehat{f_1}\otimes\widehat{f}_2\cdots\otimes \widehat{f}_d) d(\vec{\phi}_{\ast}\mu^F) \right \vert.
\end{align*}  
 
Using again the telescoping inequality, we get that 
$$\limsup_{N\to \infty}  \left \vert S_N(f_1,\ldots,f_d,x)-\int f_1\otimes\cdots \otimes f_d d\mu^F \right \vert$$ is bounded by $$2\sum_{i=1}^{d}\|f_i\circ \phi_i^{-1}-\widehat{f}_i\|_{L^1(\widehat{\mu}_i)}\leq 2d \epsilon.$$
Since $\epsilon$ is arbitrary, we get the result.
 \end{proof}

\section{An expression for the $L^2$-limit of multiple ergodic averages and pointwise convergence for distal systems} \label{Sec:AverageDistal}
In this section, we prove Theorem \ref{THM:ergodicmeasurelimit} and Theorem \ref{THM:Averagesdistal}. Our results rely on the study of the Furstenberg-Ryzhikov self-joining of a suitable extension system, and the study of invariant $\sigma$-algebras on it. 

In this section, for convenience we use the letters $S_1,\ldots,S_d$ to name the transformations of a space $(X,\mathcal{X},\mu)$. The letter $T$ will be used to denote $T_1=S_1$ and $T_i=S_1^{-1}S_i$ for $i\geq 2$.

Let $(X,\mu,S_1,\ldots,S_d)$ be a system with commuting transformations. To settle notation, in all what follows, $\mu^{F}$ is the Furstenberg joining of $(S_{1},\dots,S_{d})$,
$p\colon X^{d}\to X^{d-1}$ is the projection onto the last $d-1$ coordinates, and denote $\nu^{F}=p_{*}(\mu^{F})$.
Recall from Remark \ref{rem:proj} that $\nu^{F}$ is the Furstenberg-Ryzhikov self-joining associated to $(T_2,\ldots,T_d)$.

In all what follows, we assume that $(X,\mu,S_1,\ldots,S_d)$ is  $\mathcal{Z}^{\ast}$-sated (see Definition \ref{zast}). As in the previous section, this condition give us a good picture of the Furstenberg-Ryzhikov self-joining but also implies have the {\em magic} conditions that relate the Host's seminorms with invariant $\sigma$-algebras. To be more specific, by Lemma \ref{lem:SatedAndMagic} the condition of being sated with respect to the idempotent class
\[ \bigvee_{i\in \epsilon} Z_{\phi}^{\{1,i\}} \]
 (this condition which is included in being $\mathcal{Z}^{\ast}$-sated) implies that 

  \[ \mathbb{E}\Bigl (f \vert \bigvee_{i\in \epsilon} \Phi_{\{1,i\}} \Bigr )= \mathbb{E}\Bigl (f \vert \bigvee_{i\in \epsilon} I_{T_{i}}\Bigr )=0 \text{ if and only if } \normm{f}_{\mu,B}=0,\]
where $B=\{T_i: i\in \epsilon\}$ (recall that $\Phi_{I}(X), I\subseteq [d]$ is the factor of $X$ associated to the $\sigma$-algebra invariant under all the transformations $S_i^{-1}S_j$, for $i,j \in I$, $i\neq j$).

\subsection{Description of the $\sigma$-algebra of $(T_2 \times \ldots \times T_d)$-invariant sets} 
Theorem \ref{THM:ergodicmeasurelimit} follows from the study of the $\sigma$-algebra of $(T_2 \times \ldots \times T_d)$-invariant sets of the projections of $\mu^F$ onto the last $d-1$ coordinates.

\begin{thm}\label{Thm:key3}
	Let $d\geq 3$ and $(X,\mu,S_{1},\dots,S_{d})$ be  a $\mathcal{Z}^{\ast}$-sated system with commuting transformations. Then the spaces \[(X_{{2}}\vee \dots\vee X_{{d}},\I_{(T_2\times\dots\times T_{d})},\nu^{F})\] and \[(\bigvee_{i\in\{2,\dots,d\}}\Phi_{\{1,i\}},\nu^{F}).\]
are isomorphic (as factors of $(X^d,\mu^F)$).	
\end{thm}

\begin{rem} $X_2, \ldots, X_d$ are the $d-1$ last coordinates in $(X^d,\mu^F)$ and we think of everything as a sub $\sigma$-algebra of $(X^d,\mu^F)$. For convenience of notation (for what follows) we use the notation  $X_{{2}}\vee \dots\vee X_{{d}}$ instead of $X_2\times\dots \times X_d$.
\end{rem}

We start with some lemmas.
To avoid confusion with the transformations $S_1,\ldots,S_d$ and $T_2,\ldots,T_d$,  we use the notations $R_1,\ldots,R_d$ in the lemmas:

\begin{lem}\label{vdc}
	Let $(X,\mu,R_{1},\dots,R_{d})$ be a measure preserving system with commuting transformations. For every $f_{1},\dots,f_{d}\in L^{\infty}(\mu)$, we have
	\begin{equation}\nonumber
		\begin{split}
			\Bigl \Vert\E \Bigl(\bigotimes_{i=1}^{d}f_{i} \vert\I_{(R_1\times\dots\times R_{d})}\Bigr)\Bigr \Vert_{L^{2}(\nu)}
			\leq C\min_{1\leq i\leq d}\normm{f_{i}}_{\mu,I_{i}}
		\end{split}
	\end{equation}
	for some universal constant $C$ depending only on $d$, where $\nu$ is any d-fold self-joining of $(X,\mu,R_{1},\dots,R_{d})$ and $I_{i}=\{R_{i},R_{i}R^{-1}_{j}: j\neq i\}$ for $1\leq i\leq d$.
	
	As a consequence, for any factors $Y_{i},1\leq i\leq d$ of $X$, the space 
	\[(Y_{1}\vee \dots \vee Y_{d},\I_{(R_1\times\dots\times R_{d})},\nu)\] is isomorphic to \[(Z^{I_{1}}(Y_{1})\vee \dots\vee Z^{I_{d}}(Y_{d}),\I_{(R_1\times\dots\times R_{d})},\nu).\]
\end{lem}
\begin{proof}
The proof is done by induction on $d$ and is similar to Theorem 12.1 in \cite{HK05}. We assume the result is true for all $(d-1)$-fold self-joining of $X$. Let $f_1,\ldots,f_d$ be bounded functions. We may assume that they are bounded by $1$ in $L^{\infty}(\mu)$-norm. By the Ergodic Theorem, we have that 
\[	\Bigl \Vert\E \Bigl(\bigotimes_{i=1}^{d}f_{i} \vert\I_{(R_1\times\dots\times R_{d})}\Bigr)\Bigr \Vert_{L^{2}(\nu)}= \Bigl \Vert \lim_{N\to \infty} \frac{1}{N}\sum_{n=0}^{N-1} \bigotimes_{i=1}^{d}f_{i}\circ R_i^n  \Bigr \Vert_{L^{2}(\nu)} \]

Applying van der Corput and Cauchy-Schwartz inequalities, we have that the right hand side is bounded by 	
\begin{equation*}
\begin{split}
&\limsup_{H\to \infty} \frac{1}{H} \sum_{h=0}^{H-1}   \limsup_{N\to \infty} \Bigl \vert \frac{1}{N} \sum_{n=0}^{N-1} \int \bigotimes_{i=1}^{d}f_{i}\circ R_i^n \cdot \bigotimes_{i=1}^{d}f_{i}\circ R_i^{h+n} d\nu \Bigr \vert  \\
= & \limsup_{H\to \infty} \frac{1}{H} \sum_{h=0}^{H-1} \limsup_{N\to \infty}\Bigl \vert \frac{1}{N} \sum_{n=0}^{N-1} \int f_1\cdot f_1 R_1^h \otimes \bigotimes_{i=2}^{d}(f_{i}\circ R_i^h  \cdot f_i)\circ (R_1^{-1}R_i)^n d\nu \Bigr \vert \\
\leq & \limsup_{H\to \infty} \frac{1}{H} \sum_{h=0}^{H-1} \|f_1 \cdot f_1 R_1^h \|_{L^2(\mu)} \limsup_{N\to \infty}\Bigl \| \frac{1}{N} \sum_{n=0}^{N-1}  \bigotimes_{i=2}^{d}(f_{i}\circ R_i^h  \cdot f_i)\circ (R_1^{-1}R_i)^n    \Bigr \|_{L^2(\nu')}, 
\end{split}
\end{equation*}
where $\nu'$ is the projection of $\nu$ onto the last $d-1$ coordinates. By the induction hypothesis, the last $\limsup$ is bounded by 
\[ \min_{2\leq i \leq d} \normm{f_{i}\circ R_i^{h} \cdot f_i}_{\mu,I'_{i}},  \] 
where $I_i'=\{R_1^{-1}R_i, (R_1^{-1}R_i)^{-1}R_1^{-1}R_j: ~ 2\leq j\leq d, j\neq i \}=\{R_1^{-1}R_i, R_i^{-1}R_j: ~ 2\leq j\leq d \}=\{R_i^{-1}R_j: j\neq i\} $ (here we use the Theorem \ref{ine} (2)). 

We get the bound 
\[\min_{2\leq i \leq d} \limsup_{H\to \infty} \frac{1}{H} \sum_{h=0}^{H-1} \normm{f_{i}\circ R_i^h \cdot f_i}_{\mu,I'_{i}}, \]
which by Theorem \ref{ine} equals to $\min_{2\leq i \leq d} \normm{f_{i}}_{\mu,I_{i}}$,
where $I_{i}=\{ R_i, R_i^{-1}R_j: j\neq i\}$. 
Changing the role of the functions we get the bound 
$\min_{1\leq i \leq d} \normm{f_{i}}_{\mu,I_{i}}$
and we are done.
\end{proof}

For any $I\subseteq[d]$, the transformations $S^{-1}_{1}S_{i}, i\in I$ act in the same way in $\Phi_{I}$ (because $S_1^{-1}S_i(S_1^{-1}S_{j})^{-1}=S_iS_j^{-1}$ which acts trivially in $\Phi_{I}$). We use $T^{*}$ to denote any expression of this transformation. 
Note that $T^{*}$ acts trivially on $\Phi_{I}$ if $1\in I$. For convenience, for $\vert I\vert=1, I=\{i\}$, we denote $\Phi_{I}=X_{i}$.
We have
\begin{lem}\label{nvdc}
	Let $(X,\mu,S_{1},\dots,S_{d})$ be a $\mathcal{Z}^{\ast}$-sated measure preserving system with commuting transformations and $1\leq k\leq d-1$.
	Then the space 
	\[(\bigvee_{I\subseteq\{2,\dots,d\},\vert I\vert=k}\Phi_{I},\I_{T^{*}},\nu^{F})\] is isomorphic to \[\Bigl (\bigvee_{I\subseteq\{2,\dots,d\},\vert I\vert=k}\Phi_{\{1\}\cup I},\nu^{F}\Bigr )\vee\Bigl (\bigvee_{I\subseteq\{2,\dots,d\},\vert I\vert=k+1}\Phi_{I},\I_{T^{*}},\nu^{F} \Bigr).\]
\end{lem}
\begin{rem}
	In this lemma, we use $\nu^{F}$ to also denote its projections onto the correspondent coordinates. 
\end{rem}
\begin{proof}
Denote $A_{k}=\{I\subseteq\{2,\dots,d\}\colon\vert I\vert=k\}$ and $A_{k+1}=\{I\subseteq\{1,\dots,d\}\colon\vert I\vert=k+1\}$.
	Fix any $J\in A_{k}$. Pick $T^{J}=S^{-1}_{1}S_{j}$ for an arbitrary $j\in J$. For all $J'\in A_{k}, J'\neq J$, pick $T^{J'}=S^{-1}_{1}S_{j'}$ for an arbitrary $j'\in J'\backslash J$. Since $\vert J'\vert=\vert J\vert$, such $j'$ always exists. Moreover, for all $i\in \{2,\dots,d\}\backslash, i\neq j$, there exists $J'\in A_{k}$ such that $j'=i$ (take for example $J'=\{i\}\cup J\backslash\{j\}$).
	
Recall that $\Phi_{I}$ is the factor of $X$ corresponding to the sub-$\sigma$-algebra of $T_{i}=S_{1}^{-1}S_i$ invariant sets for all $i\in I$.	
By Lemma \ref{vdc}, $$(\Phi_{A_{k}},\I_{T^{*}},\nu^{F}):=(\bigvee_{I\in A_{k}}\Phi_{I},\I_{(\prod_{I\in A_{k}}T^{I})},\nu^{F})$$ is isomorphic to 
$$(\bigvee_{I\in A_{k}}Z^{B_I}(\Phi_{I}),\I_{(\prod_{I\in A_{k}}T^{I})},\nu^{F}),$$ where $B_I=\{T^{I},(T^{I})^{-1}T^{I'},I'\in A_{k},I'\neq I\}$. Particularly, $(\Phi_{A_{k}},\I(T^{*}),\nu^{F})$ is isomorphic to \[(Z^{A_J}(\Phi_{J})\bigvee_{J'\in A_{k},J'\neq J}\Phi_{J'},\I_{(T^{J}\times\prod_{J'\in A_{k},J'\neq J}T^{J'})},\nu^{F}).\]
	Note that 
	\begin{equation}\nonumber
		\begin{split}
			&\quad Z^{B_{J}}(\Phi_{J})=Z^{T^{J}}(\Phi_{J})\bigvee_{J'\neq J}Z^{T^{J'}}(\Phi_{J})
			\\&=Z^{S^{-1}_{1}S_{j}}(\Phi_{J})\bigvee_{J'\neq J}Z^{S^{-1}_{j}S_{j'}}(\Phi_{J})
			=\Phi_{\{1\}\cup J}\bigvee_{J'\neq J}\Phi_{\{j'\}\cup J}.
		\end{split}
	\end{equation}
	Since $X$ is $\mathcal{Z}^{\ast}$-sated, we have that $\Bigl(\bigvee_{I\in A_{k}}\Phi_{I},\I_{T^{*}},\nu^{F}\Bigr)$ is isomorphic to \[\Bigl ((\bigvee_{I\in A_{k+1},J\subseteq I}\Phi_{I})\vee(\bigvee_{I\in A_{k}, I\neq J}\Phi_{I}),\I_{T^{*}},\nu^{F}\Bigr ).\]

To justify this, it suffices to show that $\mathbb{E}\Bigl(\bigotimes_{I \in A_{k}} f_{I}\vert \I_{T^{*}}\Bigr )$ is measurable with respect to  $(\bigvee_{I\in A_{k+1}}\Phi_{I},\nu^{F})$  whenever $f_I$ is measurable with respect to $\Phi_{I}$. Choose some $I_0\in A_{k}$ and let $B_{I_{0}}=\{T^{I_{0}},(T^{I_{0}})^{-1}T^{I},I\in A_{k},I\neq I_0\}$. By the Ergodic Theorem and Lemma \ref{vdc}, we have that
\begin{align*}
&\Bigl \| \mathbb{E}\Bigl(f_{I_0}-\mathbb{E}(f_{I_0}\vert \bigvee_{I'\in A_{K+1}\atop I_0\subseteq I'} \Phi_{I'})  \bigotimes_{I \in A_{k}, I\neq I_0} f_{I}\Big\vert \I_{T^{*}}\Bigr )\Bigr \|_{L^2(\nu^F)} \\
\leq & \Bignormm{f_{I_0}-\mathbb{E}\Bigl (f_{I_0}\vert \bigvee_{I'\in A_{K+1}\atop I_0\subseteq I'} \Phi_{I'}\Bigr)}_{B_{I_0}}=0.\\
\end{align*}

So in $\mathbb{E}\Bigl(\bigotimes_{I \in A_{k}} f_{I}\vert \I_{T^{*}}\Bigr )$ we may replace each function $f_{I}$ by its respective conditional expectation with respect to $\bigvee_{I'\in A_{K+1}\atop I\subseteq I'}\Phi_{I'}$.   	
We get that $(\bigvee_{I\in A_{k}}X_{I},\I_{T^{*}},\nu^{F})$ is in fact measurable with respect to 
\[(\bigvee_{I\in A_{k}}\bigvee_{I'\in A_{k+1} \atop I\subseteq I'}\Phi_{I'},\I_{T^{*}},\nu^{F})=(\bigvee_{I\in A_{k+1}}\Phi_{I},\I_{T^{*}},\nu^{F}).\]
	 
Note that $T^{*}$ acts trivially on $\Phi_{I}$ for all $I\in A_{k+1}, 1\in I$. This finishes the proof.
\end{proof}	

\begin{proof}[Proof of Theorem \ref{Thm:key3}]
	By Lemma \ref{nvdc}, \[(X_{2}\vee\dots\vee X_{d},\I_{(S^{-1}_{1}S_{2}\times\dots\times S^{-1}_{1}S_{d})},\nu^{F})\] 
	is isomorphic to \[(\bigvee_{I\subset\{2,\dots,d\}}\Phi_{\{1\}\cup I},\nu^{F}).\] 

Since $\Phi_{\{1\}\cup I}$ is a factor of $\Phi_{\{1,i\}}$ for any $i\in I$, we conclude that 
	 $(\bigvee\limits_{I\subset\{2,\dots,d\}}\Phi_{\{1\}\cup I},\nu^{F})$ is isomorphic to
	$(\bigvee\limits_{i\in\{2,\dots,d\}}\Phi_{\{1,i\}},\nu^{F})$, which finishes the proof.
\end{proof}

\subsection{Description of the measure $\mu^F$.}\label{Sec:app3}
Let $\mu^F$ be the Furstenberg-Ryzhikov self-joining associated to $(S_1,\ldots,S_d)$ in $X^d$. Recall that the projection onto the last $d-1$ coordinates is the Furstenberg-Ryzhikov self-joining associated to $(S^{-1}S_2,\ldots,S^{-1}S_d)=(T_2,\ldots,T_d)$ is denoted by $\nu^F$. 

We may decompose $\mu^F$ with respect to $\nu^F$  \[\mu^F=\int_{X_{2}\times\dots\times X_{d}} \lambda_{\vec{x}}\times \delta_{\vec{x}} d\nu^F(\vec{x}). \]
By the invariance of $\mu_F$ under $\id \times T_2\cdots \times T_d$, we get that $\lambda_{T_2\cdots \times T_d\vec{x}}=\lambda_{\vec{x}}$. Hence the map $\vec{x}\mapsto \lambda_{\vec{x}}$, $X_{2}\times\dots\times X_{d}\to M(X)$ is $I(T_2\times \cdots\times T_d)$-measurable. By Theorem \ref{Thm:key3}, this $\sigma$-algebra is isomorphic to the $\bigvee_{j\geq 2} X_{\{1,j\}}$ factor of the first copy of $X$ in $\mu^F$. To ease notation we write $Y_1=\bigvee_{j\geq 2} X_{\{1,j\}}$ and let $\pi_1\colon X\to Y_1$ denote the corresponding factor map. 
We use the variable $s$ to denote points in $Y_1$ to avoid confusion. We can write $\lambda_{\vec{x}}=\lambda_{s}$ under the isomorphism given by Lemma \ref{nvdc}. In particular, $\lambda_s(\pi_1^{-1}s)=1$. 

Let $$\nu^F=\int_{Y_1} \nu^{F}_{s}d\mu_1(s)$$ be the disintegration of $\nu_F$ over its $(T_2\times\cdots\times T_d)$-invariant $\sigma$-algebra (here we identity this factor with $Y_1$). Since it is the disintegration over the invariants, we get that for $\mu_1$-a.e. $s\in Y_1$, the measure $\nu^F_{s}$ is $(T_2\times\cdots\times T_d)$-ergodic. 
Then \[\mu^F=\int \lambda_{\vec{x}}\times \delta_{\vec{x}} d\nu^{F}(\vec{x})=\int \int_{Y_{1}} \lambda_{s}\times\delta_{\vec{x}}d\nu^F_{s}(\vec{x})d\mu_1(s)=\int_{Y_{1}} \lambda_{s}\times \nu^F_{s} d\mu_1(s).  \]

On the other hand, projecting $\mu^F$ onto the first coordinate, we have $p_1\mu^F=\mu$ and then $\mu=\int_{Y_1} \lambda_{s}d\mu_1(s)$. Since $\lambda_s(\pi_1^{-1}(s))=1$, we have that $\mu=\int_{Y_1} \lambda_{s}d\mu_1(s)$ is in fact the disintegration of $\mu$ over $\mu_1$. 

We are now ready to proof Theorem \ref{THM:ergodicmeasurelimit}.
\begin{proof}[Proof of Theorem \ref{THM:ergodicmeasurelimit}]
Let $(Y,T_2,\ldots,T_d)$ be a system of $d-1$ commuting transformations. Here for convenience of notation we start with the index 2. We set $T_1=\id$ and regard it $(Y,T_1,T_2,\ldots,T_d)$ as a $\Z^{d}$ measure preserving system.  We may change coordinates in $\Z^d$  and regard the system $(Y,\mu, S_1,S_2,\ldots, S_d)$, where $S_1=T_1$ and $S_i=T_1T_i$, $i\geq 2$.  By Theorem \ref{Thm:ext}, we may find an $\mathcal{Z^{\ast}}$-extension of this system $(X,\mu,S_1,\ldots,S_d)$ and we work on this extension from now on. It is important to stress that we consider the action on the new system of coordinates (different system of coordinates lead to different extensions). We remark that $(X,\mu,S_{1}^{-1}S_2,\ldots,S_{1}^{-1}S_d)$ is also an extension of $(Y,\mu,T_2,\ldots,T_d)$ and thus it suffices to prove this theorem for $X$.

We claim that $\mu^{F}_{x}=\nu^{F}_{\pi_{1}(x)},x\in X$ satisfy the requirements in the statement.

Recall that \[\mu^F=\int \lambda_{\vec{x}}\times \delta_{\vec{x}} d\nu^{F}(\vec{x})=\int \int \lambda_{s}\times\delta_{\vec{x}}d\nu^F_{s}(\vec{x})d\mu_1(x)=\int \lambda_{s}\times \nu^F_{s} d\mu_1(s),  \]
where $\mu=\int_{Y_1} \lambda_{s}d\mu_1(s)$ and $\nu^F_{s}$ is $(T_2\times\cdots\times T_d)$-ergodic. 
Disintegrating $\mu$ with respect to $\mu_1$ and using Fubini's Theorem, we have the following expression 
\begin{align} \label{ExpreFurstenberg}
\begin{split} & \int_X \delta_{x}\times \nu^F_{\pi_1(x)} d\mu(x) = \int_{Y_1} \int_{X} \delta_{x}\times \nu^F_{\pi_1(x)} d\lambda_{s}(x) d\mu_1(s)=\int_{Y_1} \int_{X} \delta_{x}\times \nu^F_{s} d\lambda_{s}(x) d\mu_1(s)\\ = & \int_{Y_1}( \int_{X} \delta_{x}d\lambda_{s}(x)) \times \nu^F_{s}  d\mu_1(s)=\int_{Y_1} \lambda_{s}\times \nu^F_{s}d \mu_1(s)=\mu^F.
\end{split}
\end{align}

Let $f_2,\ldots,f_d\in L^{\infty}(\mu)$ and let $F$ be the $L^2$-limit of $\lim\limits_{N\to \infty}\frac{1}{N}\sum_{n=0}^{N-1} f_2(T_2^nx)\cdots f_d(T_d^nx)$. Then

\begin{align*}
 &\int_{X} f_1(x)F(x)d\mu(x)=\lim_{N\to \infty}\frac{1}{N}\sum_{n=0}^{N-1} \int f_1(x)f_2(T_2^nx)\cdots f_d(T_d^nx)d\mu(x)\\
 &=\lim_{N\to\infty} \int \frac{1}{N}\sum_{n=0}^{N-1} f_1(T_1^n x)f_2((T_1T_2)^n x)\cdots f_d((T_1T_d)^nx)d\mu(x)\\
  &=\lim_{N\to\infty} \int \frac{1}{N}\sum_{n=0}^{N-1} f_1(S_1^n x)f_2(S_2^n x)\cdots f_d(S_d^nx)d\mu(x)\\
 &= \mu^F(f_1\otimes f_2\cdots\otimes f_d)=\int_{X} f_1(x)(\int f_2 \otimes\cdots\otimes f_d d\nu^F_{\pi_1(x)})d\mu(x),
\end{align*}
where the last equality follows from the expression \eqref{ExpreFurstenberg}. We conclude that  $F(x)=\int f_2 \otimes\cdots\otimes f_d d\nu^F_{x}$ where we slightly abuse notation and write $\mu^F_x=\nu^F_{\pi_1(x)}$.
\end{proof}

\subsection{Multiple averages in distal systems} \label{Sec:Distalaverages}
We start with the basic definitions of distal systems and refer to \cite{Glas} Chapter 10 for further details. We use many concepts and facts used in \cite{HSY} and \cite{DS2}. 
\begin{defn}
Let $\pi\colon (X,\mathcal{X},\mu,G)\to (Y,\mathcal{Y},\nu,G)$ be a factor map between two ergodic systems. We say $\pi$ is an \emph{isometric} extension if there exist a compact group $H$, a closed subgroup $\Gamma$ of $H$, and a cocycle $\rho\colon G \times Y\to H$ such that $(X,\mathcal{X},\mu,G)\cong (Y\times H/\Gamma,\mathcal{Y}\times\mathcal{H}, \nu\times m, G)$, where $m$ is the Haar measure on $H/\Gamma$, $\mathcal{H}$ is the Borel $\sigma$-algebra on $H/\Gamma$, and that for all $g \in G$, we have
$$g(y,a\Gamma)=(g y, \rho(g,y)a\Gamma).$$

We say that $\pi\colon (X,\mathcal{X},\mu,G)\to (Y,\mathcal{Y},\nu,G)$ is an {\it isometric extension} with fiber $H/\Gamma$ and cocycle $\rho$ and denote it by $Y\times_{\rho} H/\Gamma$.
\end{defn}

\begin{rem} \label{Rem:TopologyH} The group ${\rm Aut}(X,\mu)$ of measurable transformations of $X$ which preserve the measure $\mu$ is a Polish group endowed it with the weak topology of convergence in measure (see \cite{BK}, Chapter 1). Under this topology, the convergence is characterized as follows:
\[ h_n\to h\in {\rm Aut}(X,\mu) \text{ if and only if } \| f\circ h-f\circ h_n \|_{L^2(\mu)}\to 0 \text{ for all } f\in L^2(\mu).\] 
The inclusion of the compact group $H$ in ${\rm Aut}(X,\mu)$ is continuous since measurable morphisms between Polish groups are automatically continuous (see \cite{BK}, Chapter 1, Theorem 1.2.6). This fact does not depend on the topological model chosen for $X$.  
\end{rem}

\begin{rem}
For every isometric extension $\pi\colon X\to Y$ with fiber $H/\Gamma$ and  measurable function  $f$ on $(X,\mu)$, the conditional expectation of $f$ (as a function on $(X,\mu)$) with respect to $Y$ is 
\[\mathbb{E}(f\vert \mathcal{Y})(x)=\int_{H} f(hx)dm(h).\]

Equivalently (as a function on $(Y,\mathcal{Y},\nu)$),

\[\mathbb{E}(f\vert Y)(y)=\int_{H} f(hx)dm(h) \quad \text{ for all }\pi(x)=y. \]   
  
\end{rem}

\begin{defn}\label{dis}
Let $\pi\colon (X,\mathcal{X},\mu,G)\to (Y,\mathcal{Y},\nu,G)$ be a factor map between two ergodic systems. We say $\pi$ is a \emph{distal extension} if there exist a countable ordinal $\eta$ and a directed family of factors $(X_{\theta},\mu_{\theta}, G), \theta\leq\eta$ such that
\begin{enumerate}
\item  $X_{0}=Y$, $X_{\eta}=X$;
\item  For $\theta<\eta$, the extension $\pi_{\theta}\colon X_{\theta+1}\to X_{\theta}$ is isometric and is not an isomorphism;
\item  For a limit ordinal $\l\leq \eta$, $X_{\l}=\lim\limits_{\leftarrow \theta<\l} X_{\theta}$.
\end{enumerate}
We say $X$ is a \emph{distal system} if $X$ is a distal extension of the trivial system.
\end{defn}

We adopt here the same definition when  $G$ is not ergodic. In all the cases we consider, the group $G$ is a subgroup of an ergodic action, so we will not take other approaches to non-ergodic distal systems such as in \cite{Aus2010}.

An equivalent definition of an ergodic measurable distal system can be given using {\it separating sieves}:
\begin{defn}
Let $\pi\colon (X,\mathcal{X},\mu,G)\to (Y,\mathcal{Y},\nu,G)$ be a factor map between two ergodic systems. A {\it separating sieve} for $X$ over $Y$ is a sequence of measurable subset $\{A_i\}_{i\in \mathbb{N}}$ with $A_{i+1}\subseteq A_{i}$, $\mu(A_i)>0$ and $\mu(A_i)\to 0$ such that there exists a measurable subset $X'\subseteq X$, $\mu(X')=1$ with the following property: for $x,x'\in X'$, if $\pi(x)=\pi(x')$ and for every $i\in \mathbb{N}$ there exists $g\in G$ such that $gx,gx'\in A_i$, then $x=x'$.
\end{defn}

\begin{prop}(\cite{Glas}, Chapter 10)
	Let $(X,\mathcal{X},\mu,G)$ be an extension of $(Y,\mathcal{Y},\nu,G)$. Then $X$ is a distal extension of $Y$ if and only if there exists a separating sieve for $X$ over $Y$.
\end{prop}

\begin{defn}
	Let $(X,\mu,S_1,\ldots,S_d)$ be a system with commuting transformations and let $p_i\colon$ $(X,\mu,S_1,$ $\ldots,S_d)\to (Y_i,\mu,S_1,\ldots,S_d)$, $i=2,\ldots,d$ be $d$-factor maps. We say that $(p_1,\ldots,p_d)$ is a {\it good} tuple for the pointwise convergence of multiple averages for $(S_1,\ldots,S_d)$  if 
	
	\[\frac{1}{N}\sum_{n=0}^{N-1} f_1(S_1^n x)\cdots f_d(S_d^nx)\]
	converges $\mu$-a.e. $x\in X$ whenever $f_i$ is measurable with respect to $Y_i$.
	
	This is equivalent to say that 
	\[\frac{1}{N}\sum_{n=0}^{N-1} \mathbb{E}(f_1|Y_1)(S_1^n p_1x )\cdots  \mathbb{E}(f_d|Y_d)(S_d^n p_d x)\]
	converges $\mu$-a.e. $x\in X$ for all measurable functions $f_1,\ldots,f_d$.
\end{defn}

The proof of the following proposition is similar to the one in Proposition 4.16 in \cite{DS2}.

\begin{prop}\label{Prop:iso}
Let $(X,\mu,S_1,\ldots,S_d)$ be magic for $(T_1,\ldots,T_d)$. Let $\phi'\colon X\to Y$ and $\phi\colon X \to Z$ be two factor maps such that $Y$ is an isometric extension of $Z$ with a fiber $H/\Gamma$, and $Z$ is an extension of $\bigvee_{i=1}^d \mathcal{I}_{T_i}=\I_{S_1}\bigvee_{j\neq 1} \I_{S_1^{-1}S_j}$. If $(\phi,\id,\ldots,\id)$ is good, then so is $(\phi',\id,\ldots,\id)$. 
\end{prop}
\begin{rem}
We put the condition of being magic to have a good characterization of factor $\bigvee_{j=1}^d \mathcal{I}_{T_j} $ in terms of the seminorm $\normm{\cdot}_{\mu,T_1,\ldots,T_d}$.
\end{rem}

We need some definition for the proof.
\begin{defn}
	Let $\pi\colon X\to Y$ be an isometric extension with fiber $H/\Gamma$ and let $\varphi\colon H\to \mathbb{R}_{+}$ be a continuous function. We say that $\varphi$ is a {\it weight} if $\int_H \varphi(h)dm(h) =1$ and $\varphi(h^{-1}gh)=\varphi(g)$ for all $g,h\in H$.
	
	Let $f\in L^{\infty}(\mu)$. The conditional expectation of $f$ with weight $\varphi$ over $Y$ is defined to be 
	\[\mathbb{E}_{\varphi}(f\vert \mathcal{Y})(x)=\int_{G} f(hx)\phi(h)dm(h). \]
\end{defn} 
\begin{rem}\label{ConditionPhi}
	We use the cursive symbol $\mathcal{Y}$ to stress that this function may not be constant on the fibers of $\pi$ (thus is not a function on $Y$). Remark also that if $\varphi=1$, $\mathbb{E}_{\varphi}(f\vert \mathcal{Y})(x)=\mathbb{E}(f\vert \mathcal{Y})(x)=\mathbb{E}(f\vert {Y})(\pi(x))$.
\end{rem}
These weighted conditional expectations were considered in Proposition 6.3 in \cite{HSY} and in \cite{DS2}. They are helpful when lifting the property of pointwise convergence.

The following lemma is identical to Lemma 4.15 in \cite{DS2} and so we omit the proof. 

\begin{lem} \label{lem:WeightedConditional}
	Let $\pi\colon X\to Y$ be an isometric extension with fiber $H/\Gamma$. Let $\varphi\colon H\to \mathbb{R}_{+}$ be a weight and $f\in L^{\infty}(\mu)$. Then for $R\in \langle S_1,\ldots,S_d\rangle$ we have
	\[ \mathbb{E}_{\varphi}(f\circ R \vert \mathcal{Y})(x)=\int_{H} f\circ h \circ R (x)\varphi(h)dm(h).\]
\end{lem}

\begin{proof}[Proof of Proposition \ref{Prop:iso}]

By Theorem \ref{Weiss}, we may assume that all the spaces and factor maps are topological, {\em i.e.} the spaces are compact metric and the transformations are continuous. Let $\lambda$ be any weak limit of the sequence 
\[\lambda_N=\frac{1}{N}\sum_{n=0}^{N-1} (S_1^n\times \cdots \times S_d^n) \delta_{(\phi'x,x \ldots,x)} \]

in $M(Y\times X\times\cdots\times X)$. 

We prove that $\lambda=(\phi'\times \id\times \cdots \times \id)_{\ast}\nu^F_{x}$, the measure given by the $L^2$-convergence in Theorem \ref{THM:ergodicmeasurelimit}. Note that $\lambda$ is $(S_1\times\cdots\times S_d)$-invariant since the transformations are continuous. By hypothesis, since $(\phi,\id,\ldots,\id)$ is good for the pointwise convergence, its projection
\[ (\phi_{Y,Z}\times \id\cdots \times \id)_{*}\lambda_N=\frac{1}{N}\sum_{n=0}^{N-1} (S_1^n\times \cdots \times S_d^n) \delta_{(\phi x,x,\ldots,x)} \in M(Z\times X\times\cdots\times X) \]

($\phi_{Y,Z}$ is the factor map from $Y$ to $Z$) converges to a measure which has to be $(\phi\times \id\cdots \times \id)_{\ast}\mu^{F}_{x}$ by Theorem \ref{THM:ergodicmeasurelimit}.

Let $\varphi\colon H\to \mathbb{R}_{+}$ be a weight and let $\mu^F_{\varphi}$ denote the measure in $Y\times X\times\cdots\times X$ defined by  
\[\int f_1\otimes \cdots \otimes f_d \mu^F_{\varphi}:= \int \mathbb{E}_{\varphi}(f_1\vert Z)\otimes f_2\cdots\otimes f_d  d\lambda.\]

where $f_1$ is measurable with respect to $Y$ and $f_2,\ldots,f_d$ are measurable. 

We have that 
\begin{align*}
\left \vert \int f_1\otimes \cdots \otimes f_d \mu^F_{\varphi} \right \vert & \leq \|\varphi\|_{\infty}\int \Bigl (\int|f_1(hx_1)|dm(h) \Bigr) |f_2(x_2)\cdots f_d(x_d)|d\lambda(x_1,\ldots,x_d) \\
& = \|\varphi\|_{\infty}\int  \mathbb{E}(|f_1| ~ \vert Z) (x_1)|f_2(x_2)\cdots f_d(x_d)|d\lambda(x_1,\ldots,x_d) \\
&=\|\varphi\|_{\infty} \int \mathbb{E}(|f_1| ~ \vert Z)\otimes \cdots \otimes |f_d|d\mu^{F,\phi}_x\\
&= \|\varphi\|_{\infty}\int |f_1|\circ \phi'\otimes \cdots \otimes |f_d|d\mu^F_{x},  
\end{align*} 
where $\mu^{F,\phi}_x:=(\phi\times \id \cdots \times \id)_{\ast}\mu^F_x$ and the last equalities come from the fact that the $L^2$-limit of multiple averages remains unchanged if we replace $f_1$ by $\mathbb{E}(f_1 \vert Z)$ (this factor is above the factor $\bigvee_{i=1}^d \mathcal{I}_{T_i}$). So $\mu^F_{\varphi} \ll \mu^{F}_x$.
On the other hand, by Fubini's Theorem, the invariance of $\lambda$ under $S_1\times\cdots\times S_d$ and Lemma \ref{lem:WeightedConditional}, we have

\begin{align*}
 & \quad\int  f_1 \circ S_1 \otimes\cdots \otimes f_d \circ S_d d\mu^F_{\varphi} =\int\mathbb{E}_{\varphi}(f_1\circ S_1\vert Z) \otimes f_2 \circ S_2 \otimes \cdots\otimes f_d \circ S_d d\lambda\\
& = \int \int \Bigl( f_1\circ h\circ S_1 \otimes f_2\circ S_2\otimes \cdots \otimes f_d \circ S_d  d\lambda \Bigr ) \varphi(h)dm(h) \\
& = \int\int f_1\circ h \varphi(h)\otimes f_2\otimes \cdots\otimes f_d dm(h)d\lambda\\
&= \int\mathbb{E}_{\varphi}(f_1\vert Z)\otimes f_2 \otimes \cdots \otimes f_d d\lambda=\int  f_1\otimes f_2\otimes \cdots\otimes f_d d\mu^F_{\varphi}.
 \end{align*}  
So $\mu^F_{\varphi}$ is invariant under $S_1\times\cdots \times S_d$ and we conclude that $\mu^F_{\varphi}$ coincides with $\mu^{F}_x$ because of ergodicity of the later guaranteed by Theorem \ref{THM:ergodicmeasurelimit}.

We can then take a sequence of $\varphi_n$ whose support go to identity to recover the measure $\lambda$ in the limit, as is done in \cite{DS2}.
Let $\{g_{k}: k\in \mathbb{N} \}$ be a countable set of continuous functions included and dense in the unit ball of $C(Y)$. For $k\in \mathbb{N}$, let $B_{k,n}\subseteq H$ be a ball centered at the origin such that $h\in B_{k,n}$ implies that $\|g_{k}\circ \phi'-g_{k}\circ \phi' \circ h \|_{L^1(\mu)}\leq 2^{-n}$. 

Let $\{\varphi_{k,n}\}_{n\in \mathbb{N}}$ be a sequence of weighted functions such that the support of $\varphi_{k,n}$ is included in $B_{k,n}$ (the condition on the support can always be satisfied, we refer to Proposition 6.3 in \cite{HSY}). Define the function
  \[F_{k,n}(x)= \mathbb{E}\Big ( \Bigl( \int_{H} |g_{k}\circ \phi'-g_{k}\circ \phi' \circ  h|\varphi_{k,n}(h) dm(h) \Bigr) { \big|} \I_{S_1} \Bigr) (x).\] 

and by the Markov inequality, the set $E_{n,k,i}$ of points where $F_{k,n} > \frac{1}{i}$ has a measure  smaller than 

\[ i \int_{X} F_{k,n}d\mu=i \int_{X} \int_{H} |g_{k}\circ \phi'-g_{k}\circ \phi' \circ  h|\varphi_{k,n}(h) dm(h)  \]
By Fubini and the definition of $\varphi_{k,n}$ this term is bounded by $i2^{-n}$.

The Borel-Cantelli Lemma ensures that  $$\mu(\limsup_{n} E_{n,k,i})=0.$$ So if $X''=X'\bigcap_{k,i\in \mathbb{N}}(\limsup_{n} E_{n,k,i})^c$, then $\mu(X'')=1$. 

By the Von Neumann Theorem in a subset $X'''\subset X''$ of full measure  we have the bound 
\begin{align*}
&\Bigl \vert \int   g_k\otimes f_2\otimes\cdots \otimes f_d d\lambda-\int g_k \otimes f_2\otimes\cdots\otimes  f_d d\mu^{F}_{\varphi_{k,n}} \Bigr\vert \\
= &\Bigl \vert \int \Bigl (   \int_{H} \bigl(   g_k\otimes f_{2}\otimes \cdots \otimes f_d- g_k\circ h\otimes f_2\otimes\cdots\otimes f_d\bigr)\varphi_{k,n}(h)dm(h) \Bigr) d\lambda \Bigr \vert\\
 \leq & \mathbb{E}\Big ( \Bigl( \int_{H} |g_{k}\circ \phi'-g_{k}\circ \phi' \circ  h|\varphi_{k,n}(h) dm(h) \Bigr) { \big|} \I_{S_1}\Bigr) (x)=F_{k,n}(x) 
\end{align*} 
for all $x\in X'''$ and $k,n \in \mathbb{N}$. 

We get that for $x\in X'''$ and $k\in \mathbb{N}$, for every $i\in \mathbb{N}$ there exists big enough $n$ such that $F_{k,n}\leq \frac{1}{i}$. It follows then that for $x\in X'''$ 

\[\int g_k\otimes f_{2}\otimes \cdots \otimes f_d d\lambda=\lim_{n\to\infty}\int g_k\otimes f_{2}\otimes \cdots \otimes f_d d\mu^F_{\varphi_{n,k}}=\int g_k \circ\phi'\otimes f_{2}\otimes \cdots \otimes f_d d\mu^{F}_{x}.\] 
for all $k\in \mathbb{N}$. A density argument allows to deduce $\lambda=(\phi'\times\id\times\cdots\times \id)_{\ast}\mu^F_{x}$.

Now fix $f_2,\ldots,f_d\in L^{\infty}(\mu), \epsilon >0,$ and let $f_1$ be measurable with respect to $Y$  with $\|f_1\|_{\infty}\leq 1$. By Birkhoff Theorem and telescoping, for any continuous function in a dense countable family (for example for the family $ \{g_k\}_{k\in \mathbb{N}}$), we have the bound 

\begin{equation}\label{equa4}
\begin{split}
&\quad\limsup_{N\to\infty}\left \vert\frac{1}{N}\sum_{n=0}^{N-1} f_{1}(\phi'(S_1^nx))f_{2}(S_2^nx)\cdots f_d(S_d^n x)- g_{k}(\phi'(S_1^nx))f_2(S_2^nx)\cdots  f_{d}(S_d^nx)  \right \vert
\\&\leq \mathbb{E}(\vert f_1\circ \phi'-g_k\circ \phi'\vert \mid \I_{S_1})(x)
\end{split}
\end{equation}
for $\mu$-a.e. $x\in X$. 

So 
\begin{equation}\label{equa5}
\begin{split}
&\quad\limsup_{N\to\infty}\left \vert\frac{1}{N}\sum_{n=0}^{N-1} f_{1}(\phi'(S_1^nx))f_{2}(S_2^nx)\cdots f_d(S_d^n x)-  \int f_{1}\circ \phi'\otimes f_{2}\otimes\cdots \otimes f_d d\mu^F_x   \right \vert
\\ &\leq \limsup_{N\to\infty}\left \vert\frac{1}{N}\sum_{n=0}^{N-1} f_{1}(\phi'(S_1^nx))f_{2}(S_2^nx)\cdots f_d(S_d^n x)- g_{k}(\phi'(S_1^nx))f_2(S_2^nx)\cdots  f_{d}(S_d^nx)  \right \vert \\
& +  \left \vert \int (f_1\circ\phi' \otimes f_{2}\otimes\cdots \otimes f_d- g_{k}\circ \phi' \otimes f_{2}\otimes\cdots \otimes f_d)d\mu^F_x   \right \vert
\\ & \leq 2(\mathbb{E}(\vert f_1\circ \phi'-g_k\circ \phi' \vert \mid \I_{S_1})(x)
\end{split}
\end{equation}

Let {\small \[E_{\epsilon}=\left \{x \in X: \limsup_{N\to \infty} \left \vert \frac{1}{N}\sum_{n=0}^{N-1}f_1(\phi' S_1^nx)f_2(S_2^n(x))\cdots f_d(S_d^n x)- \int f_{1}\circ \phi'\otimes f_{2}\otimes\cdots \otimes f_d d\mu^F_x \right\vert>2 \epsilon \right \}. \]}
Then $$\mu(E_{\epsilon})\leq \mu(\{x: \mathbb{E}(\vert f_1\circ \phi'-g_k\circ\phi'\vert \mid \I_{S_1})(x)\geq \epsilon \}).$$ 

Let $0<\delta<\epsilon$  and let $k\in \mathbb{N}$ with $\|f_1-g_k\|_{1}\leq \delta^2$. The Markov inequality implies that  
$$\mu(E_{\epsilon})\leq \frac{\|f_1\circ \phi'-g_k\circ\phi'\|_{1}}{\epsilon} \leq \delta.$$ Since $\delta$ is arbitrary we have that $\mu(E_{\epsilon})=0$. Hence $\bigcup_{n\in \N} E_{1/n}$ is a set of 0 measure, which means that 

\[\frac{1}{N}\sum_{n=0}^{N-1}f_1(\phi'S_1^nx)f_2(S_2^n(x))\cdots f_d(S_d^n x) \]
converges $\mu$-a.e. $x\in X$ to $\int f_1\circ \phi'\otimes \cdots f_d d\mu^F_{x}$ as $N\to\infty$. 
\end{proof}

The following proposition follows from a standard limit argument:
\begin{prop}\label{Prop:inv}
	Let $(X,\mu,S_1,\ldots,S_d)$ be a system. Let $\phi$ be a factor map of $X$ which is the inverse limit of a sequence of factor maps $\{\phi_{n}\}_{n\in\N}$. If $(\phi_{n},\id,\dots,\id)$ is good for all $n\in\N$, then so is $(\phi,\id,\dots,\id)$.
\end{prop}

The next lemma is standard 

\begin{lem}\label{Lem:ind}
	If Theorem \ref{THM:Averagesdistal} holds for $d-1$, then any tuple $\vec{\phi}=(\phi_1,\ldots,\phi_d)$ is good if $\phi_i$ is a factor of $\I_{S_i}\bigvee_{j\neq i} \I_{S_i^{-1}S_j}$ for some $i=1,\ldots,d$ 
\end{lem}
\begin{proof}
	By a density argument, it suffices to show the statement when $f_i$ is the product of $g_{j,i}$, $j=1,\ldots,d$, where $g_{i,i}$ is $\I_{S_i}$ measurable and $g_{j,i}$, $j\neq i$ is $\I_{S_i^{-1}S_j}$ measurable. In this case,
	\[\frac{1}{N}\sum_{n=0}^{N-1} f_1(S_1^nx)\cdots f_d(S_d^nx)=g_{i,i}(x)\frac{1}{N}\sum_{n=0}^{N-1} \prod_{j\neq i} g_{i,j}(S_j^nx)\]
	and the result follows  by assumption.
\end{proof}

The following lemma is a direct generalization of Lemmas 3.14 and 3.15 in \cite{Chu}. We omit the proof.

\begin{lem} \label{Lem:ExpectationMeasure}
	Let $(X,\mu,S_1,\ldots,S_d)$ be a measure preserving system and $\mu=\int \mu_x d\mu(x)$ be the ergodic decomposition of $\mu$. For a bounded function $f$ on $X$, let $\widetilde{f}$ be a version of $\mathbb{E}(f\mid \bigvee \mathcal{I}_{T_i})$. Then for $\mu$-a.e. $x\in X$, we have 
	\[\mathbb{E}_{\mu_x}(f\mid  \mathcal{I}_{T_i})=\widetilde{f}\quad \mu_x-a.e.   \] 
	
	(Here $\mathbb{E}_{\mu_x}(f\mid  \mathcal{I}_{T_i})$ is the conditional expectation of $f$ with respect to the measure $\mu_x$).
\end{lem}

We need the following proposition, which is similar to Proposition 4.8 in \cite{DS2}.
  
\begin{prop} \label{Lem:Magic extension} 
Let $(Y,\nu,S_1,\ldots,S_d)$ be an ergodic distal system with commuting transformations. Then there exists an ergodic distal extension $(Y,\mu,S_1$ $,\ldots,S_d)$ which is magic for $T_1,\ldots,T_d$. 
\end{prop}
 
\begin{proof}
We first consider the cubic extension $\pi\colon(Y^{[d]},\nu_{T_1,\ldots,T_d},\mathcal{F}_1^1,\ldots,\mathcal{F}_d^1)\to (Y,\nu,T_1,\ldots,T_d$, where the factor map $\pi$ projects $Y^{[d]}$ onto the last coordinate. In $(Y^{[d]},\mu_{T_1,\ldots,T_d})$ we consider the action of the group $\mathcal{G}_{T_1,\ldots,T_d}$ that we recall is the group generated by both upper and lower face transformations $\mathcal{F}_1^0,\ldots,\mathcal{F}_d^0$, $\mathcal{F}_1^1,\ldots,\mathcal{F}_d^1$ (see Section \ref{Sec:HostMeasures}). We claim that $(Y^{[d]},\nu_{T_1,\ldots,T_d},\mathcal{G}_{T_1.\ldots,T_d})$ is ergodic and distal. The ergodicity follows directly from page 12 in \cite{H} or can be derived as a consequence of Theorem \ref{Thm:StrictlyModel}. To show distality, we consider $(A_i)_{i\in \N}$ a separating sieve for $(Y,\nu,T_1,\ldots,T_d)$ and prove that $A_i^{2^d}=(A_i\times \cdots \times A_i)_{i\in \N}$ is a separating sieve for $(Y^{[d]},\nu_{T_1,\ldots,T_d},\mathcal{G}_{S_1,\ldots,S_d})$.

The Jensen inequality implies that
\[\nu_{T_1}(A_i\times A_i)^{1/2}=\Bigl(\int_{X} |\mathbb{E}(1_{A_i}\vert I_{T_1})|^2 d\nu \Bigr)^{1/2} \geq \int_{X} \mathbb{E}(1_{A_i}\vert I_{T_1})d\nu=\nu(A_i)>0.\] 
Similarly 
\[\nu_{T_1,T_2}(A_i\times A_i\times A_i\times A_i)^{1/4}\geq \nu_{T_1}(A_i\times A_i)^{1/2}\geq \nu(A_i)>0.\]
Thus
\[0<\nu(A_i)^{2^d}\leq \nu_{T_1,\ldots,T_d}(A_i\times\cdots\times A_i)\leq \nu_{T_1,\ldots,T_d}(A_i\times X\times \cdots \times X)=\nu(A_i)\to 0.\] 

On the other hand, if two points $\vec{x}$ and $\vec{y}$ in $Y^{[d]}$ are such that for all $i\in\mathbb{N}$, there exists $\vec{g}\in \mathcal{G}_{T_1,\ldots,T_d}$ with $ \vec{g}\vec{x}$, $\vec{g}\vec{y} \in A_i\times \cdots\times A_i$, then by the distality on each coordinate  (and that $\{A_i\}_{i\in \mathbb{N}}$ is a separating sieve), we have that  $\vec{x}=\vec{y}$. This finishes the claim.

\

Let $\nu_{T_1,\ldots,T_d}=\int \nu_{\vec{x}}~ d\nu_{T_1,\ldots,T_d}(\vec{x})$ be the ergodic decomposition of $\nu_{T_1,\ldots,T_d}$ under the face transformations $\mathcal{F}_1^1,\ldots, \mathcal{F}_d^1$. The map $\pi$ (the projection onto the last coordinate of $Y^{[d]}$)  maps $\nu_{T_1,\ldots,T_d}$ onto $\nu$ and also maps $\nu_{\vec{x}}$ to $\nu$ for $\nu_{T_1,\ldots,T_d}$-a.e. $\vec{x}\in Y^{[d]}$ by ergodicity. 

We consider a countable family $\{f_k : k\in \N\}$ on $Y^{[d]}$ which is dense in $L^p(\nu)$ for any $p\in [0,+\infty)$ and every $\mathcal{G}_{T_1,\ldots,T_d}$-invariant measure. By a density argument, to show that $\nu_{T_1,\ldots,T_d}$-a.e $x\in X$ the system $(Y^{[d]},\nu_{\vec{x}},\mathcal{F}_1^1,\ldots, \mathcal{F}_d^1$) is magic, it suffices to show that $\normm{f_k-\mathbb{E}_{\nu_{\vec{x}}}(f_k \mid \bigvee_{i=1}^d \I_{\mathcal{F}_i^1})}_{\nu_{\vec{x}},\mathcal{F}_1^1,\ldots, \mathcal{F}_d^1}=0$ for all $k\in \N$.

By Lemma \ref{Lem:ExpectationMeasure}, if $\widetilde{f}_k$ is a measurable representative of $\mathbb{E}(f_k \mid \bigvee_{i=1}^d \I_{T_i})$, then $\mathbb{E}_{\nu_{\vec{x}}}(f_k \mid \bigvee_{i=1}^d \I_{\mathcal{F}_i^1})=\widetilde{f}_k$ $\mu_{\vec{x}}$-a.e. for $\nu_{T_1,\ldots,T_d}$-a.e. $\vec{x}\in Y^{[d]}$ for all $k\in \N$.  

Since $(Y^{[d]},\nu_{T_1,\ldots,T_d},\mathcal{F}_1^1,\ldots, \mathcal{F}_d^1)$ is magic (Theorem \ref{Thm:MagicExtension}), by Theorem \ref{ine}, we have that
\[0=\normm{f_k-\widetilde{f}_k}^{2^d}_{\nu_{T_1,\ldots,T_d},\mathcal{F}_1^1,\ldots, \mathcal{F}_d^1}=\int   \Bignormm{f_k-\mathbb{E}_{\nu_x}(f_k \mid \bigvee_{i=1}^d \I_{\mathcal{F}_i^1} )}^{2^d}_{\nu_{\vec{x}},\mathcal{F}_1^1,\ldots, \mathcal{F}_d^1}d\nu_{T_1,\ldots,T_d}(\vec{x}) \]
for every $k\in \N$. Hence $\Bignormm{f_k-\mathbb{E}_{\nu_x}(f_k \mid \bigvee_{i=1}^d \I_{\mathcal{F}_i^1} )}^{2^d}_{\nu_{\vec{x}},\mathcal{F}_1^1,\ldots, \mathcal{F}_d^1}=0$  for $\nu_{T_1,\ldots,T_d}$-a.e. $\vec{x}\in Y^{[d]}$ for all $k$. This finishes the proof.
\end{proof}

\begin{proof}[Proof of Theorem \ref{THM:Averagesdistal}]
Let $(X,\mu,S_1,\ldots,S_d)$ be an ergodic distal system with commuting transformations. We prove the pointwise convergence of multiple averages inductively for the number of transformations $S_i$ considered. The case $d=1$ is just the Birkhoff Theorem. We now assume the conclusion hold for $i-1$ and prove it for $i\leq d$, i.e. we consider the system $(X,\mu,S_1,\ldots,S_i)$. We may decompose $\mu$ as $\mu=\int \mu_{x}d\mu(x)$ into $\langle S_1,\ldots,S_i \rangle$ ergodic components. For $\mu$-a.e $x\in X$, the measure $\mu_x$ is ergodic and distal for $S_1,\ldots,S_i$ (for the distality property see for instance \cite{DS2} Prop 4.9 for a proof of this fact). 

It suffices to prove that Theorem \ref{THM:Averagesdistal} holds for $\mu_{x}$ for $\mu$-a.e $x\in X$.
Fix $x\in X$. We can apply Proposition \ref{Lem:Magic extension} to find an extension of $\mu_{x}$ which is ergodic, distal and magic for $T_{1},\dots,T_{i}$. 
By Proposition \ref{Prop:iso} and \ref{Prop:inv}, it suffices to show that $(\phi,\id,\dots,\id)$ is good for $\phi\colon X\to \I_{S_1}\bigvee_{j=2}^{i} \I_{S_1^{-1}S_j}$. By induction hypothesis and Lemma \ref{Lem:ind}, we are done. 
\end{proof}

\end{document}